\documentclass[12pt]{article}

\usepackage{amssymb, amsmath, amsfonts, amsthm}
\usepackage{mathrsfs}
\usepackage{booktabs}
\usepackage{threeparttable}
\usepackage{bussproofs}
\usepackage{xcolor}
\usepackage{multirow}
\usepackage{makecell}
\usepackage{lscape}
\usepackage{stmaryrd}
\usepackage{graphicx}
\usepackage{comment}
\usepackage{mdframed}
\usepackage{enumerate}
\usepackage{float}
\usepackage{bm}
\usepackage{mathdots}
\usepackage{tikz}
\usetikzlibrary{trees, decorations.pathmorphing, shapes.geometric}
\usepackage{tikz-qtree}
\usepackage{quiver}
\usepackage{apxproof}
\usepackage{hyperref}
\usepackage{lineno}
\usepackage{macros} 

\makeatletter
\newcommand{\dotminus}{\mathbin{\text{\@dotminus}}}
\newcommand{\@dotminus}{%
  \ooalign{\hidewidth\raise1ex\hbox{.}\hidewidth\cr$\m@th-$\cr}%
}
\makeatother

\newtheorem{definition}{Definition}[section]
\newtheorem*{theorem*}{Theorem}
\newtheorem*{corollary*}{Corollary}
\newtheoremrep{theorem}[definition]{Theorem}
\newtheoremrep{lemma}[definition]{Lemma}
\newtheoremrep{proposition}[definition]{Proposition}
\newtheoremrep{corollary}[definition]{Corollary}
\newtheorem{example}[definition]{Example}
\newtheorem{remark}[definition]{Remark}

\usepackage{refcount}

\title{Predicative Ordinal Recursion on the Constructive Veblen Hierarchy}
\author{Amirhossein Akbar Tabatabai, Vitor Greati,\\ Revantha Ramanayake}
\date{\small Bernoulli Institute, University of Groningen, Netherlands}

\begin{document}

\maketitle
\begin{abstract}
Inspired by Leivant's~\cite{leivant1991foundational} work on absolute predicativism, Bellantoni and Cook~\cite{bellantoni92thesis,bellantonicook1992} in 
1992 introduced a structurally restricted form of recursion called \emph{predicative recursion}. Using this recursion scheme on the inductive structures of natural numbers and binary strings, they provide a structural and machine-independent characterization of the classes of linear-space and polynomial-time computable functions, respectively. This recursion scheme can be applied to any well-founded or inductive structure, and its underlying principle, \emph{predicativization}, extends naturally to other computational frameworks, such as higher-order functionals and nested recursion.

In this paper, we initiate a systematic project to gauge the computational power of predicative recursion
on
arbitrary well-founded structures. 
As a natural measuring stick for well-foundedness, we use \emph{constructive ordinals}.
More precisely, for any downset $\mathsf{A}$ of constructive ordinals, we define a class $\PredFuncClass{\mathsf{A}}$ of predicative ordinal recursive functions that are permitted to employ a suitable form of predicative recursion on the ordinals in $\mathsf{A}$.
We focus on the case that $\mathsf{A}$ is a downset of constructive ordinals below
$\bm{\phi}_{\bm{\omega}}(\bm{0}) = \bigcup_{k=0}^{\infty} \bm{\phi}_k(\bm{0})$, where $\{\bm{\phi}_k\}_{k=0}^{\infty}$ are the functions in the Veblen hierarchy with finite index.
We give a complete classification of $\PredFuncClass{\mathsf{A}}$---for those downsets that contain at least one infinite ordinal---in terms of the Grzegorczyk hierarchy $\{\mathcal{E}_k\}_{k=2}^{\omega}$.
In this way, we extend
Bellantoni-Cook's characterization of $\mathcal{E}_2$ (the class of linear-space computable functions) to obtain a machine-independent and structural characterization of the entire Grzegorczyk hierarchy.

\end{abstract}


\tableofcontents

\section{Introduction}
\label{sec:introduction}
Predicativism is a philosophical stance that rejects \emph{impredicative} definitions, i.e., definitions involving either a direct or indirect form of \emph{self-reference}. Classic examples include Russell’s definition of the ``set of all sets that do not include themselves'' and Berry’s definition of ``the least positive integer not definable in under sixty letters.'' However, impredicativity is not limited to paradoxes; many standard mathematical definitions are also impredicative.  
Typically, an impredicative definition arises when an object $a$ is defined by a formula $\phi(x)$ whose quantifiers range over a domain $D$ that already contains $a$, even though $D$ itself has not yet been constructed.
For example, defining $a$ as the least element of an ordered set $(A,\leq)$ relies on the formula
$$
\phi(x) \equiv \big(x \in A \wedge \forall y \in A \, (x \leq y)\big),
$$
where the universal quantifier ranges over $A$, a domain that includes $a$ itself. If one does not accept $A$ as a pre-existing totality, such a definition is impredicative.
Many familiar mathematical notions follow this pattern. For instance, the definition of ``the \emph{least} upper bound of a bounded set of real numbers'' is impredicative if one does not accept the set of real numbers as a completed structure. Likewise, the usual definition of the set of natural numbers $\mathbb{N}$ is considered impredicative if one rejects the powerset axiom. This is because that definition characterizes $\mathbb{N}$ as the \emph{least} inductive \emph{subset} of a set $I$ provided by the axiom of infinity\footnote{In $\mathsf{ZFC}$, the axiom of infinity guarantees the existence of a set $I$ containing $0$ (the empty set) and closed under the successor operation $s(x) = x \cup \{x\}$. The natural numbers are then defined as the smallest subset of $I$ containing $0$ and closed under~$s$.}, and hence one must regard the powerset of $I$ as the completed domain on which the least element is considered.

Although the set of natural numbers is technically impredicative, predicative mathematics, following Poincaré and Weyl, typically accepts it as self-evident and intuitive. Predicativity is then defined \emph{relative} to the set of natural numbers, with restrictions beginning at the level of analysis \cite{Weyl}. This choice is pragmatic: without the natural numbers, the axiom of induction and the justification for recursion lose their foundation, and ``it is obvious that such a state of affairs renders elementary mathematics impossible" \cite[Russell 1908, p.~167]{FromFregeToGodel}.

However, \emph{absolute} predicativism, i.e., the rejection of even the set of natural numbers as a definite entity, is neither sterile nor irrelevant. Quite unexpectedly, this seemingly niche philosophical stance plays a natural and foundational role in understanding the \emph{structure} of low-complexity computation. 
This point will become apparent in the following subsection, where we  outline two approaches for identifying acceptable functions in absolute predicative mathematics, and then discuss their relationship to computational complexity. This takes us to the realm of predicative recursion, a topic with many interesting results and applications, and the main focus of this work.

\subsection{Predicativism and computational complexity}

The first approach to identifying predicatively acceptable functions is based on the observation that, although the set of natural numbers is defined impredicatively, its initial segments are predicatively acceptable because they are all finite. Accordingly, restricting induction to formulas with bounded quantifiers \cite{buss1985bounded,Jan} and limiting computation to \emph{bounded recursion} is considered predicatively acceptable, where bounded recursion refers to defining a function by recursion 
provided it is bounded by a function that has already been defined. 
This creates a clear connection to complexity theory:
bounded recursion provides a machine-independent framework for capturing bounded resources, 
and the latter is a main player in
complexity theory. Indeed, bounded recursion has been used to give machine-independent characterizations of various complexity classes, including the classes of polynomial-time \cite{Cobham} and linear- and polynomial-space computable functions \cite{CloteSurvey}, as well as many others \cite{CloteSurvey}.

However, the elegant machine-independent characterizations mentioned above have a significant downside, namely that the bounds were \emph{manually} and \emph{explicitly} prescribed in the bounded recursion.
Therefore, although these characterizations improve upon machine-based definitions, they do not reveal the true \emph{structure} of the low-complexity computation.

A second, more foundational approach towards predicatively acceptable functions has been developed in various forms by Nelson \cite{Nelson}, Leivant and Marion \cite{leivant1991foundational,LeivantI,LeivantII,LeivantIII}, and Bellantoni and Cook \cite{bellantoni92thesis,bellantonicook1992}. The central idea is to adopt the usual definition of natural numbers via the \emph{subsets} of $I$, while rejecting the assumption that they form a set, thereby avoiding self-reference.\footnote{Here, $I$ is again an infinite set provided by the axiom of infinity in $\mathsf{ZFC}$.} Concretely, we define the predicate
\[
\mathbb{N}(x) := \forall X \, (\mathrm{Ind}(X) \to x \in X),
\]
which formalizes ``$x$ is a natural number,'' where $X$ ranges over subsets of $I$, and
\[
\mathrm{Ind}(X) := (0 \in X) \wedge \forall a \, (a \in X \to s(a) \in X).
\]
However, the \emph{class} $\mathbb{N}:=\{x \in I \mid \mathbb{N}(x)\}$ is no longer guaranteed to be a \emph{set}.\footnote{Formally, this requires moving from $\mathsf{ZFC}$ to second-order logic with a constant $0$ and a unary function symbol $s$. Since comprehension is restricted to positive first-order existential formulas, and $\mathbb{N}(x)$ is a $\Pi^1_1$ predicate, weak comprehension cannot justify the existence of the set of all natural numbers~\cite{leivant1991foundational}.}

Although there is no set of natural numbers, the framework still supports a form of weak recursion. Roughly speaking, this recursion allows defining total functions over $\mathbb{N}$ by recursion on total functions over $I$, provided the latter map every inductive subset $X \subseteq I$ into itself.\footnote{The actual form of recursion is more complicated, as it allows parameters from both $\mathbb{N}$ and $I$. As our primary goal here is to motivate this new recursion, there is no harm in considering this special case.}
To illustrate, consider the simple case of iteration. Let $a_0 \in I$ and let $g : I \to I$ be a total function such that, for every inductive set $X \subseteq I$, we have $a_0 \in X$ and $g(x) \in X$, if $x \in X$.  
Define a partial function $f : I \to I$ by
\[
f(0) = a_0, \qquad f(s(x)) = g(f(x)).
\]
In general, there is no guarantee that $f$ is total on $I$. However, we claim that $f$ is defined on all of $\mathbb{N}$, with values lying in any inductive $X \subseteq I$. To see this, define
\[
X_f = \{\,x \in I \mid \exists y \in X \ (f(x) = y)\,\}. \qquad \qquad (*)
\]
Clearly, $0 \in X_f$ since $a_0 \in X$. Moreover, if $a \in X_f$, then $f(a) \in X$ is defined, and by the closure of $X$ under $g$, so is $f(s(a)) \in X$. Hence $s(a) \in X_f$, and therefore $X_f \subseteq I$ is inductive. By the definition of $\mathbb{N}$, it follows that $f(x) \in X$ is defined for every $x \in \mathbb{N}$. Finally, since $X$ was arbitrary, we conclude that $f(x) \in \mathbb{N}$, i.e., $f : \mathbb{N} \to \mathbb{N}$ is a total function. 

The above argument ensures the availability of a weak form of iteration. However, by rejecting $\mathbb{N}$ as a set, the usual type of strong iteration is no longer justified. Indeed, if we simply start with $a_0 \in \mathbb{N}$ and $g : \mathbb{N} \to \mathbb{N}$, we would need to use $\exists y \in \mathbb{N}$ in $(*)$, while the domain of $y$ must be a set such as $X$, not a class such as $\mathbb{N}$.\footnote{Technically, using $\mathbb{N}$ makes the defining formula of $X_f$ to be $\Pi^1_1$, for which we have no comprehension.} Therefore, treating $\mathbb{N}$ predicatively as a class rather than a set blocks recursion in its usual form, while still permitting certain restricted forms.

To have some examples, consider the function $d(x) = 2x$, recursively defined by
\[
d(0) = 0, \qquad d(s(x)) = ss(d(x)).
\]
Since any inductive $X \subseteq I$ contains $0$ and is closed under the map $a \mapsto ss(a)$, the above argument—taking $0$ and $ss$ to play the roles of $a_0$ and $g$, respectively—ensures that $d : \mathbb{N} \to \mathbb{N}$ is total. However, if we attempt to repeat this process and define exponentiation $\mathrm{exp}(z) = 2^z$ by
\[
\mathrm{exp}(0) = s(0), \qquad \mathrm{exp}(s(z)) = d(\mathrm{exp}(z)),
\]
the argument above no longer works, since we require that $d$ is total on $I$ but we only have that it is total on $\mathbb{N}$. Establishing the totality of $d$ on $I$ requires a form of inductivity that is absent in $I$.
The inability to define total exponentiation in this manner is significant. Exponentiation, after all, serves as the prime example of a computational explosion. Consequently, this observation already suggests a connection between this weak form of recursion and low-complexity classes of computation.

To isolate this particular form of recursion from the logical argument on which it is based, and thereby develop a purely recursion-theoretic calculus for predicatively acceptable functions, Bellantoni and Cook~\cite{bellantoni92thesis,bellantonicook1992} introduced a syntactic technique to distinguish between two types of inputs involved in weak recursion: the \emph{normal} inputs from $\mathbb{N}$, on which recursion is allowed, and the \emph{safe} inputs from $I$, on which recursion is disallowed.\footnote{Leivant~\cite{LeivantI,LeivantII,LeivantIII} also introduced a similar recursion scheme using tiers to implement this separation.} Using this technique, Bellantoni and Cook introduced a purely recursion-theoretic calculus of \emph{predicative recursive functions} (see Definition \ref{dfn:ClassB}).

Connecting this calculus to computational complexity, they showed that the class of functions constructible in their calculus coincides with the class $\mathcal{E}_2$ of the Grzegorczyk hierarchy, which consists of the linear-space computable functions. They also adapted the same ideas from natural numbers to binary strings to design a calculus capturing the class of polynomial-time computable functions. These characterizations then paved the way for a plethora of similar results.

For instance, Bellantoni~\cite{bellantoni92thesis} introduced a predicative form of minimization to capture all levels of the polynomial-time hierarchy; Arai and Eguchi~\cite{AraiExp} used nested predicative recursion to characterize exponential time; Leivant~\cite{LeivantIII} employed predicative recursion on finite-type functionals to describe the class $\mathcal{E}_3$ of elementary functions; Wirz~\cite{wirz} extended predicative recursion to include more types of inputs to classify the entire Grzegorczyk hierarchy; and Arai~\cite{Arai} formalized feasible functions over arbitrary sets using a predicative variant of $\in$-recursion. This list of applications is far from exhaustive; see, e.g., \cite{Other1,Other2,Other3,Other4,LeivantPoly,CloteExp,CloteSurvey}.

The primary computational feature of these characterizations lies in their \emph{structural} nature, which allows the identification of complexity classes without relying on machines or imposing explicit bounds on recursion. This structural perspective offers several advantages. Mathematically, it enables the application of categorical and type-theoretical techniques to study complexity classes~\cite{HofmannHabil,HofmannBCK,HofmannLinear,Cockett}. Philosophically, it provides a foundational justification for the central role of certain complexity classes, such as polynomial-time computable functions. Moreover, a coherent foundational approach to complexity theory can inspire new concepts by identifying predicative counterparts to classical notions in recursion theory and translating them into the computational complexity framework.

\subsection{Our contribution}

The widespread use of predicative recursion, as discussed in the previous subsection, motivates studying this operation in its most general form, namely as predicative recursion over an arbitrary well-founded structure~\cite{curzi2022cyclic}. While several compelling individual results connect predicativism with computation, a systematic exploration of this relationship remains underdeveloped. The aim of this paper is to address this gap through a comprehensive investigation of the computational potential of predicative recursion.

To take such a systematic approach, we focus on ordinals among all well-founded structures, since they provide a natural yardstick for measuring well-foundedness. Technically, we work with \emph{constructive ordinals}, i.e., rooted well-founded trees with either single branchings (successors) or $\mathbb{N}$-indexed branchings (limits) at each non-leaf node. Constructive ordinals serve as an abstract representation system for set-theoretic ordinals with explicit construction, making them better suited to our computational goals than their set-theoretic counterparts.  

For the purposes of this paper, we are specifically interested in constructive ordinals corresponding to set-theoretic ordinals below $\bm{\phi}_{\bm{\omega}}(\bm{0})=\bigcup_{i=0}^{\infty}\bm{\phi}_i(\bm{0})$, where $\{\bm{\phi}_i\}_{i=0}^{\infty}$ are the finite levels of the Veblen hierarchy. To this end, we introduce constructive versions of these functions, denoted by $\{\phi_i\}_{i=0}^{\infty}$, and define $\Phi_{\omega}$ (resp.\ $\Phi_k$ for $k \geq 0$) as the set of constructive ordinals generated from $0$ by addition and the $\phi_i$ (resp.\ the $\phi_i$ with $i<k$). The set $\Phi_{\omega}$ (resp.\ $\Phi_k$) then corresponds to the set-theoretic ordinals below $\bm{\phi}_{\bm{\omega}}(\bm{0})$ (resp.\ $\bm{\phi}_k(\bm{0})$). Similarly, we introduce $\Psi_{\omega}$ as the set of constructive ordinals corresponding to set-theoretic ordinals below $\bm{\omega}^{\bm{\omega}}$. We then investigate the properties of these classes in an extensive and detailed manner, which may be of independent interest, as the details of such a theory are, to the best of our knowledge, absent from the literature. Consequently, parts of the present paper may serve as a complete and self-contained introduction to constructive ordinals and constructive Veblen functions.

Then, we turn to strengthening Bellantoni--Cook's class of functions by generalizing its predicative recursion on numbers to constructive ordinals. More precisely, let $\mathsf{A}$ be a given downset of constructive ordinals, i.e., a set closed under the subtree order. We define a class $\mathcal{C}_{\mathsf{A}}$ of functions with numeral outputs (interpreted as elements of $I$), and with three distinct types of inputs: new ordinal inputs from $\mathsf{A}$, used in a suitable form of predicative recursion on the subtree order; together with the original numeral inputs in normal and safe positions, interpreted as variables in $\mathbb{N}$ and $I$, respectively. To assess the computational power of $\mathcal{C}_{\mathsf{A}}$, we define $\PredFuncClass{\OrdClass}$ as the subset of $\mathcal{C}_{\mathsf{A}}$ consisting of functions with only normal inputs.

Our main result is a complete classification of $\PredFuncClass{\mathsf{A}}$ for any downset $\mathsf{A} \subseteq \Phi_{\omega}$ that contains at least one infinite ordinal. To this end, we introduce a number, called the \emph{level} of $\mathsf{A}$, to measure how large its ordinals are. First, if there exists $k \in \mathbb{N}$ such that $\mathsf{A} \subseteq \Phi_k$, we call $\mathsf{A}$ \emph{bounded}. For any bounded $\mathsf{A}$, if $\mathsf{A} \subseteq \Psi_{\omega}$, we define $l(\mathsf{A}) = 0$; otherwise, we define $l(\mathsf{A})$ as the least $k \geq 1$ such that $\mathsf{A} \subseteq \Phi_k$. We then prove:

\begin{theorem*}[Main theorem]\label{the:main-theorem-characterization}%
    Let $\mathsf{A} \subseteq \Phi_{\omega}$ be a downset of ordinals that contains at least one infinite ordinal. Then:
    \begin{itemize}
        \item[$(i)$] 
        If $\mathsf{A}$ is bounded, then $\PredFuncClass{\mathsf{A}} = \GrzClass{l(\mathsf{A})+2}$,
        \item[$(ii)$] 
        If $\mathsf{A}$ is unbounded, then $\PredFuncClass{\mathsf{A}} = \PR$,
    \end{itemize}
    where $\{\mathcal{E}_k\}_{k=2}^{\infty}$ is the Grzegorczyk hierarchy and $\PR$ is the set of all primitive recursive functions. 
\end{theorem*}

Here are some key observations about this result. First, it provides a complete classification of $\PredFuncClass{\OrdClass}$ for any $\mathsf{A} \subseteq \Phi_{\omega}$ that contains an infinite ordinal.  
This result represents a significant milestone in the program of investigating the computational potential of predicative recursion. As a next natural step, we aim to establish a precise connection between predicative recursion on well-founded structures and their corresponding ordinals, thereby enabling a complete assessment of the computational power of predicative recursion on any well-founded structure whose associated ordinal is below $\bm{\phi}_{\bm{\omega}}(\bm{0})$.

Second, it is important to highlight that the two sides of each equality in this theorem correspond to the two distinct approaches to predicatively acceptable computation that we have outlined: the left-hand side arises from predicative recursion, while the right-hand side arises from bounded recursion. In this way, the theorem formalizes the intuition that these two approaches are, in essence, two sides of the same coin.
Third, we can derive the following special cases:

\begin{corollary*}[Main corollary]\label{cor:main-cor-characterization}
$\GrzClass{2}=\PredFuncClass{\Psi_{\omega}}$,
    $\GrzClass{k}=\PredFuncClass{\Phi_{k-2}}$, for any $k \geq 3$, and 
    $
    \PR=\PredFuncClass{\Phi_{\omega}}
    $.
\end{corollary*}

This corollary provides a machine- and resource-independent characterization of the entire Grzegorczyk hierarchy, extending Bellantoni and Cook's characterization of $\mathcal{E}_2$. Roughly speaking, it first shows that even if predicative ordinal recursion extends beyond $\bm{\omega}$ and is used up to any ordinal below $\bm{\omega}^{\bm{\omega}}$, the resulting functions remain in $\mathcal{E}_2$. It then states that the predicative ordinal recursive functions using ordinals below $\bm{\phi}_{k}(\bm{0})$ (resp.\ $\bm{\phi}_{\bm{\omega}}(\bm{0})$) coincide with the functions in $\mathcal{E}_k$ (resp.\ primitive recursive functions) for $k \geq 3$.

Fourth, a notable special case of the corollary is the equality $\GrzClass{3} = \PredFuncClass{\Phi_1}$. Recall that $\Phi_1$ is the set of constructive ordinals corresponding to the set-theoretic ordinals below $\bm{\phi}_1(\bm{0}) = \bm{\epsilon_0}$. By Leivant's~\cite{LeivantIII} result, which establishes the equality between the numeral predicative recursive finite-type functionals $\mathbf{RRec^{\omega}}$ and elementary functions, we obtain $\PredFuncClass{\Phi_1} = \mathbf{RRec^{\omega}}$. This equality is a predicative analogue of the well-known classical equivalence between $\bm{\epsilon_0}$-recursive functions and numeral primitive recursive finite-type functionals. 

Fifth, Leivant~\cite{LeivantPoly} introduced the type system $\mathbf{SF_2}$ as a predicative version of Girard's~\cite{girardF2} and Reynolds’s~\cite{reynoldsF2} second-order lambda calculus $\mathbf{F_2}$, stratified by finite levels to avoid the self-referential behavior of types. He shows that the numerical functions represented in $\mathbf{SF_2}$ correspond to the class $\mathcal{E}_4$. Using a special case of the main corollary, we have $\GrzClass{4} = \PredFuncClass{\Phi_2}$, where $\Phi_2$ is the constructive counterpart of the set of set-theoretic ordinals below $\bm{\phi}_2(\bm{0})$. This observation suggests that if an ordinal analysis of $\mathbf{SF_2}$ were carried out, the ordinal assigned to the system would be $\bm{\phi}_2(\bm{0})$.

\subsection*{Outline of the paper}

The paper is organized as follows. In Section \ref{sec:preliminaries}, we recall some basic preliminaries, and in Section \ref{sec:pred-rec-func}, we review Bellantoni-Cook's class of predicative recursive functions. Then, in Section \ref{sec:const-ordinals}, we introduce constructive ordinals, their arithmetic, and some notable functions on them. In Section~\ref{sec:pred-ord-rec-functions}, we define the class~$\ClassName{\OrdClass}$ of predicative ordinal recursive functions and its subclass $\PredFuncClass{\OrdClass}$ consisting of functions that operate solely over $\mathbb{N}$. We also prove some basic properties of $\PredFuncClass{\mathsf{A}}$.
In Section~\ref{sec:const-veblen}, we define the Veblen functions on constructive ordinals and introduce the classes $\Phi_{\omega}$ and $\Psi_{\omega}$ of constructive ordinals, together with several of their subclasses. We also prove several important properties of these classes. In Section \ref{sec:main-theorem}, we present our main theorem, which uses the Grzegorczyk hierarchy to classify $\PredFuncClass{\mathsf{A}}$ for any $\mathsf{A} \subseteq \Phi_{\omega}$ containing an infinite ordinal. We also discuss the immediate consequences of this theorem and present our strategy for proving it, which is followed throughout the remainder of the paper. Finally, in Section \ref{sec:ReductionToG}, we simulate functions in the Grzegorczyk hierarchy using predicative ordinal recursive functions, and in Section \ref{sec:UpperBoundLength}, we implement the converse simulation to complete the proofs of the main theorem.


\section{Preliminaries}
\label{sec:preliminaries}
In this section, we introduce some basic notions and notations that will be used throughout the paper. Let $X$ and $Y$ be two sets, and let $\bar{x} = x_1, \ldots, x_k$ denote a finite sequence of variables. We write $\bar{x} \in X$ to mean $\bar{x} \in X^k$, i.e., $x_i \in X$ for every $1 \leq i \leq k$.  
For any function $f : X \to Y$, we write $f(\bar{x})$ to denote the tuple $(f(x_1), \ldots, f(x_k)) \in Y^k$. Similarly, for any relation $R \subseteq X \times Y$, the notation $\bar{x} R y$ means $x_i R y$ for all $1 \leq i \leq k$. The same convention applies for $x R \bar{y}$, where $x \in X$ and $\bar{y} \in Y^k$.  
If $\bar{f} = f_1, \ldots, f_m$ is a finite sequence of functions from $X$ to $Y$, then $\bar{f}(x)$ denotes the tuple $(f_1(x), \ldots, f_m(x)) \in Y^m$, for any $x \in X$.  

The set of natural numbers is denoted by $\NatSet$, and $\NatSet \setminus \{0\}$ is denoted by $\NatSet^{\geq 1}$. Similarly, we denote the set of real numbers by $\mathbb{R}$, and the set $\{x \in \mathbb{R} \mid x \geq 2\}$ by $\mathbb{R}^{\geq 2}$.
Functions with numeral outputs are called \emph{numeral functions}. Numeral functions whose inputs are also numbers are referred to as \emph{functions over numbers}.  
A function $f : \NatSet^k \to \NatSet$ is said to be \emph{monotone} if $f(\bar{x}) \leq f(\bar{y})$ whenever $x_i \leq y_i$ for all $1 \leq i \leq k$.  
It is said to be \emph{strictly monotone} if $f(\bar{x}) < f(\bar{y})$ whenever $x_i \leq y_i$ for all $1 \leq i \leq k$ and $x_j < y_j$ for some $1 \leq j \leq k$.
The function $f$ is said to be \emph{expansive} if $f(\bar{x}) \geq \max_i x_i$ for all $\bar{x} \in \NatSet^k$.  
We write $\log(n)$ to denote the base-2 logarithm of $n$. Finally, unless otherwise specified, whenever we refer to linear functions or polynomials, we assume their coefficients are natural numbers, and thus they are always monotone.

We now define the \emph{Grzegorczyk hierarchy}~\cite{grzegorczyk1953some} as a hierarchy of classes of primitive recursive functions. There are many equivalent formulations of this hierarchy in the literature; here, we follow the one presented in~\cite[Ch.~2]{rose1984}. The first step is to introduce a family of \emph{fast-growing functions}.

\begin{definition}
\label{def:fast-growing-hie}
~\cite[Ch.~2]{rose1984}
The \emph{fast-growing hierarchy}
	$\{h_k : \NatSet \to \NatSet \}_{k=1}^{\infty}$
	is defined by $h_1(n) := n^2+2$ and 
	$h_{k+1}(n) := h_{k}^{(n)}(2)$.
\end{definition}
It is straightforward to verify that each $h_k$ is strictly monotone and expansive~\cite[Ch.~2]{rose1984}. For the second step, recall that a numeral function $f$ is said to be defined by \emph{bounded primitive recursion} from functions $g$ and $h$, with bounding function $u$, if
\[
f(0,\overline x) = g(\overline x),
\]
\[
f(y+1,\overline x) = h(y,\overline x,f(y,\overline x)),
\]
and moreover $f(y,\overline x) \leq u(y,\overline x)$ for all $y, \overline x \in \NatSet$.

\begin{definition}\label{def:Grz-Hie}~\cite[Ch.~2]{rose1984}
For any $k \geq 2$, let $\mathcal{I}_k$ be the set consisting of the constant zero function, the successor function, all projection functions, and $h_1, \ldots, h_{k-1}$.  
Define $\GrzClass{k}$ as the closure of $\mathcal{I}_k$ under composition and bounded primitive recursion.  
The hierarchy $\{ \GrzClass{k} \}_{k=2}^{\infty}$ is called the \emph{Grzegorczyk hierarchy}.
\end{definition}
It is well known that the class of primitive recursive functions
$\PR$ coincides with $\bigcup_{k \geq 2} \GrzClass{k}$~\cite[VIII.9.16]{odifreddi1999crtv2}.

Given a function $f : \NatSet^l \to \NatSet$ and a position $1 \leq i \leq l$, define by primitive recursion the
\emph{iteration of $f$ at argument $i$}, denoted by $\ItFunc{f}{i}$, as
\begin{align*}
\ItFunc{f}{i}(0,\overline x) &:= x_i,\\
\ItFunc{f}{i}(n+1,\overline x) &:= f(x_1,\ldots,x_{i-1},\ItFunc{f}{i}(n,\overline x),x_{i+1},\ldots,x_l).
\end{align*}
When the argument position $i$ is clear from the context, we may write $f^{y}(\overline x)$ or $f^{(y)}(\overline x)$ for $\ItFunc{f}{i}(y,\overline x)$.

\begin{lemma}\label{lem:grz-properties}
	The following hold:
	\begin{itemize}
        \item if $f(\bar x) \in \GrzClass{2}$,
		then there is 
        a polynomial 
        $p(\bar x)$ with natural coefficients
        such that
        $f(\bar x) \leq p(\bar x)$
        for all $\bar x \in \NatSet$.
         \item if $f(\bar x) \in \GrzClass{3}$,
		then there is 
        $M \in \mathbb{N}$
        such that
        $f(\bar x) \leq 2^{2^{\iddots^{\max_j x_j}}}$, for any $\bar x \in \NatSet$, where the number of twos in the tower is $M$.
		\item 
       if $f(x_1, \ldots, x_l) \in \GrzClass{k}$,
		then there is $M \in \NatSet$ such that 
        $f(x_1, \ldots, x_l) \leq h_{k-1}^M(\max_{i=1}^l x_i)$, for any $\bar{x} \in \mathbb{N}$
        and $k \geq 3$.
		\item if $f(x_1, \ldots, x_l) \in \GrzClass{k}$,
		then $\ItFunc{f}{i} \in \GrzClass{k+1}$
		for any $1 \leq i \leq l$
        and $k \geq 2$.
	\end{itemize}
\end{lemma}
\begin{proof}
Respectively by
\cite[Lem.~2.7]{rose1984},
a modification of
\cite[VIII.7.8 and VIII.7.21(b)]{odifreddi1999crtv2},
\cite[Lem.~2.9]{rose1984},
and \cite[Lem.~2.10]{rose1984}.
\qedhere
\end{proof}

\begin{remark}\label{RemarkOnEk}
By Lemma \ref{lem:grz-properties}, any function in $\GrzClass{k}$ with $k \geq 2$ is bounded by a function in $\GrzClass{k}$ that is both expansive and monotone.
\end{remark}

The Grzegorczyk hierarchy is originally defined for functions over natural numbers. To extend it to functions over an arbitrary finite alphabet, we first introduce a complexity-theoretic, machine-dependent characterization of this hierarchy. For this purpose, we begin by fixing some notations and definitions.

Let $\Gamma$ be a fixed finite alphabet, and let $\Gamma^*$ denote the set of all finite strings over $\Gamma$. For any $w \in \Gamma^*$, we write $\RepLen{w}$ for the length of $w$. By a \emph{computation}, we mean one carried out by a multitape Turing machine over $\Gamma$.  
Let $f : (\Gamma^\ast)^m \to \Gamma^\ast$ with $m \ge 0$. If there exists a Turing machine that computes $f(\bar w)$ using at most $S(\RepLen{\bar w})$ cells on its work tapes (resp.\ at most $T(\RepLen{\bar w})$ computation steps) for every input $\bar w \in (\Gamma^*)^m$, then $f$ is said to be \emph{computable in space $S$} (resp.\ \emph{time $T$}).
For functions over numbers, we perform computations over the alphabet $\{0,1\}$ using the binary expansion of numbers. We let $\IntermCode{n}$ denote the binary representation of $n$, and $\RepLen{n}$ the length of $\IntermCode{n}$. Note that for all $n \in \mathbb{N}$,
\[
\log(n+1) \leq \RepLen{n} \leq \log(n+1) + 1 \leq n+1
\quad\text{and}\quad
n < 2^{\RepLen{n}}.
\]
Returning to the computational characterization of the Grzegorczyk hierarchy,  
Ritchie~\cite{ritchie1963} showed that $\GrzClass{2}$ coincides with the class of all functions over numbers computable in linear space.  
For higher levels of the hierarchy, he established the following computational characterization:

\begin{theorem}[{\cite{ritchie1963},\cite[VIII.8.14]{odifreddi1999crtv2}}]
\label{the:elem-characterizatons}
	 Let $k \geq 3$. Then,
     the following are equivalent:
     \begin{itemize}
         \item $f \in \GrzClass{k}$. 
         \item $f$ is computable in time $T$, for some $T \in \GrzClass{k}$.
          \item $f$ is computable in space $S$, for some $S \in \GrzClass{k}$.
    \end{itemize}
\end{theorem}

Now, we can use the above-mentioned computational characterization to define the Grzegorczyk hierarchy for functions defined on any finite alphabet $\Gamma$. For any $k \geq 2$, define $\GrzClassSig{\Gamma, k}$ as the class of functions  
$f : (\Gamma^\ast)^m \to \Gamma^\ast$ ($m \geq 0$) computable in space $S$, where $S$ is a linear function if $k=2$, and $S \in \GrzClass{k}$ if $k > 2$.  Notice that, for $k \geq 3$, as space can be simulated by an exponential growth in time, it is equivalent to define $\GrzClassSig{\Gamma, k}$ as the class of functions computable in time in $\GrzClass{k}$.

Later in the paper, to encode ordinals, we prefer to use the specific alphabet $\Sigma \Def \{0,1,(,),;,\bot\}$ rather than the usual binary alphabet $\{0,1\}$. As usual, we employ $\bot$ as a symbol for ``undefined.''  
With this symbol, partial functions can be made total in the standard way. Since $\Sigma$ is the only non-standard alphabet we use, we simply drop $\Sigma$ from $\GrzClassSig{\Sigma, k}$ and denote it by $\GrzClassSig{k}$.  
As $\{0,1\} \subseteq \Sigma$, any function over numbers can be canonically lifted to a function over $\Sigma^*$:

\begin{definition}\label{def:coded-func-on-naturals}
Given a function $f:\NatSet^m \to \NatSet$ $(m \geq 0)$, define its \emph{lift} $\CodedFunc{f} : (\Sigma^\ast)^m \to \Sigma^\ast$ by:
\[
\CodedFunc{f}(\bar{u}) :=
\begin{cases}
\IntermCode{f(\bar{n})} & \exists \bar{n} \in \mathbb{N}\, (\bar{v} = \IntermCode{\bar{n}}), \\
\bot & \text{otherwise.}
\end{cases}
\]
\end{definition}

\begin{remark}\label{rem:equiv-E-sig-E}
We observe that $f \in \GrzClass{k}$ if and only if $\CodedFunc{f} \in \GrzClassSig{k}$, for any $k \geq 2$ and any function $f$ over numbers.  
For $k \geq 3$, assume $f \in \GrzClass{k}$. By Theorem~\ref{the:elem-characterizatons}, $f$ is computable in space $\GrzClass{k}$. Since checking whether a string in $\Sigma^*$ is the binary expansion of a number can be done in linear space, it follows that $\CodedFunc{f}$ is computable in space $\GrzClass{k}$. Hence, by definition, $\CodedFunc{f} \in \GrzClassSig{k}$.  
Conversely, if $\CodedFunc{f} \in \GrzClassSig{k}$, then by definition it is computable in space $\GrzClass{k}$. As the alphabet $\Sigma$ can be encoded into the binary alphabet with only a linear-space overhead, the same computation of $\CodedFunc{f}$ can be carried out on binary strings. Therefore, $f$ is computable in space $\GrzClass{k}$, and by Theorem~\ref{the:elem-characterizatons} we conclude that $f \in \GrzClass{k}$.  
For $k=2$, the same argument applies, replacing ``space $\GrzClass{k}$'' with ``linear space'' throughout.
\end{remark}

The following gives the closure properties of the classes $\GrzClassSig{k}$ under two computational schemes that will be used later in the paper.

\begin{theorem}\label{the:Esigma-closure}
Let $k \geq 2$. Then, the following holds:
\begin{itemize} 
    \item[$(i)$] \emph{(closure under length-bounded primitive recursion)}
   Let $g,h \in \GrzClassSig{k}$ and let $u$ be a function over numbers that is linear if $k=2$ and belongs to $\GrzClass{k}$ if $k \geq 3$.  
Define the function $f$ by
\[
f(z, \bar w) =
\begin{cases}
    g(\bar w) & \text{if } z = \IntermCode{0}, \\[4pt]
    h(\IntermCode{y}, \bar w, f(\IntermCode{y}, \bar w))
    & \text{if } z = \IntermCode{y+1}, \\[4pt]
    \bot & \text{otherwise}.
\end{cases}
\]

If for all $z, \bar w \in \Sigma^\ast$ we have
$
\RepLen{f(z,\bar w)} \;\leq\; u(\RepLen{z}, \RepLen{\bar w}),
$
then $f \in \GrzClassSig{k}$.  
In this case, we say that $f$ is defined by \emph{length-bounded primitive recursion} from $g$ and $h$ with bounding function $u$.

\item[$(ii)$] \emph{(closure under bounded $\Sigma$-minimization)}
Let $g \in \GrzClassSig{k}$ be a function, $c \in \Sigma^\ast$ be a constant, and $u$ be a function over numbers of the form $u(\bar x) = 2^{v(\bar x)}$, where $v$ is linear if $k=2$, and $v \in \GrzClass{k}$ if $k \geq 3$.  
Define the function $f$ by setting $f(\bar w) = \IntermCode{i}$, where $i \in \mathbb{N}$ is the least number bounded by $u(\RepLen{\bar w})$ such that $g(\IntermCode{i}, \bar w) = c$, if such a number exists, and $i = u(\RepLen{\bar w})+1$ otherwise. Then $f \in \GrzClassSig{k}$.  
In this case, we say that $f$ is defined by \emph{bounded $\Sigma$-minimization} from $g$ with bounding function $u$.  
We also write
$f(\bar w) \;=\; \Bmin_{i \leq u(\RepLen{\bar w})} \; [\, g(\IntermCode{i},\bar w) = c \,]$.
\end{itemize}
\end{theorem}
\begin{proof}
We present algorithms for computing the functions constructed in $(i)$ and $(ii)$ and show that these algorithms run in space $\GrzClass{k}$ when $k \geq 3$, and in linear space when $k=2$.  
This establishes that these functions indeed belong to $\GrzClassSig{k}$, as required.

For $(i)$, consider first the case $k \geq 3$. 
Without loss of generality, assume that all the space functions mentioned below are monotone and that $u(\bar x) \leq b(\bar x)$ for some monotone $b \in \GrzClass{k}$. 
Since $g, h \in \GrzClassSig{k}$, by definition, they are computable in space $S_g, S_h \in \GrzClass{k}$. 
The algorithm for $f$ first checks, in linear space, whether $z = \IntermCode{n}$ for some $n \in \NatSet$. If not, it outputs $\bot$ and halts. Otherwise, it simulates the recursion, starting by computing $f(\IntermCode{0},\bar w) = g(\bar w)$ in space $S_g(\RepLen{\bar w})$, and then, for each $1 \leq i < n$, it computes $f(\IntermCode{i+1},\bar w)$ from $f(\IntermCode{i},\bar w)$ in space  
\[
S_h(\RepLen{i+1}, \RepLen{\bar w}, \RepLen{f(\IntermCode{i},\bar w)}) 
\leq S_h(\RepLen{z}, \RepLen{\bar w}, b(\RepLen{z},\RepLen{\bar w})).
\]
The algorithm can reuse space in each recursive step, and it only needs to remember the stage $i$. Therefore, it requires the space  
\[
\RepLen{z}+S_g(\RepLen{\bar w}) + S_h(\RepLen{z},\RepLen{\bar w}, b(\RepLen{z},\RepLen{\bar w})),
\]  
which belongs to $\GrzClass{k}$. Hence, $f \in \GrzClassSig{k}$, by definition.  
For $k=2$, all these space bounds and $b$ are linear, and since linear functions are closed under composition, the same algorithm applies and runs in linear space, as desired.

For $(ii)$, we again first consider the case $k \geq 3$,
and assume that all the space functions mentioned below are monotone and that $u(\bar x) \leq b(\bar x)$ for some monotone $b \in \GrzClass{k}$. 
The algorithm for $f$ simply searches for
the least $0 \leq i \leq u(\RepLen{\bar w})$
satisfying the equality
$g(\IntermCode{i}, \bar w) = c$.
That is, it
computes
$u(\RepLen{\bar w})$
in space $\GrzClass{k}$, writes it into memory,
and, starting from $i=0$
up to $u(\RepLen{\bar w})$,
computes $g(\IntermCode{i},\bar w)$ in space 
$S_g(\RepLen{i},\RepLen{\bar w}) \leq S_g(b(\RepLen{\bar w})+1,\RepLen{\bar w})$, where
$S_g \in \GrzClass{k}$,
stopping at the first $i$ (if it exists) such that $g(\IntermCode{i},\bar w) = c$. 
If this test fails for all such $i$'s, the algorithm outputs 
$\IntermCode{u(\RepLen{\bar w})+1}$. 
Note that the algorithm also needs to keep the current $i$ at each step, whose length is bounded by $\RepLen{u(\RepLen{\bar w})} \leq 1+b(\RepLen{\bar w})$, which is in $\GrzClass{k}$.
Therefore, the whole algorithm
requires space in $\GrzClass{k}$, as desired.
If $k=2$,
we have $u(\RepLen{\bar w}) = 2^{v(\RepLen{\bar w})}$, for $v$ a linear function.
It is clear that $u(\RepLen{\bar w})$
can be computed in space $O(v(\RepLen{\bar w}))$, i.e., in linear space,
and thus the space required to store the values of $i$ is also linear in $\RepLen{\bar w}$.
Since $g$ is assumed to be computable in linear space, the entire algorithm,
by an argument similar to the above, runs in linear space.
\end{proof}

\section{Predicative Recursive Functions}
\label{sec:pred-rec-func}
As explained in the introduction, Bellantoni and Cook~\cite{bellantoni92thesis,bellantonicook1992} introduced a structurally restricted form of recursion to capture low complexity classes of numeral functions.  The key idea is to distinguish between two types of inputs to a function: \emph{normal} and \emph{safe}. To reflect this distinction, functions are written in the form $f(\bar{x}; \bar{a})$, where $\bar{x}$ denotes variables in the normal position and $\bar{a}$ denotes variables in the safe position. A useful intuition is that normal inputs are used for recursion, while safe inputs are for composition.
The formal definition of the class of these recursive functions appears below.

\begin{definition}\label{dfn:ClassB}
	Let $B$ be the smallest class of functions over numbers 
	containing (1)-(5) and closed under (6) and (7):

	\begin{enumerate}
		\item (constant) 0
		\item (projections) $\pi_j^{n,m}(x_1,\ldots,x_n;x_{n+1},\ldots,x_{n+m}) = x_j$ for $1 \leq j \leq n+m$
		\item (successor) $s(;a) = a + 1$
		\item (predecessor) $p(;0) = 0$ and $p(;a+1) = a$
		\item (conditional) 
		$
		C(;a,b,c)= b$ if $a=0$ and $C(;a, b, c)=c$, otherwise
		\item (predicative recursion) Given $g, h \in  B$, define $f$ by
			 \begin{align*}
             f(0, \bar x; \bar a) &= g(\bar x; \bar a)\\
            f(y+1,\bar x; \bar a) &= h(y,\bar x;\bar a,f(y,\bar x;\bar a))
        \end{align*}
		\item (safe composition) Given $h, \bar r, \bar t \in B$, define $f$ by 
        $$f(\bar x; \bar a) =  h(\bar{r}(\bar x;);\bar{t}(\bar x;\bar a))$$
		\end{enumerate}
Define $\mathrm{PredR}$ as the following class of numeral functions: 
\[
\{ f : \NatSet^k \to \NatSet \mid \text{there is } f' \in B  \text{ such that } f(\bar n) = f'(\bar n;), \text{ for any } \bar n \in \NatSet^k\}.
\]
We call any function in $\mathrm{PredR}$ a \emph{predicative recursive function}.
\end{definition}

\begin{remark}
Here are some remarks on Definition \ref{dfn:ClassB}. First, all basic functions are defined on safe variables. Second, the projection functions operate on all inputs, both normal and safe. Third, in safe composition, functions $\bar{r}(\bar{x};)$ do not have access to the safe inputs $\bar a$. Fourth,
in the definition of $f(y, \bar{x};\bar{a})$ by predicative recursion, the recursion is on the input $y$ which is in the normal position and in the recursive call, the value $f(y, \bar{x};\bar{a})$ is only available in the safe position of $h$. These constraints on safe composition and predicative recursion are the structural restrictions we need to put recursion on a predicatively acceptable ground on the one hand and harness the power of primitive recursion on the other.
\end{remark}

\begin{example}\label{ex:add-mult-bc}
Addition and multiplication are in $\mathrm{PredR}$. For addition, first, consider the function $\mathrm{sum}_0(x;a)$ with the following definition:
\begin{equation*}
			\mathrm{sum}_0(0; a) = a
            \qquad
			\mathrm{sum}_0(x+1; a) = s(; \mathrm{sum}_0(x;a)).
		\end{equation*}
It is clear that $\mathrm{sum}_0(x;a)=x+a$. Then, by using projections and safe composition, define: 
\[
\mathrm{sum}(x, y;)=\mathrm{sum}_0(\pi^{2,0}_1(x, y; ); \pi_2^{2,0}(x,y;))
\]
and observe that $\mathrm{sum}(x, y;)=\mathrm{sum}_0(x; y)=x+y$.
For multiplication, consider the function $\mathrm{mult}(x, y;)$ with the following definition:
\begin{equation*}
			\mathrm{mult}(0, y; ) = 0
            \qquad
			\mathrm{mult}(x+1, y; ) = \mathrm{sum}_0(y; \mathrm{mult}(x, y;)).
		\end{equation*}
Observe that in the recursive definition, in both cases, the value of the functions $\mathrm{sum}_0(x;a)$ and $\mathrm{mult}(x, y;)$ on a lower input appears in the safe position.

As for non-examples, the canonical recursive definition of the function $\mathrm{exp}(x) = 2^x$, namely, $\mathrm{exp}(0;)=1$ and $\mathrm{exp}(x+1; ) = \mathrm{sum}_0(\mathrm{exp}(x;);\mathrm{exp}(x;))$, 
defined via predicative recursion, is not acceptable because that scheme does not permit the value $\mathrm{exp}(x;)$ to be placed in normal position.
{Of course, this alone does not show that $\exp \notin \mathrm{PredR}$.} 
However, by Theorem \ref{thm:cb}, we will see that the structural restrictions indeed block the construction of the exponentiation function.
\end{example}

The following theorem interestingly shows that the \emph{predicatively acceptable} notion of computation and a certain type of \emph{low-complexity} computation coincide. One can also interpret this theorem as expressing the equivalence between two distinct approaches to predicative computation, as outlined in the introduction.
\begin{theorem}[Bellantoni-Cook,{\cite[Ch.~5]{bellantoni92thesis}}]\label{thm:cb}
$\mathrm{PredR}= \GrzClass{2}$.
\end{theorem}

In this paper, we aim to extend this result by generalizing the predicative recursive functions in Definition~\ref{dfn:ClassB} from numbers (equivalently, finite ordinals) to infinite ordinals, and by identifying the corresponding complexity classes within the Grzegorczyk hierarchy.

\section{Constructive Ordinals}
\label{sec:const-ordinals}
\newcommand{\ConsOrd}{\Omega}
\newcommand{\StConsOrd}{\Omega^{\mathsf{S}}}
In this section, we introduce constructive ordinals, their relationship with traditional set-theoretic ordinals, and various ordinal and numeral functions defined over both constructive and set-theoretic ordinals.

A countable \emph{constructive ordinal} (or simply an \emph{ordinal}) \cite{fairtlough1998ptchapter} is a rooted, well-founded tree in which each non-leaf node has either a single branch or an $\mathbb{N}$-indexed branching. We denote the set of all ordinals by $\Omega$. The single-node tree is denoted by $0$ and called the \emph{zero} ordinal; the result of adding one node below the tree $\alpha$ is denoted by $\alpha + 1$ and is called a \emph{successor} ordinal; and the result of adding one node below the sequence $\{\alpha_i\}_{i \in \mathbb{N}}$ of trees is denoted by $\langle \alpha_i \rangle_{i \in \mathbb{N}}$ and is called a \emph{limit} ordinal. If there is no risk of confusion, we denote limit ordinals simply as $\langle \alpha_i \rangle_i$ or $\langle \alpha_i \rangle$:
\vspace{5pt}

\begin{figure}[h!]
\centering
\begin{tikzpicture}[
    node style/.style={circle, draw, fill=black, inner sep=1.5pt},
    tree style/.style={
        draw, 
        fill=gray!20, 
        regular polygon, 
        regular polygon sides=3, 
        minimum size=1.5cm, 
        inner sep=0pt, 
        shape border rotate=180 
    }
]

\node[node style] at (0, 0) {};
\node at (0, -0.7) {$0$};

\begin{scope}[xshift=3.5cm] 
    \node[node style] (root2) at (0, 0) {};
    \node[tree style, label=center:$\alpha$] (alpha_subtree) at (0, 1.2) {};
    \draw (root2) -- (alpha_subtree.south);
    \node at (0, -0.7) {$\alpha+1$};
\end{scope}

\begin{scope}[xshift=8cm]
    \node[node style] (root3) at (0, 0) {};
    \node[tree style, label=center:$\alpha_0$] (a0) at (-1.5, 1.2) {};
    \node[tree style, label=center:$\alpha_1$] (a1) at (0, 1.2) {};
    \node[tree style, label=center:$\alpha_2$] (a2) at (1.5, 1.2) {};
    \node at (2.5, 1.2) {$\dots$};
    
    \draw (root3) -- (a0.south);
    \draw (root3) -- (a1.south);
    \draw (root3) -- (a2.south);
    
    \node at (0, -0.7) {$\langle \alpha_i \rangle_{i \in \mathbb{N}}$};
\end{scope}

\end{tikzpicture}
\caption{Zero, successor and limit constructive ordinals.}
\label{fig:ordinal_trees_final}
\end{figure}
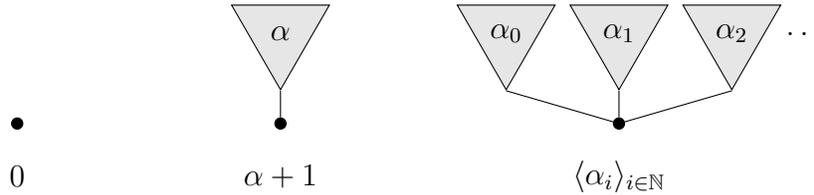

Equivalently, we can define $\Omega$ as the set of infinitary expressions constructed from the following grammar: 
\begin{itemize}
    \item $0 \in \ConsOrd$;  
    \item If $\alpha \in \ConsOrd$, then $\alpha + 1 \in \ConsOrd$;  
    \item If $\alpha_i \in \ConsOrd$ for all $i \in \mathbb{N}$, then $\langle \alpha_i \rangle_{i \in \mathbb{N}} \in \ConsOrd$.
\end{itemize}
Note that $\Omega$ supports a form of \emph{structural induction}, whereby one can prove a property for all elements of $\Omega$ by verifying that it holds for $0$ and is preserved under the successor and limit operations, as defined above.

The \emph{sub-tree ordering} $\prec$ on $\Omega$ is defined as the transitive closure of the following relations:  
\begin{itemize}
    \item $\alpha \prec \alpha + 1$, for all $\alpha \in \ConsOrd$;  
    \item $\alpha_m \prec \langle \alpha_i \rangle$, for all $\langle \alpha_i \rangle \in \ConsOrd$ and $m \in \mathbb{N}$.  
\end{itemize}
The relation $\prec$ is clearly well-founded. We define $\alpha \preceq \beta$ to mean that either $\alpha \prec \beta$ or $\alpha = \beta$. Some properties of $\prec$ follow directly from its definition. For instance, for any $\alpha \in \Omega$, we have $\alpha \nprec 0$. Moreover, for any $\alpha, \beta \in \Omega$, we have $\alpha \prec \beta + 1$ if and only if $\alpha = \beta$ or $\alpha \prec \beta$. Finally, $\alpha \prec \langle \beta_i \rangle$ if and only if there exists $i \in \mathbb{N}$ such that $\alpha \prec \beta_i$. A set $\mathsf{A} \subseteq \Omega$ of ordinals is called a \emph{downset} if, for any $\alpha \in \mathsf{A}$ and $\beta \prec \alpha$, we have $\beta \in \mathsf{A}$.  
A \emph{principal downset} is a downset of the form $\mathsf{D}_{\alpha} = \{\beta \in \Omega \mid \beta \prec \alpha\}$, where $\alpha \in \Omega$ is an ordinal.

Any natural number can be transformed into an ordinal by iteratively applying the successor operation. Formally, we define the map $o: \mathbb{N} \to \Omega$ by:
\[
o(0_N) \Def 0 \quad \text{and} \quad o(s_N(n)) \Def o(n) + 1,
\]
where $0_N$ and $s_N$ denote the number zero and the numeral successor function, respectively. Throughout this paper, we omit the subscript $N$ and simply write $o(n)$ as $n$ whenever there is no risk of confusion.

Next, we define the ordinal
$
\omega \Def \langle o(i+1) \rangle_i,
$
and observe that $\{\alpha \mid \alpha \prec \omega\}$ precisely corresponds to the range of $o$. One might initially expect to define $\omega$ as $\langle o(i) \rangle_i$, but the chosen definition is preferred in the literature due to its nice structural properties. For further details, see \cite{fairtlough1998ptchapter}.

To compare constructive ordinals with set-theoretic ordinals, let $\mathbf{On}$ be the class of all set-theoretic ordinals, with $\mathbf{0} \in \mathbf{On}$ as its least element and $\bm{\alpha} \mapsto \bm{\alpha+1}$ as its successor operation. There is a canonical mapping $[\![-]\!]: \Omega \to \mathbf{On}$ from constructive ordinals to set-theoretic ordinals defined as follows: $[\![0]\!] = \bm{0}$, $[\![\alpha + 1]\!] = [\![\alpha]\!] \bm{+ 1}$, and $[\![\langle \alpha_i \rangle]\!] = \bigcup_i [\![\alpha_i]\!]$. It is easy to see that the range of $[\![-]\!]$ consists precisely of the countable set-theoretic ordinals. 
However, this map is not injective, as a limit set-theoretic ordinal $\bm{\alpha}$ may correspond to different sequences converging to it. For instance, $[\![\langle o(i) \rangle]\!] = [\![\langle o(i + 1) \rangle]\!]$, yet $\langle o(i) \rangle \neq \langle o(i + 1) \rangle$.
Indeed, one can interpret constructive ordinals as the \emph{constructive} counterpart of set-theoretic ordinals, where each limit element is equipped with a canonical converging sequence.

It is straightforward to see that if $\alpha \preceq \beta$, then $[\![\alpha]\!] \subseteq [\![\beta]\!]$. However, the converse is not true as $\langle o(i) \rangle \npreceq \langle o(i + 1) \rangle$ while $[\![\langle o(i) \rangle]\!] = [\![\langle o(i + 1) \rangle]\!]$. Let $\mathbf{o}(n)$ denote the set-theoretic ordinal corresponding to $n \in \mathbb{N}$, and define $\bm{\omega} \in \mathbf{On}$ as the set-theoretic ordinal $\{\mathbf{o}(0), \mathbf{o}(1), \ldots\}$. Then, for any $n \in \mathbb{N}$, we have $[\![o(n)]\!] = \mathbf{o}(n)$, and $[\![\omega]\!] = \bm{\omega}$. Similar to the constructive case, when there is no risk of confusion, we denote $\bm{o}(n)$ by $\bm{n}$, for any $n \in \mathbb{N}$.

\subsection{Functions on constructive ordinals}
\label{sec:functions-const-ordinals}

In this subsection, we introduce a number of useful functions with (set-theoretic) ordinal inputs or outputs, and examine some of their properties. We shall refer to functions that take natural numbers and (set-theoretic) ordinals as inputs and return (set-theoretic) ordinals as \emph{(set-theoretic) ordinal functions}.

First, we define addition, multiplication, and exponentiation on constructive ordinals using their standard recursive definitions:
\begin{align*}
\alpha + 0 &\Def \alpha, &
\alpha + (\beta+1) &\Def (\alpha + \beta) + 1, &
\alpha + \langle \beta_i \rangle &\Def \langle \alpha + \beta_i \rangle, \\
\alpha \cdot 0 &\Def 0, &
\alpha \cdot (\beta+1) &\Def (\alpha \cdot \beta) + \alpha, &
\alpha \cdot \langle \beta_i \rangle &\Def \langle \alpha \cdot \beta_i \rangle, \\
\alpha^0 &\Def 1, &
\alpha^{(\beta+1)} &\Def (\alpha^\beta) \cdot \alpha, &
\alpha^{\langle \beta_i \rangle} &\Def \langle \alpha^{\beta_i} \rangle.
\end{align*}
For example, $\omega \cdot 2 = \omega + \omega = \langle \omega + 1, \omega + 2, \ldots \rangle$, $\omega^2 = \langle \omega, \omega \cdot 2, \ldots \rangle$, and $\omega^\omega = \omega^{\langle 1, 2, \ldots \rangle} = \langle \omega, \omega^2, \ldots \rangle$.

\begin{remark}
The expression $\alpha + 1$ has two interpretations: it can be read either as the successor of $\alpha$, or as the sum of $\alpha$ and $o(1)$. Fortunately, by the definition of addition, these two interpretations coincide. To see this, let us temporarily denote the successor of an ordinal $\beta$ by $s(\beta)$. Then, using the definition of addition, we have $\alpha + o(1) = \alpha + s(0) = s(\alpha + 0) = s(\alpha)$.
\end{remark}

For the set-theoretic counterpart, addition, multiplication, and exponentiation on set-theoretic ordinals can be defined in the usual way:
\begin{align*}
\bm{\alpha + 0} &\Def \bm{\alpha}, &
\bm{\alpha + (\beta+1)} &\Def \bm{(\alpha + \beta) + 1}, &
\bm{\alpha + \beta} &\Def \bigcup_{\bm{\gamma} \subset \bm{\beta}}\, \bm{\alpha + \gamma}, \\
\bm{\alpha \cdot 0} &\Def \bm{0}, &
\bm{\alpha \cdot (\beta+1)} &\Def \bm{(\alpha \cdot \beta) + \alpha}, &
\bm{\alpha \cdot \beta} &\Def \bigcup_{\bm{\gamma} \subset \bm{\beta}}\, \bm{\alpha \cdot \gamma}, \\
\bm{\alpha^0} &\Def \bm{1}, &
\bm{\alpha^{(\beta+1)}} &\Def \bm{(\alpha^\beta) \cdot \alpha}, &
\bm{\alpha^\beta} &\Def \bigcup_{\bm{\gamma} \subset \bm{\beta}}\, \bm{\alpha^\gamma}.
\end{align*}
where, in the third column, $\bm{\beta}$ is a non-zero limit set-theoretic ordinal. It is easy to see that $[\![\alpha+\beta]\!]=[\![\alpha]\!]\bm{+}[\![\beta]\!]$, $[\![\alpha \cdot \beta]\!]=[\![\alpha]\!] \bm{\cdot} [\![\beta]\!]$, and $[\![\alpha^\beta]\!]=[\![\alpha]\!]^{[\![\beta]\!]}$, for any $\alpha, \beta \in \Omega$.

Set-theoretic ordinals enjoy many interesting properties. For example, addition and multiplication are associative, multiplication distributes over addition on the left, and addition, multiplication, and exponentiation are all monotone with respect to $\subseteq$. Furthermore, addition is expansive; that is, $\bm{\alpha} \subseteq \bm{\alpha + \beta}$ and $\bm{\beta} \subseteq \bm{\alpha + \beta}$.
Some of these properties carry over to constructive ordinals, but not all. For instance, although $\alpha \preceq \alpha + \beta$ holds for all $\alpha, \beta \in \Omega$, it need not be the case that $\beta \preceq \alpha + \beta$. For example,  $1 + \omega \npreceq \omega$, because, as trees, we have $\langle o(i+2) \rangle \npreceq \langle o(i+1) \rangle$. Another counter-intuitive example is that although $\bm{0 \cdot \omega}=\bm{0}$, the constructive ordinal $0 \cdot \omega$ is the sequence of zeros, which is not equal to zero.
Typically, such anomalies are avoided by working with a restricted subclass known as the \emph{structured} constructive ordinals. However, as we will explain in Remark~\ref{WhyStructurednessFails}, the use of unstructured ordinals in this paper is unavoidable. We must therefore proceed with care regarding the specific properties we rely on. In what follows, we summarize the properties of constructive ordinals that are used throughout this paper.

\begin{lemma}\label{AdditionProp}
For any $\alpha, \beta, \gamma \in \Omega$, the following hold:
\begin{itemize}
    \item[$(i)$] 
  $0 + \alpha = \alpha$, $1 \cdot \alpha=\alpha \cdot 1=\alpha$, and $\alpha^1=\alpha$.
    \item[$(ii)$] 
    $\alpha + (\beta + \gamma) = (\alpha + \beta) + \gamma$,
    \item[$(iii)$] 
    If $\beta \neq 0$, then $\alpha \prec \alpha + \beta$. Hence, $\alpha \preceq \alpha+\beta$.
    \item[$(iv)$] 
    If $\beta \prec \gamma$, then $\alpha+\beta \prec \alpha + \gamma$. The same also holds for $\preceq$.
    \item[$(v)$]
    If $\alpha+\beta = 0$, then $\alpha=\beta=0$.
     \item[$(vi)$]
    If $\alpha\cdot\beta = 0$, then either $\alpha=0$ or $\beta=0$.
     \item[$(vii)$]
   If $\alpha \neq 0$, then $\alpha^{\beta} \neq 0$.
\end{itemize}
\end{lemma}
\begin{proof}
For $(i)$, we first use an induction on $\alpha$ to prove $0+\alpha=\alpha$. If $\alpha=0$, since $0+0=0$, there is nothing to prove. If $\alpha=\alpha'+1$ is a successor, then by the definition of addition and the induction hypothesis, we have: 
\[
0+\alpha = 0+(\alpha'+1) = (0+\alpha')+1 = \alpha'+1 = \alpha.
\]
If $\alpha = \langle \alpha_i \rangle$ is a limit, then again by the definition of addition and the induction hypothesis, we have 
$0+\alpha = \langle 0+\alpha_i \rangle = \langle \alpha_i \rangle = \alpha$. To prove $\alpha \cdot 1=\alpha$, by definition of multiplication and $0+\alpha=\alpha$, we have
$
\alpha \cdot 1=\alpha \cdot 0 + \alpha=0+\alpha=\alpha
$.

To prove $1 \cdot \alpha=\alpha$, we again use an induction on $\alpha$. If $\alpha=0$, we have $1 \cdot \alpha=1 \cdot 0 =0=\alpha$. If $\alpha=\alpha'+1$ is a successor, by the induction hypothesis, we have:
\[
1 \cdot \alpha=1 \cdot (\alpha'+1)=1 \cdot \alpha'+1=\alpha'+1=\alpha.
\]
If $\alpha=\langle \alpha_i \rangle$ is a limit, by the induction hypothesis, we have $1 \cdot \alpha=\langle 1\cdot \alpha_i \rangle=\langle \alpha_i \rangle=\alpha$. Finally, to show $\alpha^1=\alpha$, by the definition of exponentiation and $1 \cdot \alpha=\alpha$, we have
\[
\alpha^1=\alpha^0 \cdot \alpha=1 \cdot \alpha=\alpha.
\]

For $(ii)$, we use induction on $\gamma$. If $\gamma=0$, we have $\alpha+(\beta+0)=\alpha+\beta=(\alpha+\beta)+0$. If $\gamma=\gamma'+1$ is a successor, using the definition of addition and the induction hypothesis, we have
\[
\alpha+(\beta+(\gamma'+1))=\alpha+((\beta+\gamma')+1)=(\alpha+(\beta+\gamma'))+1
\]
\[
=((\alpha+\beta)+\gamma')+1=(\alpha+\beta)+(\gamma'+1).
\]
If $\gamma=\langle \gamma_i \rangle$ is a limit, then by the definition of addition and the induction hypothesis, we have
\[
\alpha+(\beta+\langle \gamma_i \rangle)=\alpha+\langle \beta+\gamma_i \rangle= \langle \alpha+(\beta+\gamma_i) \rangle=\langle (\alpha+\beta)+\gamma_i \rangle=(\alpha+\beta)+\langle \gamma_i \rangle.
\]

For $(iii)$, for the first part, we use induction on $\beta$. For $\beta=0$, there is nothing to prove. If $\beta=\beta'+1$ is a successor, then either $\beta'=0$ or $\beta' \neq 0$. If $\beta'=0$, then $\beta=1$ and we have $\alpha \prec \alpha+1=\alpha+\beta$. If $\beta' \neq 0$, then by the induction hypothesis, we have 
\[
\alpha \prec \alpha + \beta' \prec (\alpha+ \beta')+1=\alpha +(\beta'+1)=\alpha+\beta.
\]
If $\beta=\langle \beta_i \rangle$ is a limit, then either $\beta_0=0$ or $\beta_0 \neq 0$. If $\beta_0=0$, then 
\[
\alpha=\alpha+\beta_0 \prec \langle \alpha+\beta_i \rangle=\alpha+\beta.
\]
If $\beta_0 \neq 0$, then by the induction hypothesis, we have
\[
\alpha \prec \alpha+\beta_0 \prec \langle \alpha + \beta_i \rangle=\alpha+\langle \beta_i \rangle=\alpha+\beta.
\]
For the second part, it is enough to use the first part and notice that if $\beta =0$, we have $\alpha=\alpha+\beta$.

For $(iv)$, for the first part, we use induction on $\gamma$. If $\gamma=0$, then as $\beta \prec 0$ is impossible, there is nothing to prove. If $\gamma=\gamma'+1$ is a successor, then as $\beta \prec \gamma'+1$, either $\beta \prec \gamma'$ or $\beta=\gamma'$. 
In the first case, by the induction hypothesis, we have 
\[
\alpha+\beta \prec \alpha+\gamma' \prec (\alpha+\gamma')+1= \alpha+(\gamma'+1)=\alpha+\gamma.
\]
In the second case, we have $\alpha+\beta=\alpha+\gamma' \prec (\alpha+\gamma')+1=\alpha+(\gamma'+1)$.
If $\gamma=\langle \gamma_i \rangle $ is a limit, then as $\beta \prec \langle \gamma_i \rangle$, there is $i \in \mathbb{N}$ such that $\beta \prec \gamma_i$. By the induction hypothesis, we have $\alpha+\beta \prec \alpha+\gamma_i \prec \langle \alpha+\gamma_i \rangle=\alpha+\gamma$. For the second part, it is enough to use the first part and notice that $\beta=\gamma$ implies $\alpha+\beta=\alpha+\gamma$.

For $(v)$, if $\beta \neq 0$, by $(iii)$, we obtain $\alpha \prec \alpha + \beta = 0$, which is impossible. 
Therefore, $\beta=0$ which implies $0=\alpha+\beta=\alpha+0=\alpha$.

For $(vi)$, we prove the claim by induction on $\beta$. If $\beta = 0$, there is nothing to prove. If $\beta = \beta' + 1$ is a successor, then
$
0 =\alpha \cdot \beta = \alpha \cdot (\beta' + 1) = \alpha \cdot \beta' + \alpha.
$
Therefore, by $(v)$, we get $\alpha = 0$. If $\beta = \langle \beta_i \rangle$ is a limit, then $0=\alpha \cdot \beta = \langle \alpha \cdot \beta_i \rangle$ is a limit which is a contradiction.

For $(vii)$, we use induction on $\beta$. For $\beta = 0$, we have $\alpha^{\beta} = 1 \neq 0$. If $\beta = \beta' + 1$ is a successor, then
$
0 = \alpha^{\beta + 1} = \alpha^{\beta} \cdot \alpha.
$
By $(vi)$, either $\alpha^{\beta} = 0$ or $\alpha = 0$. The first is impossible by the induction hypothesis, and the second contradicts the assumption. 
If $\beta = \langle \beta_i \rangle$ is a limit, then $0=\alpha^\beta = \langle \alpha^{\beta_i} \rangle$ is a limit which is a contradiction.
\end{proof}

\begin{remark}
By Lemma~\ref{AdditionProp}, addition on constructive ordinals is associative. Hence, we may use the usual summation notation $\sum_{i=1}^m \alpha_i$ without parentheses. 
\end{remark}

We introduce several functions and hierarchies of functions that will be used later. Define the \emph{predecessor function} $p(\alpha, n)$ recursively as follows:
\[
p(0, n) \Def 0, \quad p(\alpha + 1, n) \Def \alpha, \quad \text{and} \quad p(\langle \alpha_i \rangle, n) \Def \alpha_n.
\]
In terms of trees, if $\alpha$ is a single node, then so is $p(\alpha, n)$. However, if $\alpha$ is a successor (resp.\ limit) ordinal, $p(\alpha, n)$ returns the only child (resp.\ the $n$-th child) of $\alpha$. Fixing $n \in \mathbb{N}$, we define the \emph{$n$-predecessor} function as the map $\alpha \mapsto p(\alpha, n)$ on $\Omega$. Next, define the \emph{iterated predecessor} function $R^b(i, \alpha, n)$ by:
\[
R^b(0, \alpha, n) = \alpha, \quad R^b(i+1, \alpha, n) = p(R^b(i, \alpha, n), n).
\]
Clearly, $R^b(i, \alpha, n)$ is the result of $i$ applications of the $n$-predecessor function to $\alpha$. Since $\prec$ is well-founded, there exists an $l \in \mathbb{N}$ such that $R^b(l, \alpha, n) = 0$:
\[\footnotesize\begin{tikzcd}
	{0=R^b(l, \alpha, n) } & {R^b(l-1, \alpha, n)} & {\ldots } & {R^b(1, \alpha, n) } & {R^b(0, \alpha, n)=\alpha}
	\arrow["\prec"{marking, allow upside down}, draw=none, from=1-1, to=1-2]
	\arrow["{p(-, n)}"', curve={height=30pt}, from=1-2, to=1-1]
	\arrow["\prec"{marking, allow upside down}, draw=none, from=1-2, to=1-3]
	\arrow["\prec"{marking, allow upside down}, draw=none, from=1-3, to=1-4]
	\arrow["\prec"{marking, allow upside down}, draw=none, from=1-4, to=1-5]
	\arrow["{p(-, n)}"', curve={height=30pt}, from=1-5, to=1-4]
\end{tikzcd}\]
We call the decreasing sequence $\{R^b(i, \alpha, n)\}_{i=0}^l$ the \emph{sequence of $n$-predecessors of $\alpha$}, where $l \in \mathbb{N}$ is the least number satisfying $R^b(l, \alpha, n)=0$. In terms of trees, this sequence consists of subtrees of $\alpha$, starting from $\alpha$ itself and proceeding along the unique branch at successor nodes and the $n$-th branch at limit nodes, until reaching a leaf.

We now introduce two ordinal-indexed hierarchies of numeral functions:

\begin{definition}
	\label{def:slow-growing-hierarchy}
	The
\emph{slow-growing hierarchy} is
the family $\{ G_\alpha: \mathbb{N} \to \mathbb{N} \}_{\alpha \in \ConsOrd}$ of functions defined by
\[
G_0(n) := 0, \quad
G_{\alpha+1}(n) := G_\alpha(n) + 1, \quad \text{and} \quad
G_{\langle \alpha_i \rangle}(n) := G_{\alpha_n}(n).
\]
The \emph{length hierarchy} is
the family $\{ L_\alpha: \mathbb{N} \to \mathbb{N} \}_{\alpha \in \ConsOrd}$ of functions defined by
\[
L_0(n) := 0, \quad
L_{\alpha+1}(n) := L_\alpha(n) + 1, \quad \text{and} \quad
L_{\langle \alpha_i \rangle}(n) := L_{\alpha_n}(n) + 1.
\]
For convenience, we sometimes write $G(\alpha,n)$ and $L(\alpha,n)$ instead of $G_{\alpha}(n)$ and $L_{\alpha}(n)$, in order to emphasize that $\alpha$ also serves as an input to the function.
\end{definition}
By induction on $\alpha \in \Omega$, one can easily show that $L_{\alpha}(n)$ is the least number $l$ such that $R^b(l, \alpha, n) = 0$. This justifies the term ``length hierarchy'' that we use. In terms of trees, $L_{\alpha}(n)$ is the total number of successor and limit ordinals in the sequence of $n$-predecessors of $\alpha$, while $G_{\alpha}(n)$ counts only the number of successor ordinals in that sequence. Therefore, $G_{\alpha}(n) \leq L_{\alpha}(n)$ for any $n \in \mathbb{N}$.

Finally, the last function we introduce in this section is the function $R(i, \alpha, n)$ defined by:
\[
R(i, \alpha, n) \Def R^b(L_{\alpha}(n) \dotminus i, \alpha, n)
\]
where $\dotminus$ is the truncated difference operation, i.e., 
$a \dotminus b = a - b$
if $b \leq a$, and $a \dotminus b = 0$ otherwise.
Clearly, $\dotminus \in \GrzClass{2}$.
In words, $R(i, \alpha, n)$ enumerates the sequence of $n$-predecessors of $\alpha$ in a forward manner, that is, from $0$ up to $\alpha$. In particular, we have $R(0, \alpha, n) = 0$ and $R(L_{\alpha}(n), \alpha, n) = \alpha$:
\[\small\begin{tikzcd}
	{0=R^b(L_{\alpha}(n), \alpha, n) } & {\ldots } & {R^b(1, \alpha, n) } & {R^b(0, \alpha, n)=\alpha} \\
	{0=R(0, \alpha, n)} & \ldots & {R(L_{\alpha}(n)-1, \alpha, n)} & {R(L_{\alpha}(n), \alpha, n)=\alpha}
	\arrow["\prec"{marking, allow upside down}, draw=none, from=1-1, to=1-2]
	\arrow["{=}"{marking, allow upside down}, draw=none, from=1-1, to=2-1]
	\arrow["\prec"{marking, allow upside down}, draw=none, from=1-2, to=1-3]
	\arrow["\prec"{marking, allow upside down}, draw=none, from=1-3, to=1-4]
	\arrow["{=}"{marking, allow upside down}, draw=none, from=1-3, to=2-3]
	\arrow["{=}"{marking, allow upside down}, draw=none, from=1-4, to=2-4]
	\arrow["\prec"{marking, allow upside down}, draw=none, from=2-1, to=2-2]
	\arrow["\prec"{marking, allow upside down}, draw=none, from=2-2, to=2-3]
	\arrow["\prec"{marking, allow upside down}, draw=none, from=2-3, to=2-4]
\end{tikzcd}\]
Even more generally, $R^b(j,\alpha,n) = 0$ and
$R(j,\alpha,n) = \alpha$, for any $j \geq L_\alpha(n)$.

In the following, we state some properties of the functions in the slow-growing and length hierarchies:

\begin{lemma}\label{lem:g-properties}
For any $\alpha, \beta \in \ConsOrd$ and $n \in \mathbb{N}$, the following hold: 
\begin{itemize}
\item[$(i)$]
	$G_{\alpha+\beta}(n) = G_\alpha(n) + G_\beta(n)$ and $
	G_{\alpha \cdot \beta}(n) = G_\alpha(n) \cdot G_\beta(n)$.
    \item[$(ii)$] 
    If $G_{\alpha}(n) \geq 1$, then $
	G_{\alpha^\beta}(n) = G_\alpha(n)^{G_\beta(n)}$.
    \item[$(iii)$] 
If $\alpha \neq 0$, then $L_{\alpha}(n) \geq 1$.
    \item[$(iv)$]
    $L_{\alpha+\beta}(n) = L_\alpha(n) + L_\beta(n)$.
     \item[$(v)$]
If $\alpha \neq 0$, then  $L_{\alpha \cdot \beta}(n) \leq L_\alpha(n) \cdot L_\beta(n)$.
    \item[$(vi)$]
    If $L_{\alpha}(n) \geq 2$, then $L_{\alpha^\beta}(n) \leq L_\alpha(n)^{L_\beta(n)}$.
\end{itemize}
\end{lemma}
\begin{proof}
For $(i)$ and $(ii)$, it is enough to use an induction on $\beta$. The proof can also be found in \cite{fairtlough1998ptchapter}. 
For $(iii)$, if $\alpha \neq 0$, then either $\alpha$ is a successor or a limit. If $\alpha = \alpha' + 1$ is a successor, then $L_\alpha(n) = L_{\alpha'}(n) + 1 \geq 1$. If $\alpha = \langle \alpha_i \rangle$ is a limit, then $L_\alpha(n) = L_{\alpha_n}(n) + 1 \geq 1$. 

The proof of $(iv)$ is easy and similar to that of $(i)$.  
For $(v)$, we use an induction on $\beta$. For $\beta = 0$, as $\alpha \cdot \beta = 0$ and $L_0(n) = 0$, there is nothing to prove. 
If $\beta = \beta' + 1$ is a successor, then using part $(iv)$ and the induction hypothesis, we have:
\[
L_{\alpha \cdot (\beta' + 1)}(n) = L_{\alpha \cdot \beta' + \alpha}(n) = L_{\alpha \cdot \beta'}(n) + L_\alpha(n) 
\leq L_\alpha(n) L_{\beta'}(n) + L_\alpha(n)
\]
\[
= L_\alpha(n)(L_{\beta'}(n) + 1) = L_\alpha(n) L_{\beta' + 1}(n).
\]
If $\beta = \langle \beta_i \rangle$ is a limit, then since $\alpha \neq 0$, we have $L_\alpha(n) \geq 1$, by part $(iii)$. Therefore, using the induction hypothesis, we obtain:
\[
L_{\alpha \cdot \langle \beta_i \rangle}(n) = L_{\langle \alpha \cdot \beta_i \rangle}(n) = L_{\alpha \cdot \beta_n}(n) + 1 
\leq L_\alpha(n) L_{\beta_n}(n) + 1 \leq L_\alpha(n) L_{\beta_n}(n) + L_\alpha(n)
\]
\[
\leq L_\alpha(n)(L_{\beta_n}(n) + 1) = L_\alpha(n) L_{\langle \beta_i \rangle}(n).
\]

For $(vi)$, we use an induction on $\beta$. For $\beta = 0$, as $\alpha^{\beta} = 1$, $L_\alpha(n) \geq 2$ and $L_\beta(n) = 0$, there is nothing to prove. 
If $\beta = \beta' + 1$ is a successor, then as $L_{\alpha}(n) \geq 2$, we have $\alpha\neq 0$. Hence, by Lemma \ref{AdditionProp}, part $(vii)$, we have $\alpha^{\beta'} \neq 0$. Therefore, by part $(v)$ and the induction hypothesis, we have:
\[
L_{\alpha^{(\beta' + 1)}}(n) = L_{\alpha^{\beta'} \cdot \alpha}(n) \leq L_{\alpha^{\beta'}}(n) L_\alpha(n) 
\leq L_\alpha(n)^{L_{\beta'}(n)} L_\alpha(n)
\]
\[
= L_\alpha(n)^{L_{\beta'}(n) + 1} = L_\alpha(n)^{L_{\beta' + 1}(n)}.
\]
\noindent If $\beta = \langle \beta_i \rangle$ is a limit, then using $L_\alpha(n) \geq 2$ and the induction hypothesis, we reach:
\[
L_{\alpha^{\langle \beta_i \rangle}}(n) = L_{\langle \alpha^{\beta_i} \rangle}(n) = L_{\alpha^{\beta_n}}(n) + 1 
\leq L_\alpha(n)^{L_{\beta_n}(n)} + 1 \leq 2L_\alpha(n)^{L_{\beta_n}(n)}
\]
\[
\leq L_\alpha(n)^{L_{\beta_n}(n) + 1} = L_\alpha(n)^{L_{\langle \beta_i \rangle}(n)}.
\qedhere
\]
\end{proof}

\section{Predicative Ordinal Recursive Functions}
\label{sec:pred-ord-rec-functions}
In this section, we introduce the class $\PredFuncClass{\OrdClass}$ of predicative ordinal-recursive functions up to the ordinals in $\mathsf{A}$, where $\mathsf{A} \subseteq \Omega$ is a downset of ordinals. This generalizes $\PredFuncClass{}$ by extending recursion from finite numbers to the infinite ordinals in $\mathsf{A}$. We also provide illustrative examples and establish several properties of $\PredFuncClass{\OrdClass}$.

We first define the class $\mathcal{C}_{\mathsf{A}}$ of functions with natural number outputs and three distinct types of inputs: ordinal variables instantiated by the ordinals in $\OrdClass$, numeral variables in \emph{parameter} (also called \emph{normal}) positions, and numeral variables in \emph{safe} positions.
We use the notation $f(\bar{\NormalOrdA}, \bar{\SafeOrdA}; \bar{\SafeNumbA})$ to represent these inputs. Lowercase Greek letters denote ordinals, lowercase Latin letters such as $m$, $n$, and $p$ are used for parameters, and the letters $a$, $b$, and $c$ are reserved for the safe variables. The semicolon separates the tuple $\bar{\SafeNumbA}$ of safe variables from the rest.

\begin{definition}[The class $\ClassName{\OrdClass}$]
Let $\mathsf{A} \subseteq \Omega$ be a downset of ordinals. We define $\ClassName{\OrdClass}$ as the smallest class of numeral functions with the inputs from $\mathsf{A}$ and $\mathbb{N}$, containing the initial functions (1)--(5) and closed under (6)--(9):
	\begin{enumerate}
		\item (constant)
        the number $0$ as a nullary function
		\item (projections)
$\pi^{k,l}_j(\bar{\SafeOrdA};\bar{\SafeNumbA})=\xi_j$ where  
$\bar{\SafeOrdA}=\SafeOrdA_{1},\ldots,\SafeOrdA_{k}$,
$\bar{\SafeNumbA}=\SafeNumbA_{k+1},\ldots,\SafeNumbA_{k+l}$
and $\xi \in \{ \SafeOrdA, \SafeNumbA \}$ depending on $j$ ($1 \leq j \leq k+l$).    
		\item (successor) 
        $s(;\SafeNumbA)=\SafeNumbA+1$
		\item (predecessor) 
        $\Pred(;0) = 0$ and
        $\Pred(;\SafeNumbA+1) = \SafeNumbA$
		\item (conditional)
        $
		\Cond(;\SafeNumbA,\SafeNumbB,\SafeNumbC)= \SafeNumbB$ if $\SafeNumbA=0$ and $\Cond(;\SafeNumbA,\SafeNumbB,\SafeNumbC)=\SafeNumbC$, otherwise
		\item (predicative ordinal recursion)
        for $g,h_{\mathsf{suc}},h_{\mathsf {lim}},q(\bar{\SafeOrdA};)\in\ClassName{\OrdClass}$,
		define $f$ by
        \begin{align*}
			f(0, \bar \NormalOrdB, \bar \SafeOrdA; \bar \SafeNumbA)
			&= g(\bar \NormalOrdB,
			\bar \SafeOrdA; \bar \SafeNumbA)\\
			f(\NormalOrdA+1,\bar \NormalOrdB, \bar \SafeOrdA; \bar \SafeNumbA) &= h_{\mathsf{suc}}(\NormalOrdA, \bar \NormalOrdB,\bar \SafeOrdA; 
			f(\NormalOrdA,\bar \NormalOrdB, \bar \SafeOrdA; \bar \SafeNumbA), \bar \SafeNumbA)\\
			f(\langle \NormalOrdA_i \rangle_i, \bar \NormalOrdB,\bar \SafeOrdA;\bar \SafeNumbA)
			&= h_{\mathsf{lim}}(\langle \NormalOrdA_i \rangle_i, \bar \NormalOrdB, \bar \SafeOrdA;
			f(\NormalOrdA_{q(\bar \SafeOrdA;)}, \bar \NormalOrdB, \bar \SafeOrdA; \bar \SafeNumbA),
			\bar \SafeNumbA
			)
		\end{align*}
		The function $q(\bar \SafeOrdA;)$ is called the \emph{selector} of the recursion as it selects the smaller ordinal $\NormalOrdA_{q(\bar \SafeOrdA;)}$ that is used instead of 
        $\langle \NormalOrdA_i \rangle_i$ in the limit ordinal case.
		\item (safe composition)
		for $h,
        \bar s,
		 \bar t\in\ClassName{\OrdClass}$,
		define the new function $f$ by
        \[
		f(\bar \NormalOrdA, \bar \SafeOrdA; \bar \SafeNumbA)
		= h(\bar{\NormalOrdA},\bar s(\bar \NormalOrdA, \bar \SafeOrdA;); \bar t(\bar \NormalOrdA, \bar\SafeOrdA; \bar \SafeNumbA))
		\]
        \item (constant substitution)
        for any $k \geq 1$, any $1 \leq i \leq k$, any function $g(\NormalOrdA_1, \ldots, \NormalOrdA_k, \bar \SafeOrdA; \bar \SafeNumbA) \in \ClassName{\OrdClass}$,
        and any ordinal
        $\alpha \in \OrdClass$,
        define the new function $f \in \ClassName{\OrdClass}$
        by
        \[
        f(\NormalOrdA_1,\ldots,\NormalOrdA_{i-1},\NormalOrdA_{i+1},\ldots,\NormalOrdA_{k}, \bar \SafeOrdA; \bar \SafeNumbA)
        =
        g(\NormalOrdA_1,\ldots,\NormalOrdA_{i-1},\alpha,\NormalOrdA_{i+1},\ldots,\NormalOrdA_{k},\bar\SafeOrdA;\bar\SafeNumbA)
        \]
        \item (structural rules)
        \begin{enumerate}[(i)]
            \item (exchange)
            for $g \in \ClassName{\OrdClass}$, define the new function $f$ by
            \[
f(\bar{\NormalOrdA},\NormalOrdB_1,\NormalOrdB_2,\bar{\NormalOrdC},\bar \SafeOrdA; \bar \SafeNumbA)
            =g(\bar{\NormalOrdA},\NormalOrdB_2,\NormalOrdB_1,\bar{\NormalOrdC},\bar \SafeOrdA; \bar \SafeNumbA)
            \]
            \item (weakening)
            for $g \in \ClassName{\OrdClass}$, define the new function $f$ by
            \[
f(\bar{\NormalOrdA},\NormalOrdB,\bar{\NormalOrdC},\bar \SafeOrdA; \bar \SafeNumbA)
            =g(\bar{\NormalOrdA},\bar{\NormalOrdC},\bar \SafeOrdA; \bar \SafeNumbA)
            \]
            \item (contraction)
            for $g \in \ClassName{\OrdClass}$, define the new function $f$ by
            \[
f(\bar{\NormalOrdA},\NormalOrdB,\bar{\NormalOrdC},\bar \SafeOrdA; \bar \SafeNumbA)
            =g(\bar{\NormalOrdA},
            \NormalOrdB,\NormalOrdB,\bar{\NormalOrdC},\bar \SafeOrdA; \bar \SafeNumbA)
            \]
        \end{enumerate}
	\end{enumerate}
We denote $\mathcal{C}_{\mathsf{D}_{\alpha}}$ by $\mathcal{C}_{\alpha}$, for any $\alpha \in \Omega$.
\end{definition}

\begin{remark}\label{rem:pred-rec-ord-funcs}
First, all basic functions are numeral functions with no ordinal input. Specifically, given the constraint that all the functions in the class $\mathcal{C}_{\mathsf{A}}$ output numbers, the projections do not have any ordinal input. However, using weakening, we can have projections with dummy ordinal variables that we never project over. Second, when defining $f$ by predicative ordinal recursion, one can only use the lower values of $f$ in the safe position. Moreover, in the limit case, we use the function $q(\bar{n};)$ to specify the lower value of $f$ that the recursion call uses. Note that only functions without any ordinal or safe input are allowed for $q(\bar{n};)$.
Third, in safe composition, no composition is allowed in ordinal positions, as there is no way to construct functions that output ordinals. Moreover, for the normal positions, one can only use functions with no safe inputs while for the safe position, all numeral variables are allowed. Therefore, parameters are allowed anywhere in the composition. This explains the intuition behind the normal and the safe positions. The safe variable is used for composition as before with Bellantoni-Cook characterization. But now the recursion operates on ordinals and the normal variables are there to be used in the selector function $q(\bar{n}; )$. 
Fourth, one can substitute any number of ordinal variables by constants from $\mathsf{A}$. Intuitively, this ensures that all recursions proceed only up to ordinals contained in $\mathsf{A}$. Moreover, since numerals can be constructed as nullary functions by successive composition of the successor function with the zero function, any numeral variable can be made a constant using safe composition. Therefore, no separate constant substitution is needed for numbers.
Fifth, structural rules guarantee that one can define functions 
by structural changes in ordinal inputs, namely by
exchanging input positions, by adding dummy variables and by removing duplicates. These rules are already admissible for numeral positions since  they admit composition with projections. However, since functions with ordinal outputs are not allowed, ordinal variables cannot be handled in the same way. Therefore, structural rules are necessary to provide analogous flexibility for ordinal variables. Sixth, in predicative ordinal recursion, the recursion is always performed on the first ordinal variable $\mu$. However, by using the exchange rule, it is clear that recursion on other ordinal variables is also possible. Seventh, in safe composition, we assumed that the ordinal variables in all the functions involved in the composition are the same. Typically, composition is defined in a more relaxed way, allowing different functions to use different variables. This relaxed version is still possible, as one can apply the weakening rule to introduce dummy variables, thereby equalizing the sets of ordinal variables. 
For numeral variables, a similar adjustment can be made by composing with projections.
\end{remark}

\begin{remark}\label{RemOnCr}
Let $r \in \mathbb{N}$ be a fixed natural number. Consider the following function as a generalization of the conditional function $C$:
\[
C_r(; a, b_0, \dots, b_r, c) = \begin{cases} 
    b_0 & \text{if } a = 0, \\ 
    b_1 & \text{if } a = 1, \\        
    \;\vdots & \;\;\quad\vdots \\
    b_r & \text{if } a = r, \\ 
    c   & \text{if } a > r.
\end{cases} 
\]
In fact, $C = C_0$. Using safe composition of the conditional function with itself and with the predecessor function, it is easy to see that $C_r \in \ClassName{\OrdClass}$. For instance, one can define $C_1(; a, b_0, b_1, c)$ as
$C(; a, b_0,\,
    C(; \Pred(; a), b_1, c)
)$.
\end{remark}

In the following, we will see some examples of functions in the class $\mathcal{C}_{\mathsf{A}}$.

\begin{example}\label{ex:grz-def}
	Define the numeral function
	$G(\NormalOrdA,\SafeOrdA;) \in \ClassName{\OrdClass}$ by $G(0, \SafeOrdA;) = 0$,
		$G(\NormalOrdA+1, \SafeOrdA;) = s(; G(\NormalOrdA, \SafeOrdA;))$, and
		$G(\langle \NormalOrdA_i \rangle, \SafeOrdA;) =G(\NormalOrdA_n, \SafeOrdA;)$.
    Observe that $G$ has a single ordinal input, and a single number as normal input (no input in safe position).
Moreover, note that $G(\NormalOrdA, \SafeOrdA;)=G_{\mu}(n)$, for any $\mu \in \OrdClass$ and any $n \in \mathbb{N}$. Similarly, one can define the numeral function
	$L(\NormalOrdA,\SafeOrdA;) \in \ClassName{\OrdClass}$ by $L(0, \SafeOrdA;) = 0$,
		$L(\NormalOrdA+1, \SafeOrdA;) = s(; L(\NormalOrdA, \SafeOrdA;))$, and
		$L(\langle \NormalOrdA_i \rangle, \SafeOrdA;) =s( ;L(\NormalOrdA_n, \SafeOrdA;))$.
\end{example}

In the next example, we present an instance of a general technique for simulating predicative recursion on numbers using predicative ordinal recursion; see Lemma~\ref{lem:ord-to-nat-and-back} for the full technique.

\begin{example}\label{ExampleOfUsualRecbyOrdRec}
Let $\mathsf{A}$ be a downset of ordinals such that $\omega \in \mathsf{A}$. We show that the function $\mathrm{sum}_0(n; a)=n+a$ is in $\mathcal{C}_{\mathsf{A}}$.
First, using predicative ordinal recursion, define the function $f(\mu, n; a)$ by:
\[
f(0, n; a) = a, \quad f(\mu + 1, n; a) = s(; f(\mu, n; a)), \quad f(\langle \mu_i \rangle, n; a) = f(\mu_n, n; a).
\]
It is easy to see that for any natural number $l \in \mathbb{N}$, we have $f(o(l), n; a) = l + a$. Moreover, recalling that $\omega$ is defined as $\langle o(i+1) \rangle_i$, we obtain:
\[
f(\omega, n; a) = f(o(n+1), n; a) = n + a + 1.
\]
Now, by constant substitution and the assumption that $\omega \in \mathsf{A}$, we may define
$\mathrm{sum}_0(n; a) = \Pred(; f(\omega, n; a))$.
\end{example}

To gauge the power of predicative ordinal recursion, we isolate the part of $\ClassName{\OrdClass}$ that only operates over~$\mathbb{N}$. This way, we can compare these functions to known classes of numeral functions. 
	\begin{definition}[$\PredFuncClass{\OrdClass}$]
		For any downset $\OrdClass \subseteq \ConsOrd$ of ordinals,
        let
		$
		\PredFuncClass{\OrdClass}
		$ be the class of all functions
        $f : \NatSet^k \to \NatSet$ such that there is $\hat f \in \ClassName{\OrdClass}$
        satisfying
		$f(\bar n) = \hat f(\bar{n};)$
        for any $\bar{n} \in \mathbb{N}$. We denote $\PredFuncClass{\mathsf{D}_{\alpha}}$ by $\PredFuncClass{\alpha}$, for any $\alpha \in \Omega$.
	\end{definition}

\begin{example}\label{Exam:GInPred}
Let $\mathsf{A} \subseteq \Omega$ be a downset of ordinals. Example~\ref{ex:grz-def} shows that for any $\mu \in \OrdClass$, we have $G_{\mu}, L_{\mu} \in \PredFuncClass{\OrdClass}$. As another example, if $\omega \in \mathsf{A}$, Example~\ref{ExampleOfUsualRecbyOrdRec} shows that the function $\mathrm{sum}_0(n; a) = n + a$ is in $\mathcal{C}_{\mathsf{A}}$. Therefore, using safe composition as in Example~\ref{ex:add-mult-bc}, the function $\mathrm{sum}(n, m;) = n + m$ is in $\mathcal{C}_{\mathsf{A}}$ and hence in $\PredFuncClass{\mathsf{A}}$.
\end{example}

\subsection{Simulating $\mathcal{E}_2$}

The class $\PredFuncClass{\mathsf{A}}$ is a robust class of functions. We have seen that $\PredFuncClass{\mathsf{A}}$ contains some basic functions. It is also clear that $\PredFuncClass{\OrdClass}$ is closed under composition, which shows that it contains many more natural functions. However, one can go one step further and show (Corollary~\ref{cor:E2-subset-predr} below) that the class $\PredFuncClass{\OrdClass}$ contains all linear-space computable functions, provided that $o(r) + \omega \in \mathsf{A}$ for some $r \in \mathbb{N}$.

To obtain this result, recall that $\mathcal{E}_2 = \PredFuncClass{}$ from Theorem \ref{thm:cb} and the definition of $\PredFuncClass{}$ in terms of the class $B$. Therefore, it suffices to simulate functions in $B$ by functions in $\ClassName{\OrdClass}$, using the idea that $\ClassName{\OrdClass}$ generalizes $B$ by extending the predicative recursion of $B$ from finite numbers to ordinals in $\OrdClass$. 
For this purpose, we first establish the following lemma, which shows how to simulate finite ordinal inputs by \emph{normal} inputs (and vice versa) in functions of the class $\ClassName{\OrdClass}$. To simulate finite ordinals with numbers, we use a technique similar to that explained in Example \ref{ExampleOfUsualRecbyOrdRec}. For the converse direction, we use slow-growing functions that imitate numbers by ordinals. 
For the simulation to work, it is important to assume that $\mathsf{A}$ contains an ordinal of the form $o(r)+\omega$, for some $r \in \mathbb{N}$.

\begin{lemma}\label{lem:ord-to-nat-and-back} Let $\OrdClass \subseteq \Omega$ be a downset of ordinals such that $o(r)+\omega \in \OrdClass$, for some $r \in \NatSet$. Then:
    \begin{enumerate}[(i)]
        \item For any
        $f(\bar{\alpha},\beta,
        \bar{\gamma},\bar{\SafeOrdA};\bar{\SafeNumbA}) \in \ClassName{\OrdClass}$, there is
        $f'(\bar{\alpha}, \bar{\gamma},\bar{\SafeOrdA},\SafeOrdB;\bar{\SafeNumbA}) \in \ClassName{\OrdClass}$
        such that
        $f(\bar{\alpha},o(m),
        \bar{\gamma},\bar{\SafeOrdA};\bar{\SafeNumbA}) = 
        f'(\bar{\alpha}, \bar{\gamma},\bar{\SafeOrdA},\SafeOrdB;\bar{\SafeNumbA})$
        for any $m \in \NatSet$.
        \item 
        For any
        $f(\bar{\alpha},\bar{\SafeOrdB},
        \SafeOrdA,\bar{\SafeOrdC};\bar{\SafeNumbA}) \in \ClassName{\OrdClass}$, there is
        $f'(\bar{\alpha}, \beta,\bar{\SafeOrdB},\bar{\SafeOrdC};\bar{\SafeNumbA}) \in \ClassName{\OrdClass}$
        such that
        $f(\bar{\alpha},\bar{\SafeOrdB},
        \SafeOrdA,\bar{\SafeOrdC};\bar{\SafeNumbA}) = 
        f'(\bar{\alpha}, o(\SafeOrdA),\bar{\SafeOrdB},\bar{\SafeOrdC};\bar{\SafeNumbA})$
        for any $\SafeOrdA \in \NatSet$.
    \end{enumerate}
\end{lemma}
\begin{proof}
For $(i)$, we prove the following stronger claim: for any
$f(\bar{\alpha},\bar{\SafeOrdA};\bar{\SafeNumbA}) \in \ClassName{\OrdClass}$, any $k \geq 1$ and any possible splitting of $\bar{\alpha}$ in segments of the form $\bar{\alpha}=\bar{\beta}_1,\bar{\gamma}_1,\ldots,\bar{\beta}_k,\bar{\gamma}_k$, there is $f'(\bar{\beta}_1,\ldots,\bar{\beta}_k,\bar{\SafeOrdA},\bar{\SafeOrdB}_1,\ldots,\bar{\SafeOrdB}_k;\bar{\SafeNumbA}) \in \ClassName{\OrdClass}$ such that 
\[
f(\bar{\beta}_1,o(\bar{\SafeOrdB}_1),\ldots,\bar{\beta}_k,o(\bar{\SafeOrdB}_k),\bar{\SafeOrdA};\bar{\SafeNumbA}) =  f'(\bar{\beta}_1,\ldots,\bar{\beta}_k,\bar{\SafeOrdA},\bar{\SafeOrdB}_1,\ldots,\bar{\SafeOrdB}_k;\bar{\SafeNumbA})
\]
for any $\bar{\SafeOrdB}_1,\ldots,\bar{\SafeOrdB}_k \in \NatSet$. For simplicity, we denote $\bar{\beta}_1,\ldots,\bar{\beta}_k$ and $\bar{\SafeOrdB}_1,\ldots,\bar{\SafeOrdB}_k$, by $\bar{\beta}$ and $\bar{m}$, respectively. In words, we claim that any set of ordinal variables can be changed to parameters such that the former on finite ordinals is the same as the latter on the corresponding numbers.

We prove this claim by an induction on the construction of $f$ as an element of $\PredFuncClass{\mathsf{A}}$. For the basic functions, the claim follows by vacuity since these functions have no ordinal inputs. If $f$ is constructed by safe composition, i.e.,
    \[
f(\bar\alpha,\bar \SafeOrdA; \bar \SafeNumbA)
		= h(\bar\alpha,\bar s(\bar \alpha, \bar \SafeOrdA;);\bar t(\bar\alpha,\bar\SafeOrdA; \bar \SafeNumbA)),
    \]
define 
    \[
    {f'}(
\bar\beta,\bar \SafeOrdA, \bar{\SafeOrdB}; \bar \SafeNumbA)
		= {h'}(\bar\beta,\bar{s}'(\bar\beta, \bar \SafeOrdA,\bar{\SafeOrdB};), \bar{\SafeOrdB};\bar{t}'(\bar\beta,\bar\SafeOrdA,\bar{\SafeOrdB}; \bar \SafeNumbA))
    \]
    where 
    $h'$,$\bar{s}'$
    and $\bar{t}'$ are obtained by the induction hypothesis.
    Then, for any $\bar{m}_1, \ldots, \bar{m}_k \in \mathbb{N}$, we have:
    \begin{align*}
&f(\bar{\beta}_1,o(\bar{\SafeOrdB}_1),\ldots,\bar{\beta}_k,o(\bar{\SafeOrdB}_k),\bar{\SafeOrdA};\bar{\SafeNumbA})\\
		&= h(\bar{\beta}_1,o(\bar{\SafeOrdB}_1),\ldots,\bar{\beta}_k,o(\bar{\SafeOrdB}_k),\bar s(\bar{\beta}_1,o(\bar{\SafeOrdB}_1),\ldots,\bar{\beta}_k,o(\bar{\SafeOrdB}_k), \bar \SafeOrdA;);\\
        &\qquad \bar t(\bar{\beta}_1,o(\bar{\SafeOrdB}_1),\ldots,\bar{\beta}_k,o(\bar{\SafeOrdB}_k),\bar\SafeOrdA; \bar \SafeNumbA))\\
	&= h'(\bar{\beta}_1,\ldots,\bar{\beta}_k,\bar {s}'(\bar{\beta}_1,\ldots,\bar{\beta}_k, \bar \SafeOrdA,\bar{\SafeOrdB};),\bar{\SafeOrdB};\bar{t}'(\bar{\beta}_1,\ldots,\bar{\beta}_k,\bar\SafeOrdA,\bar{\SafeOrdB}; \bar\SafeNumbA))\\
    &=
    {f'}(\bar{\beta}_1,\ldots,\bar{\beta}_k,\bar\SafeOrdA,\bar{\SafeOrdB}; \bar \SafeNumbA).
    \end{align*}
If $f$ is constructed by predicative ordinal recursion, we have:
    \begin{align*}
			f(0, \bar \NormalOrdB, \bar \SafeOrdA; \bar \SafeNumbA)
			&= g(\bar \NormalOrdB,
			\bar \SafeOrdA; \bar \SafeNumbA)\\
			f(\NormalOrdA+1,\bar\NormalOrdB, \bar\SafeOrdA;\bar \SafeNumbA) &= h_{\mathsf{suc}}(\NormalOrdA, \bar\NormalOrdB,\bar\SafeOrdA; 
			f(\NormalOrdA,\bar\NormalOrdB, \bar\SafeOrdA;\bar \SafeNumbA),\bar \SafeNumbA)\\
			f(\langle \NormalOrdA_i \rangle_i, \bar\NormalOrdB,\bar\SafeOrdA;\bar \SafeNumbA)
			&= h_{\mathsf{lim}}(\langle \NormalOrdA_i \rangle_i, \bar\NormalOrdB, \bar\SafeOrdA;
			f(\NormalOrdA_{q(\bar \SafeOrdA;)}, \bar\NormalOrdB,\bar\SafeOrdA; \bar \SafeNumbA),
			\bar \SafeNumbA
			)
		\end{align*}
As $\bar{\beta}_1,\bar{\gamma}_1,\ldots,\bar{\beta}_k,\bar{\gamma}_k$ is a splitting of $\mu, \bar{\nu}$, there are two cases to consider: either $\mu \in \bar{\beta}_1$ or $\bar{\beta_1}$ is the empty list and hence $\mu \in \bar\gamma_1$. In the first case, assume $\bar{\beta}_1=\mu, \bar{\beta}_1'$ and consider the splitting $\bar{\beta}'_1,\bar{\gamma}_1,\ldots,\bar{\beta}_k,\bar{\gamma}_k$ of $\bar{\nu}$ and denote $\bar{\beta}'_1, \bar{\beta}_2 \ldots, \bar{\beta_k}$ by $\bar{\beta}'$. 
That is, $\bar{\beta}'$
is obtained from $\bar{\beta}$ by removing $\mu$ from the first segment $\bar{\beta}_1$, i.e., by replacing $\bar{\beta}_1$ with $\bar{\beta'_1}$. Note that in this case $\mu$ will not be converted into a numerical variable.
Define $f'$ by recursion in the following way:
    \begin{align*}
			f'(0, \bar{\beta}', \bar \SafeOrdA,\bar{\SafeOrdB}; \bar \SafeNumbA)
			&= 
            g'(\bar{\beta}', \bar\SafeOrdA,\bar{\SafeOrdB}; \bar \SafeNumbA) \\
			f'(\NormalOrdA+1,\bar{\beta}', \bar\SafeOrdA,\bar{\SafeOrdB};\bar \SafeNumbA) &=
            h'_{\mathsf{suc}}(\NormalOrdA, \bar{\beta}',\bar\SafeOrdA,\bar{\SafeOrdB}; 
			f'(\NormalOrdA,\bar{\beta}', \bar\SafeOrdA,\bar{\SafeOrdB};\bar \SafeNumbA),\bar \SafeNumbA)
            \\
			f'(\langle \NormalOrdA_i \rangle_i, \bar{\beta}',\bar\SafeOrdA,\bar{\SafeOrdB};\bar \SafeNumbA)
			&= 
			h'_{\mathsf{lim}}(\langle\NormalOrdA_i\rangle_i, \bar{\beta}',\bar\SafeOrdA,\bar{\SafeOrdB}; 
			f'(\NormalOrdA_{q(\bar n;)},\bar{\beta}', \bar\SafeOrdA,\bar{\SafeOrdB};\bar \SafeNumbA),\bar \SafeNumbA)
		\end{align*}
        where $g'$ is (resp. $h'_{\mathsf{suc}}$
    and $h'_{\mathsf{lim}}$ are) the corresponding functions provided by the induction hypothesis, using the induced (resp. original) splitting of $\bar{\nu}$ (resp. $\mu, \bar{\nu}$). By induction on $\mu$, it is not hard to see that 
  {\small\[
f(\mu,\bar{\beta}'_1,o(\bar{\SafeOrdB}_1),\bar{\beta}_2,o(\bar{\SafeOrdB}_2),\ldots,\bar{\beta}_k,o(\bar{\SafeOrdB}_k),\bar{\SafeOrdA};\bar{\SafeNumbA}) =  f'(\mu, \bar{\beta}'_1,\ldots,\bar{\beta}_k,\bar{\SafeOrdA},\bar{\SafeOrdB}_1,\ldots,\bar{\SafeOrdB}_k;\bar{\SafeNumbA})
\]}
for any $\bar{\SafeOrdB}_1,\ldots,\bar{\SafeOrdB}_k \in \NatSet$. 

    For the second case, $\bar{\beta}_1$ is an empty list and $\mu \in \bar{\gamma}_1$.
    Let $\bar{\gamma}_1=\mu, \bar{\gamma}_1'$ and set $\bar{m}_1=p,\bar{m}_1'$. Therefore, the variable $p$ is there to simulate $\mu$. To define the function $f'$, we proceed in two stages. In the first stage, we simulate $f$ on all variables in $\bar{\gamma}_1, \ldots, \bar{\gamma}_k$, except for $\mu$. For the variable $\mu$, we only ensure that the resulting function $f''$ correctly imitates $f$ on finite ordinals of the form $o(l)$, where $l \in \mathbb{N}$. This yields an intermediate function, denoted by $f''$. In the second stage, to obtain $f'$ from $f''$, it suffices to get the ordinal $o(p)$ from the input $p$ and use $\mu=o(p)$ in $f''$. We now begin by implementing the first stage.
    
    Consider the splitting $\mu, \bar{\gamma}'_1,\bar{\beta}_2,\bar{\gamma}_2,\ldots,\bar{\beta}_k,\bar{\gamma}_k$ of $\mu, \bar{\nu}$. Note that the only difference between this splitting and the original one is that we moved out $\mu$ from $\bar{\gamma}_1$ to be one of the ordinals that we keep rather than change to numbers. 
    Let $\bar{m}' = \bar{m}'_1,\bar{m}_2,\ldots,\bar{m}_k$.
    Now, consider $g'(\bar{\beta}, \bar{n},\bar{m}';\bar{a})$ (resp. $h'_{\mathsf{suc}}(\mu,\bar{\beta}, \bar{n},\bar{m}';\bar{a})$) as the corresponding functions provided by the induction hypothesis, using the new splitting of $\bar{\nu}$ (resp. $\mu, \bar{\nu}$). Then, define $g''(\bar{\beta}, \bar{n},p, \bar{m}_1', \ldots, \bar{m}_k;\bar{a})=g'(\bar{\beta}, \bar{n},\bar{m}';\bar{a})$ and $h''_{\mathsf{suc}}(\mu,\bar{\beta}, \bar{n}, p, \bar{m}_1', \ldots, \bar{m}_k;\bar{a})=h'_{\mathsf{suc}}(\mu,\bar{\beta}, \bar{n},\bar{m}';\bar{a})$ from $g'$ and $h'_{\mathsf{suc}}$ by adding the dummy normal numeral variable $p$ which is possible by safe composition. As $\bar{m}_1=p, \bar{m}_1'$, for simplicity, we can denote the new functions by $g''(\bar{\beta}, \bar{n},\bar{m};\bar{a})$ and $h''_{\mathsf{suc}}(\mu,\bar{\beta}, \bar{n},\bar{m};\bar{a})$. 
    Then define:
    \begin{align*}
			f''(0, \bar \beta, \bar \SafeOrdA,\bar{\SafeOrdB}; \bar \SafeNumbA)
			&= 
            g''(\bar \beta,
			\bar\SafeOrdA,\bar{\SafeOrdB}; \bar \SafeNumbA) \\
			f''(\NormalOrdA+1,\bar\beta, \bar\SafeOrdA,\bar{\SafeOrdB};\bar \SafeNumbA) &=
            h''_{\mathsf{suc}}(\NormalOrdA, \bar\beta,\bar\SafeOrdA,\bar{\SafeOrdB}; 
			f''(\NormalOrdA,\bar\beta, \bar\SafeOrdA,\bar{\SafeOrdB};\bar \SafeNumbA),\bar \SafeNumbA)
            \\
			f''(\langle \NormalOrdA_i \rangle_i, \bar\beta,\bar\SafeOrdA,\bar{\SafeOrdB};\bar \SafeNumbA)
			&= 
			f''(\NormalOrdA_{q(\bar n, p, \bar{m}'_1, \ldots,\bar{m}_k;)}, \bar\beta,\bar\SafeOrdA,\bar{\SafeOrdB}; \bar \SafeNumbA)
		\end{align*}
In limit case above, we used the function $q(\bar n,p, \bar{m}'_1,\ldots,\bar{m}_k;) = \Pred^{r+1}(;p)$, the $(r+1)$-fold composition of $\Pred{}$, clearly definable by safe composition using 
$\Pred{}$ and
a projection. Recall that $r \in \mathbb{N}$ is a fixed natural number with the property that $o(r)+\omega \in \mathsf{A}$. We will make use of the limit case clause in the definition of $f''$ later, during the second stage of the construction of $f'$.

We now show that, for any $l \in \NatSet$, we have:
$$f''(o(l),\bar{\beta}, \bar \SafeOrdA,\bar{\SafeOrdB}; \bar \SafeNumbA)=
f(o(l),o(\bar{\SafeOrdB}_1'),
\bar{\beta}_2,
o(\bar{m}_2),
\ldots,\bar{\beta}_k,o(\bar{m}_k),\bar{\SafeOrdA};\bar{\SafeNumbA}). \qquad (*)$$
We use an induction on $l$. For $l=0$, by the induction hypothesis, for any $\bar{\SafeOrdB}_1,\ldots,\bar{\SafeOrdB}_k \in \NatSet$, we have:
\begin{align*}
f''(o(0),\bar \beta, \bar \SafeOrdA,\bar{\SafeOrdB}; \bar \SafeNumbA)
&=
g''(\bar \beta,
			\bar\SafeOrdA,\bar{\SafeOrdB}; \bar \SafeNumbA)\\
&=
g'(\bar \beta, \bar\SafeOrdA,\bar{\SafeOrdB}'_1,\ldots,\bar{\SafeOrdB}_k; \bar \SafeNumbA)\\
&=
g(o(\bar{\SafeOrdB}'_1),\bar{\beta}_2, \ldots,\bar{\beta}_k,o(\bar{\SafeOrdB}_k),\bar{\SafeOrdA};\bar{\SafeNumbA})\\
&=f(o(0),o(\bar{\SafeOrdB}'_1),\bar{\beta}_2, \ldots,\bar{\beta}_k,o(\bar{\SafeOrdB}_k),\bar{\SafeOrdA};\bar{\SafeNumbA})
\end{align*}
For the inductive step, by the induction hypothesis, for any $\bar{\SafeOrdB}_1,\ldots,\bar{\SafeOrdB}_k \in \NatSet$, we obtain: 
\begin{align*}   
&f''(o(l+1),\bar \beta, \bar \SafeOrdA,\bar{\SafeOrdB}; \bar \SafeNumbA)\\
&=
h''_{\mathsf{suc}}(o(l), \bar\beta,\bar\SafeOrdA,\bar{\SafeOrdB}; 
			f''(o(l),\bar\beta, \bar\SafeOrdA,\bar{\SafeOrdB};\bar \SafeNumbA),\bar \SafeNumbA)\\
&=h_{\mathsf{suc}}(o(l),o(\bar{\SafeOrdB}'_1),\ldots,\bar{\beta}_k,o(\bar{\SafeOrdB}_k),\bar{\SafeOrdA}; 
f(o(l),o(\bar{\SafeOrdB}'_1),\ldots,\bar{\beta}_k,o(\bar{\SafeOrdB}_k),\bar{\SafeOrdA};\bar{\SafeNumbA}),\bar \SafeNumbA)\\
&=f(o(l+1),o(\bar{\SafeOrdB}'_1),\ldots,\bar{\beta}_k,o(\bar{\SafeOrdB}_k),\bar{\SafeOrdA};\bar{\SafeNumbA}).
\end{align*}
This completes the proof of $(*)$. 

Now, for the second stage, we first explain our strategy in a rough manner. To define $f'$ from $f''$, as previously discussed, we need to get the ordinal $o(p)$ from the input $p$ and feed $\mu=o(p)$ to $f''$. There are $r+2$ distinct cases to handle: either $p > r$ or $p \in \{0, \ldots, r\}$. 
For the case $p > r$, we evaluate $f''$ at $\mu = o(r) + \omega$. Then, by the definition of $f''$ in the limit case and by the choice of $q$—which has exactly $r+1$ applications of the predecessor function—we ultimately reach the computation of $f''$ at $\mu = o(p)$.
For the remaining $r+1$ cases, where $p \in \{0, \ldots, r\}$, we simply substitute the constant ordinal $o(p)$ into $f''$ using constant substitution.
Finally, to handle the case distinction across these $r+2$ branches, we employ safe composition with the conditional function $C_r$, as introduced in Remark~\ref{RemOnCr}. The details of this construction are explained below.

First, observe that $o(r) + \omega \in \mathsf{A}$. Since $\mathsf{A}$ is a downset and $o(j) \prec o(r) + \omega$ for any $j \in \mathbb{N}$, it follows that $o(j) \in \mathsf{A}$ for all $j \in \mathbb{N}$. Using this fact, and applying constant substitution for $\mu$ in $f''$, we can construct the function $f''(o(j), \bar{\beta}, \bar{\SafeOrdA}, \bar{\SafeOrdB}; \bar{\SafeNumbA})$ in $\mathcal{C}_{\mathsf{A}}$ for any $j \leq r$. By the same reasoning, we can also construct $f''(o(r) + \omega, \bar{\beta}, \bar{\SafeOrdA}, \bar{\SafeOrdB}; \bar{\SafeNumbA})$ in $\mathcal{C}_{\mathsf{A}}$. Now, using safe composition with $C_r$, define:
\[
f'(\bar \beta, \bar \SafeOrdA,\bar{\SafeOrdB}; \bar \SafeNumbA)=\begin{cases} 
    f''(o(0), \bar \beta,
			\bar\SafeOrdA,\bar{\SafeOrdB}; \bar \SafeNumbA) & \text{if } p = 0 \\ 
    f''(o(1), \bar \beta,
			\bar\SafeOrdA,\bar{\SafeOrdB}; \bar \SafeNumbA) & \text{if } p = 1 \\        
            \qquad\qquad\vdots&\;\;\quad\vdots\\
    f''(o(r), \bar \beta,
			\bar\SafeOrdA,\bar{\SafeOrdB}; \bar \SafeNumbA) & \text{if } p = r \\ 
    f''(o(r)+\omega,\bar \beta, \bar \SafeOrdA,\bar{\SafeOrdB}; \bar \SafeNumbA) & \text{if } p > r
\end{cases} 
\]
Note that the variable $p$ is present on both sides, as it appears as a component of the tuple $\bar{m}$. For $p \leq r$, using $(*)$ for $l \leq r$, we reach: 
\[
f'(\bar \beta, \bar \SafeOrdA,\bar{\SafeOrdB}; \bar \SafeNumbA)=f(o(p), o(\bar{\SafeOrdB}'_1), \bar{\beta}_2, \ldots,\bar{\beta}_k,o(\bar{\SafeOrdB}_k),\bar{\SafeOrdA};\bar{\SafeNumbA}).
\]
For $p > r$, using $(*)$ and the definition of $f''$ in the limit case, we have:
\begin{align*}
f'(\bar \beta, \bar \SafeOrdA, p, \bar{\SafeOrdB}'_1,\ldots,\bar{\SafeOrdB}_k; \bar \SafeNumbA)
&= f''(o(r)+\omega,\bar \beta, \bar \SafeOrdA,p, \bar{\SafeOrdB}'_1,\ldots,\bar{\SafeOrdB}_k; \bar \SafeNumbA)\\
&=f''(o(r + (p - (r+1))+1),\bar{\beta}, \bar \SafeOrdA,\bar{\SafeOrdB}; \bar \SafeNumbA)\\
&=f''(o(p),\bar{\beta}, \bar \SafeOrdA,\bar{\SafeOrdB}; \bar \SafeNumbA)\\
&=f(o(p),o(\bar{\SafeOrdB}'_1), \bar{\beta}_2, \ldots,\bar{\beta}_k,o(\bar{\SafeOrdB}_k),\bar{\SafeOrdA};\bar{\SafeNumbA})\\
&=f(o(\bar{\SafeOrdB}_1), \bar{\beta}_2, \ldots,\bar{\beta}_k,o(\bar{\SafeOrdB}_k),\bar{\SafeOrdA};\bar{\SafeNumbA}).
\end{align*}
If $f$ is obtained by a constant substitution, thus replacing the variable $\mu$ in the function $g$ by an ordinal $\alpha \in \mathsf{A}$, then we first consider the given splitting for the variables of $f$ and add $\mu$ to  $\beta$'s to get a splitting for the variables of $g$. Then, it is enough to use the induction hypothesis to get $g'$ and then substitute $\mu$ with $\alpha$. The resulting function clearly satisfies the claim.

Finally, to address the cases where $f$ is obtained by a structural rule, first recall from Remark \ref{rem:pred-rec-ord-funcs} that the similar rules are admissible for normal numeral variables. It is just enough to use composition with suitable projections. Now, consider the exchange rule. Let $\mu$ and $\nu$ be the variables being exchanged. If only one of them is transferred to a numeral variable, the claim is straightforward. If both are transferred, then by the above observation, the result follows directly from the induction hypothesis together with the exchange rule for numeral normal variables. The same reasoning applies to the other two structural rules.

For $(ii)$, using the more general safe composition discussed in {Remark~\ref{rem:pred-rec-ord-funcs}}, define:
\[
f'(\bar{\alpha}, \beta,\bar{\SafeOrdB},\bar{\SafeOrdC};\bar{\SafeNumbA})= f(\bar{\alpha}, \bar{\SafeOrdB}, G(\beta,0;),\bar{\SafeOrdC};\bar{\SafeNumbA}).
\]
Then, as $G(o(n), 0;)=G_{o(n)}(0)=n$, for any $n \in \mathbb{N}$, we reach:
\[
f'(\bar{\alpha}, o(n),\bar{\SafeOrdB},\bar{\SafeOrdC};\bar{\SafeNumbA})= f(\bar{\alpha}, \bar{\SafeOrdB}, G(o(n),0;),\bar{\SafeOrdC};\bar{\SafeNumbA})=f(\bar{\alpha}, \bar{\SafeOrdB}, n,\bar{\SafeOrdC};\bar{\SafeNumbA}).
\qedhere
\]
\end{proof}

\begin{remark}
Here are two remarks regarding the condition $o(r) + \omega \in \mathsf{A}$ in Lemma \ref{lem:ord-to-nat-and-back}. First, one might think that $o(r) + \omega = \omega$, and thus the assumption that there exists some $r \in \mathbb{N}$ with $o(r) + \omega \in \mathsf{A}$ is equivalent to assuming $\omega \in \mathsf{A}$. However, as discussed earlier, when $r>0$, in our constructive setting the ordinals 
$
\omega =\langle o(1), o(2), \ldots \rangle
$
and
$
o(r) + \omega = \langle o(r+1), o(r+2), \ldots \rangle
$
are distinct, and in fact incomparable with respect to $\preceq$. Therefore, having $o(r) + \omega \in \mathsf{A}$ for some $r > 0$ does not imply that $\omega \in \mathsf{A}$.
Second, since our assumption is weaker than $\omega \in \mathsf{A}$, one might think it is more natural to use the stronger condition $\omega \in \mathsf{A}$ when implementing the simulation. However, as we will see later, the weaker condition plays a crucial role in providing a characterization of $\PredFuncClass{\mathsf{A}}$ for a more natural family of downsets $\mathsf{A}$. This characterization is tight, thereby motivating the use of the weaker condition.
\end{remark}

We are now ready to show that predicative ordinal recursion simulates predicative recursion (recall Section~\ref{sec:pred-rec-func}).
    
\begin{lemma}
\label{lem:agree-finite-ordinals}
Let $ \OrdClass \subseteq \Omega$ be a downset of ordinals such that 
$o(r)+\omega \in \OrdClass$, for some $r \in \NatSet$. Then, for any $f(\bar n;\bar \SafeNumbA) \in B$, there is a function
    $\hat f(\bar n;\bar\SafeNumbA) \in \ClassName{\OrdClass}$
    satisfying
    $f(\bar n;\bar a) = \hat f (\bar{n};\bar{\SafeNumbA})$, for any
    $\bar{n}, \bar{\SafeNumbA} \in \mathbb{N}$.
\end{lemma}
\begin{proof}
We prove the claim by an induction on the construction of $f(\bar n;\bar a)$ as an element of the class $B$. 
The base cases concerning the zero, successor, predecessor, conditional and projection functions are easy, as we can use the corresponding basic function in $\ClassName{\OrdClass}$. 

For the induction step, there are two cases to consider. If
$f$ is constructed by safe composition, the claim is a direct consequence of the induction hypothesis. If $f(m, \bar n; \bar a)$ is constructed by predicative recursion:
\begin{align*}
             f(0, \bar n; \bar a) &= g(\bar n; \bar a)\\
            f(m+1,\bar n; \bar a) &= h(m,\bar n;\bar a,f(m,\bar n;\bar a))
        \end{align*}
then by the induction hypothesis, there are $\hat g(\bar n;\bar a),\hat h(m, \bar n;\bar{a}, b) \allowbreak\in \ClassName{\OrdClass}$ such that $g(\bar n;\bar a) = \hat g (\bar{n};\bar{a})$ and $h(m, \bar n;\bar a, b) = \hat h (m, \bar{n};\bar{a}, b)$, for any $m, \bar{n}, \bar{a}, b \in \mathbb{N}$. Let $h'(\NormalOrdA, \bar n;\bar a, b)$ be the function obtained from $\hat h$ by making the numeral variable $m$ into the ordinal variable $\mu$ using Lemma~\ref{lem:ord-to-nat-and-back}, part $(ii)$. Therefore, $h'(o(m), \bar{n};\bar{a}, b)=\hat{h}(m, \bar{n};\bar{a}, b)$, for any $m, \bar{n}, \bar{a}, b \in \mathbb{N}$.
In $\ClassName{\OrdClass}$, define $f'$ by predicative ordinal recursion as:
\begin{align*}
			f'(0, \bar n; \bar \SafeNumbA)
			&= \hat{g}(\bar n; \bar \SafeNumbA)\\
			f'(\NormalOrdA+1, \bar n; \bar \SafeNumbA) &= h'(\NormalOrdA, \bar n; \bar a, 
			f'(\NormalOrdA, \bar n; \bar \SafeNumbA))\\
			f'(\langle \NormalOrdA_i \rangle_i,\bar n;\bar \SafeNumbA)
			&= 0
		\end{align*}
By induction on $l$, it is clear that $f'(o(l),\bar n; \bar a)=f(l,\bar n; \bar a)$, for any $l, \bar{n}, \bar{a} \in \mathbb{N}$.
Now let $\hat f$ be the function obtained from $f'$ by making the ordinal variable $\mu$ into the numeral variable $m$ using Lemma~\ref{lem:ord-to-nat-and-back}, part $(i)$. Thus, $\hat{f}(m, \bar{n};\bar{a})=f'(o(m), \bar{n};\bar{a})$, for any $m, \bar{n}, \bar{a} \in \mathbb{N}$. Therefore, we finally get $\hat{f}(m, \bar{n};\bar{a})=f(m, \bar{n};\bar{a})$, for any $m, \bar{n}, \bar{a} \in \mathbb{N}$. 
\end{proof}

\begin{corollary}\label{cor:E2-subset-predr}
Let $ \OrdClass \subseteq \Omega$ be a downset of ordinals such that $o(r)+\omega \in \OrdClass$,
for some $r \in \NatSet$. Then,  $\GrzClass{2} \subseteq \PredFuncClass{\OrdClass}$.
\end{corollary}
\begin{proof}
It is enough to use Theorem~\ref{thm:cb} and Lemma~\ref{lem:agree-finite-ordinals}.
\end{proof}

\subsection{Monotonicity and continuity of $\PredFuncClass{\OrdClass}$}

In this subsection, we prove that $\PredFuncClass{\OrdClass}$ is both \emph{monotone} and \emph{continuous} with respect to $\OrdClass$. The precise meaning of these terms will become clear in the course of the discussion.

\begin{lemma}[Monotonicity]
\label{MonotonicityOfPredR}
Let $\mathsf{A}$ and $\mathsf{B}$ be two downsets of ordinals such that $\mathsf{B} \subseteq \mathsf{A}$. Then, for every
$f(\bar{\NormalOrdA},\bar{\SafeOrdA};\bar{\SafeNumbA}) \in \ClassName{\mathsf{B}}$,
there is
$f'(\bar{\NormalOrdA},\bar{\SafeOrdA};\bar{\SafeNumbA}) \in \ClassName{\mathsf{A}}$
such that
$f(\bar{\NormalOrdA},\bar{\SafeOrdA};\bar{\SafeNumbA}) = f'(\bar{\NormalOrdA},\bar{\SafeOrdA};\bar{\SafeNumbA})$
for any $\bar{\NormalOrdA} \in \mathsf{B}$ and
any $\bar\SafeOrdA,\bar\SafeNumbA \in \NatSet$. 
Consequently, $\PredFuncClass{\mathsf{B}} \subseteq \PredFuncClass{\mathsf{A}}$.
\end{lemma}
\begin{proof}
The proof is straightforward by induction on the
construction of $f$ as an element of $\mathcal{C}_{\mathsf{B}}$. The only non-trivial case is when $f$ is constructed using constant substitution, where the constant lies in $\mathsf{B}$ and hence also in $\mathsf{A}$. Therefore, we can apply the same constant substitution within $\mathcal{C}_{\mathsf{A}}$.
\end{proof}

\begin{lemma}[Continuity]
\label{lem:PredInLimit}
Let $\{\OrdClass_i\}_{i \in \mathbb{N}} \subseteq \ConsOrd$ be a family of downsets of ordinals such that $\mathsf{A}_i \subseteq \mathsf{A}_{i+1}$, for any $i \in \mathbb{N}$. Then
$ \PredFuncClass{\bigcup_{i \in \mathbb{N}} \OrdClass_i} = \bigcup_{i \in \NatSet} \PredFuncClass{\OrdClass_i}$.
\end{lemma}
\begin{proof}
As $\OrdClass_i \subseteq \bigcup_{i \in \mathbb{N}}\OrdClass_i$, by Lemma~\ref{MonotonicityOfPredR}, we obtain $\PredFuncClass{\mathsf{A}_i} \subseteq \PredFuncClass{\bigcup_{i \in \mathbb{N}}\OrdClass_i}$. Hence, $\bigcup_{i \in \NatSet} \PredFuncClass{\OrdClass_i} \allowbreak\subseteq \PredFuncClass{\bigcup_{i \in \mathbb{N}}\OrdClass_i}$. For the converse, i.e.,  $\PredFuncClass{\bigcup_{i \in \mathbb{N}}\OrdClass_i} \subseteq \bigcup_{i \in \NatSet} \PredFuncClass{\OrdClass_i}$, we first prove the following stronger claim: \\

\noindent \textbf{Claim.} For every $f(\bar{\NormalOrdA},\bar{\SafeOrdA};\bar{\SafeNumbA}) \in \ClassName{\bigcup_{i \in \mathbb{N}} \OrdClass_i}$,
there is $i \in \NatSet$ such that for any $j \geq i$, there is
$f'(\bar{\NormalOrdA},\bar{\SafeOrdA};\bar{\SafeNumbA}) \in \ClassName{\mathsf{A}_j}$ satisfying
$f(\bar{\NormalOrdA},\bar{\SafeOrdA};\bar{\SafeNumbA}) = f'(\bar{\NormalOrdA},\bar{\SafeOrdA};\bar{\SafeNumbA})$
for any $\bar{\NormalOrdA} \in \mathsf{A}_j$
and any $\bar\SafeOrdA,\bar\SafeNumbA \in \NatSet$.\\

\noindent 
The intuition behind the claim is that, although the function~$f$ is defined on the entire set $\bigcup_{i \in \mathbb{N}} \OrdClass_i$, its construction relies only on a finite number of ordinals, all introduced via constant substitution. Consequently, for sufficiently large~$j$, the downset~$\mathsf{A}_j$ will contain all these finitely many ordinals. Therefore, the function~$f'$ can be constructed in the same way as~$f$, and it will agree with~$f$ on~$\OrdClass_j$.

To prove the claim, we use an induction on the construction of $f$. For the base case,
if $f$ is an initial function, since it has no ordinal inputs, it is enough to take $i=0$ and $f'=f$. If $f$ is constructed by the safe composition:
    $$
f(\bar\NormalOrdA,\bar \SafeOrdA; \bar \SafeNumbA)
		= h(\bar\NormalOrdA,\bar s(\bar\NormalOrdA, \bar \SafeOrdA;);\bar t(\bar\NormalOrdA,\bar\SafeOrdA; \bar \SafeNumbA)),
    $$
apply the induction hypothesis to $h,\bar s$ and $\bar t$, to obtain one index for each of these functions. Then, set $i$ to be the largest among them. Let $j \geq i$ and take $h',\bar{s}',\bar{t}' \in \ClassName{\mathsf{A}_j}$ that agree with $h,\bar s$ and
$\bar t$ on the ordinals in $\mathsf{A}_j$.
It is then enough to set $f'(\bar\NormalOrdA,\bar \SafeOrdA; \bar \SafeNumbA)	= h'(\bar\NormalOrdA,\bar {s}'(\bar\NormalOrdA, \bar \SafeOrdA;);\bar {t}'(\bar\NormalOrdA,\bar\SafeOrdA; \bar \SafeNumbA))$.

If $f$ is constructed by predicative ordinal recursion:
    \begin{align*}
			f(0, \bar \NormalOrdB, \bar \SafeOrdA; \bar \SafeNumbA)
			&= g(\bar \NormalOrdB,
			\bar\SafeOrdA; \bar \SafeNumbA)\\
			f(\NormalOrdA+1,\bar\NormalOrdB, \bar\SafeOrdA;\bar \SafeNumbA) &= h_{\mathsf{suc}}(\NormalOrdA, \bar\NormalOrdB,\bar\SafeOrdA; 
			f(\NormalOrdA,\bar\NormalOrdB, \bar\SafeOrdA;\bar \SafeNumbA),\bar \SafeNumbA)\\
			f(\langle \NormalOrdA_i \rangle_i, \bar\NormalOrdB,\bar\SafeOrdA;\bar \SafeNumbA)
			&= h_{\mathsf{lim}}(\langle \NormalOrdA_i \rangle_i, \bar\NormalOrdB, \bar\SafeOrdA;
			f(\NormalOrdA_{q(\bar \SafeOrdA;)}, \bar\NormalOrdB,\bar\SafeOrdA; \bar \SafeNumbA),
			\bar \SafeNumbA
			)
		\end{align*}
apply the induction hypothesis to $g$, $h_{\mathsf{suc}}$, and $h_{\mathsf{lim}}$ to obtain indices for each of them, and set $i$ as the largest of those indices. Note that $q$ does not require the induction hypothesis, as it has no ordinal input. Let $j \geq i$ and take $g', h'_{\mathsf{suc}}, h'_{\mathsf{lim}} \in \ClassName{\mathsf{A}_j}$
that agree with $g$, $h_{\mathsf{suc}}$, and $h_{\mathsf{lim}}$ on the ordinals in $\mathsf{A}_j$. Then, it is enough to
define $f'$ as an element of $\mathcal{C}_{\mathsf{A}_j}$ by:
    \begin{align*}
			f'(0, \bar \NormalOrdB, \bar \SafeOrdA; \bar \SafeNumbA)
			&= g'(\bar \NormalOrdB,
			\bar\SafeOrdA; \bar \SafeNumbA)\\
			f'(\NormalOrdA+1,\bar\NormalOrdB, \bar\SafeOrdA;\bar \SafeNumbA) &= h'_{\mathsf{suc}}(\NormalOrdA, \bar\NormalOrdB,\bar\SafeOrdA; 
			f'(\NormalOrdA,\bar\NormalOrdB, \bar\SafeOrdA;\bar \SafeNumbA),\bar \SafeNumbA)\\
			f'(\langle \NormalOrdA_i \rangle_i, \bar\NormalOrdB,\bar\SafeOrdA;\bar \SafeNumbA)
			&= h'_{\mathsf{lim}}(\langle \NormalOrdA_i \rangle_i, \bar\NormalOrdB, \bar\SafeOrdA;
			f'(\NormalOrdA_{q(\bar \SafeOrdA;)}, \bar\NormalOrdB,\bar\SafeOrdA; \bar \SafeNumbA),
			\bar \SafeNumbA
			)
		\end{align*}
Note that the ordinal inputs of $f'$ are taken from $\mathsf{A}_j$, and it is clear that $f$ and $f'$ agree on these inputs.

If $f$ is obtained by constant substitution:
\[
f(\NormalOrdA_1,\ldots,\NormalOrdA_{i-1},\NormalOrdA_{i+1},\ldots,\NormalOrdA_{k}, \bar \SafeOrdA; \bar \SafeNumbA)
        =
        g(\NormalOrdA_1,\ldots,\NormalOrdA_{i-1},\alpha,\NormalOrdA_{i+1},\ldots,\NormalOrdA_{k},\bar\SafeOrdA;\bar\SafeNumbA)
\]
where $\alpha \in \bigcup_{i \in \mathbb{N}} \mathsf{A}_i$, pick $p \in \NatSet$ such that $\alpha \in \mathsf{A}_{p}$ and apply the induction hypothesis to $g$ to obtain the index $i_g$. Set $i$ as the maximum of $i_g$ and $p$. Then, given $j \geq i$,
let $g' \in \ClassName{\mathsf{A}_j}$ be the function
that agrees with $g$ on the ordinals in $\mathsf{A}_j$.
Now, it is enough to use constant substitution to define:
\[
f'(\NormalOrdA_1,\ldots,\NormalOrdA_{i-1},\NormalOrdA_{i+1},\ldots,\NormalOrdA_{k}, \bar \SafeOrdA; \bar \SafeNumbA)
        =
        g'(\NormalOrdA_1,\ldots,\NormalOrdA_{i-1},\alpha,\NormalOrdA_{i+1},\ldots,\NormalOrdA_{k},\bar\SafeOrdA;\bar\SafeNumbA)
\]
which is possible, because $\alpha \in \mathsf{A}_p \subseteq \mathsf{A}_i \subseteq \mathsf{A}_j$.

Finally, the case when $f$ is defined by the structural rules is clear, by following essentially the same strategy as above. This completes the proof of the claim.

To prove the lemma, if $f(\bar{\SafeOrdA}) \in \PredFuncClass{\bigcup_{i \in \mathbb{N}} \OrdClass_i}$, then there is $g(\bar{n};)$ such that $g \in \ClassName{\bigcup_{i \in \mathbb{N}} \OrdClass_i}$ and $f(\bar{n})=g(\bar{n};)$, for any $\bar{n} \in \mathbb{N}$. Use the claim for $g$ to find $i$ and pick $j=i$. Then, there is $g'(\bar{\SafeOrdA};) \in \ClassName{\OrdClass_j}$ such that $g(\bar{\SafeOrdA};)=g'(\bar{\SafeOrdA};)$, for any $\bar{\SafeOrdA},\bar{\SafeNumbA} \in \mathbb{N}$. As $g'\in \mathcal{C}_{\OrdClass_j}$, we get $f \in \PredFuncClass{\OrdClass_j}$.
\end{proof}

\section{Constructive Veblen Hierarchy}
\label{sec:const-veblen}
In this section, we introduce the constructive Veblen hierarchy of ordinal functions both for constructive ordinals and for set-theoretic ordinals, and we show how the former can represent the latter. We then proceed to define the set $\Phi_{\omega}$ of ordinals (Definition~\ref{Phi-sets}) as the closure of the set $\{0\}$ under addition and the constructive Veblen functions. We prove the uniqueness of such constructions (Theorem~\ref{UniquenessTheorem}) and show a dichotomy theorem stating that for any $\alpha \in \Phi_{\omega}$, there exists $r \in \mathbb{N}$ such that either $\alpha = o(r)$ or $o(r) + \omega \preceq \alpha$ (Theorem~\ref{FiniteOrInfiniteThm}). We also introduce some interesting subsets of $\Phi_{\omega}$ and establish some key properties of these classes.

\begin{definition}[Constructive Veblen hierarchy]\label{dfn:VeblenHierarchy}
	The \emph{Constructive Veblen hierarchy} $\{ \Veb{k}: \Omega \to \Omega \}_{k \in \mathbb{N}}$ is a sequence of ordinal functions recursively defined as follows:
\begin{align*}
&\Veb{0}(\gamma) = \omega^{\gamma}, &&
\Veb{k+1}(0) = \langle \Veb{k}^{(i)}(0) \rangle_i, \\
&\Veb{k+1}(\gamma + 1) = \langle \Veb{k}^{(i)}(\Veb{k+1}(\gamma) + 1) \rangle_i, &&
\Veb{k+1}(\langle \gamma_i \rangle_i) = \langle \Veb{k+1}(\gamma_i) \rangle_i.
\end{align*}
\end{definition}
\begin{example}\label{ex-Veblen}
It is clear that $\Veb{0}(0) = \omega^0 = 1$. For another example, consider:
\[
\Veb{1}(0) = \langle \Veb{0}^{(i)}(0) \rangle_i = \langle 0, \Veb{0}(0),\allowbreak \Veb{0}^2(0),\allowbreak \Veb{0}^3(0), \ldots \rangle = \langle 0, 1, \omega, \omega^\omega, \omega^{\omega^{\omega}}, \ldots \rangle.
\]
We denote the ordinal $\Veb{1}(0)$ by~$\epsilon_0$.
\end{example}

To motivate the constructive Veblen hierarchy, it is helpful to consider its counterpart over the class~$\mathbf{On}$ of set-theoretic ordinals. First, recall that a monotone class function $f: \mathbf{On} \to \mathbf{On}$ is called \emph{continuous} if, for any limit ordinal $\bm{\alpha}$, we have
$f(\bm{\alpha}) = \bigcup_{\bm{\beta} \subset \bm{\alpha}} f(\bm{\beta})$.
It is known (see \cite{veblen1908}) that any
strictly monotone and
continuous class function $f: \mathbf{On} \to \mathbf{On}$ has arbitrarily large fixed points.
We now define the hierarchy $\{\bm{\phi}_k: \mathbf{On} \to \mathbf{On}\}_{k \in \mathbb{N}}$ of class functions as follows:
\begin{align*}
\bm{\phi}_0(\bm{\alpha}) &= \bm{\omega}^{\bm{\alpha}}, \\[1ex]
\bm{\phi}_{k+1}(\bm{0}) &= \text{the least fixed point of } \bm{\phi}_k, \\[1ex]
\bm{\phi}_{k+1}(\bm{\alpha+1}) &= \text{the least fixed point of } \bm{\phi}_k \text{ strictly greater than } \bm{\phi}_{k+1}(\bm{\alpha}), \\[1ex]
\bm{\phi}_{k+1}(\bm{\alpha}) &= \displaystyle\bigcup_{\bm{\beta} \subset \bm{\alpha}} \bm{\phi}_{k+1}(\bm{\beta}) \quad \text{if } \bm{\alpha} \text{ is a limit}.
\end{align*}
Additionally, we define
\[
\bm{\phi}_{\bm{\omega}}(\bm{0}) = \bigcup_{k \in \mathbb{N}} \bm{\phi}_k(\bm{0}).
\]
To see that $\bm{\phi}_{k+1}$ is well-defined, it suffices to observe that $\bm{\phi}_k$ is strictly monotone (see Theorem~\ref{PropertiesOfRealVeblen} below) and continuous, and hence it has arbitrarily large fixed points.
Informally, one often describes $\bm{\phi}_{k+1}(\bm{\alpha})$ as the $\bm{\alpha}$-th fixed point of $\bm{\phi}_k$, where the indexing begins at zero. Thus, $\bm{\phi}_{k+1}(\bm{0})$ denotes the first fixed point of $\bm{\phi}_k$, and $\bm{\phi}_{k+1}(\bm{1})$ denotes the second.

The following is a list of well-known properties of Veblen functions on set-theoretic ordinals:

\begin{theorem}\cite{Schutte1977} \label{PropertiesOfRealVeblen}
For any set-theoretic ordinals $\bm{\alpha}, \bm{\beta}, \bm{\gamma} \in \mathbf{On}$ and any $k, l \in \mathbb{N}$, the following hold:
\begin{itemize}
    \item 
If $\bm{\alpha} \subset \bm{\beta}$ then $\bm{\phi}_k(\bm{\alpha}) \subset \bm{\phi}_{k}(\bm{\beta})$.
  \item 
If $k \leq l$ then $\bm{\phi}_k(\bm{\alpha}) \subseteq \bm{\phi}_{l}(\bm{\alpha})$.
    \item 
If $k< l$ then $\bm{\phi}_k(\bm{\phi}_l(\bm{\alpha}))=\bm{\phi}_l(\bm{\alpha})$.
\item 
$\bm{\alpha} \subseteq \bm{\phi}_k(\bm{\alpha})$.
    \item 
If $\bm{\beta}, \bm{\gamma} \subset \bm{\phi}_k(\bm{\alpha})$ then $\bm{\beta+\gamma} \subset \bm{\phi}_k(\bm{\alpha})$.
\item 
If $\bm{0} \subset \bm{\alpha} \subset \bm{\phi}_{\bm{\omega}}(\bm{0})$, then $\bm{\alpha}$ has a unique representation in the form
\[
\bm{\alpha} = \bm{\phi}_{n_1}(\bm{\beta_1}) \bm{+ \cdots +} \bm{\phi}_{n_m}(\bm{\beta_m}),
\]
where $\bm{\phi}_{n_1}(\bm{\beta_1}) \supseteq \cdots \supseteq \bm{\phi}_{n_m}(\bm{\beta_m})$; for any $1 \leq i \leq m$ we have $\bm{\beta_i} \subset \bm{\phi}_{n_i}(\bm{\beta}_i)$; and $n_1 \geq \cdots \geq n_m$. In particular, for any $k \geq 1$, every set-theoretic ordinal $\bm{0} \subset \bm{\alpha} \subset \bm{\omega}^k$ has a unique representation of the form
$
\bm{\alpha} = \bm{\omega}^{p_1} + \cdots + \bm{\omega}^{p_m},
$
where $k > p_1 \geq p_2 \geq \cdots \geq p_m$.
\end{itemize}    
\end{theorem}

Note that, from $\bm{\alpha} \subseteq \bm{\phi}_{k+1}(\bm{\alpha})$, the monotonicity of $\bm{\phi}_k$, and the identity $\bm{\phi}_k(\bm{\phi}_{k+1}(\bm{\alpha})) = \bm{\phi}_{k+1}(\bm{\alpha})$, one can conclude
$\bm{\phi}_k^m(\bm{\alpha}) \subseteq \bm{\phi}_{k+1}(\bm{\alpha})$,
for any $k, m \in \mathbb{N}$ and any $\bm{\alpha} \in \mathbf{On}$. 

\begin{corollary}\label{ClosureUnderAdditionAndPhi}
For any set-theoretic ordinal $\bm{\alpha} \in \mathbf{On}$ and any $k \in \mathbb{N}$ and $m \in \mathbb{N}^{\geq 1}$, the set $\{\bm{\beta} \in \mathbf{On} \mid \bm{\beta} \subset \bm{\phi}_k^m(\bm{\alpha})\}$ contains $\bm{0}$ and is closed under addition and $\bm{\phi}_i$, for any $i < k$.   
\end{corollary}
\begin{proof}
As $m \geq 1$, by Theorem~\ref{PropertiesOfRealVeblen}, we have 
$\bm{0} \subset \bm{1} = \bm{\phi}_0(\bm{0}) \subseteq \bm{\phi}_k(\bm{\phi}_k^{m-1}(\bm{\alpha}))$.
Therefore, $\bm{0} \subset \bm{\phi}_k^m(\bm{\alpha})$. 
If $\bm{\beta},\bm{\gamma}\subset \bm{\phi}_k^m(\bm{\alpha})=\bm{\phi}_k(\bm{\phi}_k^{m-1}(\bm{\alpha}))$, then by Theorem~\ref{PropertiesOfRealVeblen} we have $\bm{\beta}+\bm{\gamma}\subset \bm{\phi}_k(\bm{\phi}_k^{m-1}(\bm{\alpha}))$. For the closure under $\bm{\phi}_i$ with $i < k$, note that if $\bm{\beta} \subset \bm{\phi}_k^m(\bm{\alpha})$, then since $m \geq 1$, by Theorem~\ref{PropertiesOfRealVeblen} we have
\[
\bm{\phi}_i(\bm{\beta}) \subset \bm{\phi}_i\bm{\phi}_k(\bm{\phi}_k^{m-1}(\bm{\alpha})) = \bm{\phi}_k(\bm{\phi}_k^{m-1}(\bm{\alpha}))=\bm{\phi}_k^{m}(\bm{\alpha}),
\]
which completes the proof.
\end{proof}

Recall the mapping~$[\![-]\!]\colon \Omega \to \mathbf{On}$ from constructive ordinals to set-theoretic ordinals, defined in Section~\ref{sec:const-ordinals}. Then:

\begin{theorem}\label{TheInterpretationOfPhi}
$[\![\phi_k(\alpha)]\!] = \bm{\phi}_k([\![\alpha]\!])$, for any $k \in \mathbb{N}$ and any $\alpha \in \Omega$.
\end{theorem}
\begin{proof}
We prove the claim by induction on $k$. For $k=0$, the claim is clear, as $[\![\omega^{\alpha}]\!]=[\![\omega]\!]^{[\![\alpha]\!]}=\bm{\omega}^{[\![\alpha]\!]}$, for any $\alpha \in \Omega$. For the induction step, we assume the claim for $k$ to prove it for $k+1$. For that purpose, we use an induction on $\alpha$. If $\alpha=0$, first,
by the induction hypothesis for $k$, we have:
\[
[\![\phi_{k+1}(0)]\!]=[\![\langle \phi^{(i)}_{k}(0) \rangle]\!]=\bigcup_{i \in \mathbb{N}} \bm{\phi}_k^{(i)}(\bm{0}).
\]
Second, as $\bm{\phi}_k^{(i)}(\bm{0})\subseteq \bm{\phi}_{k+1}(\bm{0})$, for any $i \in \mathbb{N}$, we have $\bigcup_{i \in \mathbb{N}} \bm{\phi}_k^{(i)}(\bm{0}) \subseteq \bm{\phi}_{k+1}(\bm{0})$.
Third, we claim that $\bigcup_{i \in \mathbb{N}} \bm{\phi}_k^{(i)}(\bm{0})$ is a fixed point of $\bm{\phi}_k$, i.e., $\bm{\phi}_k(\bigcup_{i \in \mathbb{N}} \bm{\phi}_k^{(i)}(\bm{0}))=\bigcup_{i \in \mathbb{N}} \bm{\phi}_k^{(i)}(\bm{0})$. To prove that, it is enough to use $\bm{\phi}_k^{(0)}(\bm{0})=\bm{0}$ and the continuity of $\bm{\phi}_k$ (see after Example~\ref{ex-Veblen}) to have:
\[
\bigcup_{i \in \mathbb{N}}\bm{\phi}^{(i)}_k(\bm{0})= \bigcup_{i \geq 1}\bm{\phi}^{(i)}_k(\bm{0})=\bm{\phi}_k(\bigcup_{i \in \mathbb{N}}\bm{\phi}^{(i)}_k(\bm{0})).
\]
Now, as $ \bm{\phi}_{k+1}(\bm{0})$ is the least fixed point of $ \bm{\phi}_{k}$, we have $\bm{\phi}_{k+1}(\bm{0}) \subseteq \bigcup_{i \in \mathbb{N}} \bm{\phi}_k^{(i)}(\bm{0})$. Therefore, $\bm{\phi}_{k+1}(\bm{0}) = \bigcup_{i \in \mathbb{N}} \bm{\phi}_k^{(i)}(\bm{0})$ which implies $[\![\phi_{k+1}(0)]\!]=\bm{\phi}_{k+1}([\![0]\!]) $.

If $\alpha=\beta+1$, first, by the induction hypothesis for $k$ and for $\beta$, we have:
\[
[\![\phi_{k+1}(\alpha)]\!]=[\![\phi_{k+1}(\beta+1)]\!]=[\![\langle \phi^{(i)}_{k}(\phi_{k+1}(\beta)+1) \rangle]\!]=\bigcup_{i \in \mathbb{N}} \bm{\phi}_k^{(i)}(\bm{\phi}_{k+1}([\![\beta]\!])\bm{+1}).
\]
Second, as $\bm{\phi}_{k+1}([\![\beta]\!]\bm{+1})$ is the least fixed point of $\bm{\phi}_k$ above $\bm{\phi}_{k+1}([\![\beta]\!])$, we have $\bm{\phi}_{k+1}([\![\beta]\!])\bm{+1} \subseteq \bm{\phi}_{k+1}([\![\beta]\!]\bm{+1})$. Then, as $\bm{\phi}_k$ is monotone, we reach:
\[
\bigcup_{i\in \mathbb{N}}\bm{\phi}^{(i)}_k(\bm{\phi}_{k+1}([\![\beta]\!])\bm{+1}) \subseteq \bigcup_{i\in \mathbb{N}}\bm{\phi}^{(i)}_k(\bm{\phi}_{k+1}([\![\beta]\!]\bm{+1}))=\bm{\phi}_{k+1}([\![\beta]\!]\bm{+1}).
\]
Third, we claim that 
\[
\bigcup_{i \in \mathbb{N}}\bm{\phi}^{(i)}_k(\bm{\phi}_{k+1}([\![\beta]\!])\bm{+1})
\]
is a fixed point of $\bm{\phi}_k$. To prove that,
recall that $\bm{\gamma} \subseteq \bm{\phi}_{k}(\bm{\gamma})$, for any $k \in \mathbb{N}$ and any $\bm{\gamma} \in \mathbf{On}$. Therefore,
$\bm{\phi}_{k+1}([\![\beta]\!])\bm{+1} \subseteq \bm{\phi}_k(\bm{\phi}_{k+1}([\![\beta]\!])\bm{+1})$. By
the continuity of $\bm{\phi}_k$, we reach
\[
\bigcup_{i \in \mathbb{N}}\bm{\phi}^{(i)}_k(\bm{\phi}_{k+1}([\![\beta]\!])\bm{+1})= \bigcup_{i \geq 1}\bm{\phi}^{(i)}_k(\bm{\phi}_{k+1}([\![\beta]\!])\bm{+1})=\bm{\phi}_k(\bigcup_{i \in \mathbb{N}}\bm{\phi}^{(i)}_k(\bm{\phi}_{k+1}([\![\beta]\!])\bm{+1})).
\]
Therefore, as $\bigcup_{i \in \mathbb{N}}\bm{\phi}^{(i)}_k(\bm{\phi}_{k+1}([\![\beta]\!])\bm{+1})\supseteq \bm{\phi}_{k+1}([\![\beta]\!])\bm{+1}$ and $\bm{\phi}_{k+1}([\![\beta]\!]\bm{+1})$ is the least fixed point of $\bm{\phi}_k$ after $\bm{\phi}_{k+1}([\![\beta]\!])$, we reach:
\[
\bm{\phi}_{k+1}([\![\beta]\!]\bm{+1}) \subseteq \bigcup_{i \in \mathbb{N}}\bm{\phi}^{(i)}_k(\bm{\phi}_{k+1}([\![\beta]\!])\bm{+1}).
\]
Thus, $\bigcup_{i \in \mathbb{N}} \bm{\phi}_k^{(i)}(\bm{\phi}_{k+1}([\![\beta]\!])\bm{+1})=\bm{\phi}_{k+1}([\![\beta]\!]\bm{+1})$. Therefore, we can conclude $[\![\phi_{k+1}(\beta+1)]\!]=\bm{\phi}_{k+1}([\![\beta +1]\!])$.

Finally, if $\alpha=\langle \alpha_i \rangle$ is a limit, using the induction hypothesis for $\alpha_i$'s and the fact that $\bm{\phi}_{k+1}$ is continuous, we have:
\[
[\![\phi_{k+1}(\langle \alpha_i \rangle)]\!]=[\![\langle \phi_{k+1}(\alpha_i) \rangle]\!]=\bigcup_{i \in \mathbb{N}} \bm{\phi}_{k+1}([\![\alpha_i]\!])=\bm{\phi}_{k+1}(\bigcup_{i \in \mathbb{N}}  [\![\alpha_i]\!])=\bm{\phi}_{k+1}([\![\langle \alpha_i \rangle]\!])
\]
which completes the proof.
\end{proof}

\begin{remark}
Theorem~\ref{TheInterpretationOfPhi}, together with the definition of $[\![-]\!]$, provides a \emph{constructive} way to define $\bm{\phi}_{k+1}(\bm{0})$, i.e., the first fixed point of $\bm{\phi}_k$, as the limit of all finite iterations of $\bm{\phi}_k$ on the ordinal zero. Similarly, $\bm{\phi}_{k+1}(\bm{\alpha + 1})$, the $(\bm{\alpha + 1})$-th fixed point of $\bm{\phi}_k$, can be defined constructively as the limit of all finite iterations of $\bm{\phi}_k$ applied to the first ordinal above the $\bm{\alpha}$-th fixed point, i.e., $\bm{\phi}_{k+1}(\bm{\alpha}) + \bm{1}$. From this perspective, $\phi_k$ serves as the constructive counterpart of the set-theoretic function $\bm{\phi}_k$. 
\end{remark}

\begin{remark}\label{WhyStructurednessFails}
When working with constructive ordinals, one often restricts attention to \emph{structured ordinals}, a subclass of constructive ordinals whose behavior more closely resembles that of set-theoretic ordinals \cite{fairtlough1998ptchapter}. We will not define structured ordinals formally here, but their key property is that for any structured ordinal $\alpha$, the set $\mathsf{D}_{\alpha}$ is linearly ordered by $\prec$.  
Unfortunately, Veblen functions generate \emph{unstructured} ordinals. Since our focus is on these functions, we must therefore work with unstructured ordinals and confront their potential anomalies. To illustrate why structuredness fails, let us show that the set $\mathsf{D}_{\phi_1(1)}$ is not linearly ordered.  
First, observe that both $\epsilon_0 = \phi_1(0)$ and $\omega^{\epsilon_0} = \phi_0(\phi_1(0))$ lie below $\phi_1(1)$, since  
\[
\epsilon_0 \prec \epsilon_0 + 1 = \phi^0_0(\epsilon_0 + 1) 
\prec \langle \phi^i_0(\epsilon_0 + 1) \rangle_i = \phi_1(1),
\]
and  
\[
\omega^{\epsilon_0} \prec 
\langle \omega^{\epsilon_0} \cdot o(i + 1) \rangle_i 
= \omega^{\epsilon_0 + 1} = \phi_0(\epsilon_0 + 1) 
\prec \langle \phi^i_0(\epsilon_0 + 1) \rangle_i = \phi_1(1).
\]
Second, note that $\omega^{\epsilon_0} \neq \epsilon_0$, since $\epsilon_0 = \langle \phi_0^{i}(0) \rangle_i$ and $\omega^{\epsilon_0} = \langle \phi_0^{i+1}(0) \rangle_i$ are sequences of ordinals whose first elements are $0$ and $\omega^0 = 1$, respectively, which are distinct.  
Finally, we show that $\epsilon_0$ and $\omega^{\epsilon_0}$ are incomparable. Suppose $\epsilon_0 \prec \omega^{\epsilon_0}$. Then there exists $i \in \mathbb{N}$ such that $\epsilon_0 \preceq \phi_0^{i+1}(0)$. Interpreting these ordinals as set-theoretic ones would then yield $\bm{\epsilon_0} \subseteq \bm{\phi}_0^{i+1}(\bm{0})$, which is impossible, as the latter represents a finite tower of $\bm{\omega}$’s. Similarly, $\omega^{\epsilon_0} \prec \epsilon_0$ would imply $\omega^{\epsilon_0} \preceq \phi_0^{i}(0)$, which is again impossible for the same reason.  
Hence, $\epsilon_0$ and $\omega^{\epsilon_0}$ are two incomparable elements of $\mathsf{D}_{\phi_1(1)}$, and therefore $\phi_1(1)$ is not structured.
\end{remark}

\subsection{$\Phi_{\omega}$ and its subclasses}

Theorem~\ref{TheInterpretationOfPhi} shows that the function $\phi_k$ serves as a constructive analogue of the set-theoretic Veblen function $\bm{\phi}_k$. By Theorem \ref{PropertiesOfRealVeblen}, every set-theoretic ordinal below $\bm{\phi}_{\bm{\omega}}(\bm{0})$ can be built from $\bm{0}$ using addition and the functions $\bm{\phi}_k$. Therefore, we can imitate this construction and use the constructive Veblen functions $\phi_k$ and addition on constructive ordinals to provide a representation for these ordinals. In this way, the constructive framework offers a canonical representation for the set-theoretic ordinals below $\bm{\phi}_{\bm{\omega}}(\bm{0})$. First, we need to define some classes of ordinals.

\begin{definition}\label{Phi-sets}
Let $k \geq 0$ be a natural number. Define the set $\Phi_{\omega} \subseteq \Omega$ (resp. $\Phi_k \subseteq \Omega$) of constructive ordinals as the smallest set containing $0$ and closed under addition and the functions $\phi_i$ for any $i \in \mathbb{N}$ (resp. for any $i < k$). Additionally, for any $m \geq 1$, define $\Phi_k^m$ recursively by setting $\Phi^1_k = \Phi_k$, and defining $\Phi^{m+1}_k$ as the smallest set containing $\phi_k[\Phi_k^m] \cup \{0\}$ and closed under addition and the functions $\phi_i$ for $i < k$.
\end{definition}

In words, the class $\Phi_\omega$ (resp.\ $\Phi_k$) consists of all ordinals constructible from $0$ by addition and the functions $\phi_i$ for $i \in \mathbb{N}$ (resp.\ $i < k$). Similarly, $\Phi^m_k$ is the set of all ordinals constructible from $0$ by addition and the functions $\phi_i$ for $i \leq k$, where the depth of nested applications of $\phi_k$ is at most $m-1$. For instance, $\Phi^1_0 = \Phi_0 = \{0\}$ and $\Phi^2_0 = \{ o(i) \mid i \in \mathbb{N} \}$ is the set of all finite ordinals.

\begin{lemma}\label{MonotonicityOfPhi}
Let $k \geq 0$ and $m \geq 1$ be natural numbers. Then:
\begin{itemize}
    \item[$(i)$] 
$\Phi_k^m \subseteq \Phi^{m+1}_k$.
    \item[$(ii)$] 
$\Phi_{k+1}=\bigcup_{m=1}^{\infty}\Phi^{m}_k$.
\item[$(iii)$]
$\Phi_{\omega}=\bigcup_{k=0}^{\infty} \Phi_k$.
\end{itemize}
\end{lemma}
\begin{proof}
For $(i)$, we use an induction on $m$. For $m=1$, we have $\Phi_k^1 = \Phi_k$. Since $0 \in \Phi^{2}_k$ and $\Phi^{2}_k$ is closed under addition and all $\phi_i$'s for $i < k$, it follows that $\Phi_k \subseteq \Phi^{2}_k$. For the induction step, by definition, $\Phi^{m+2}_k$ contains $0$ and is closed under addition and $\phi_i$'s for all $i < k$. Therefore, to show that 
$\Phi_k^{m+1} \subseteq \Phi^{m+2}_k$, it suffices to prove that $\phi_k[\Phi_k^m] \subseteq \Phi^{m+2}_k$. For that, by the induction hypothesis, $\Phi_k^m \subseteq \Phi^{m+1}_k$, which implies $\phi_k[\Phi_k^m] \subseteq \phi_k[\Phi_k^{m+1}]$. By definition, $\phi_k[\Phi_k^{m+1}] \subseteq \Phi^{m+2}_k$. Hence, $\phi_k[\Phi_k^m] \subseteq \Phi^{m+2}_k$. 

For $(ii)$, an easy induction on $m$ shows that $\Phi^m_k \subseteq \Phi_{k+1}$. Therefore, $\bigcup_{m=1}^{\infty} \Phi^{m}_k \subseteq \Phi_{k+1}$. For the converse, i.e., $\Phi_{k+1} \subseteq \bigcup_{m=1}^{\infty} \Phi^{m}_k$, note that $0 \in \bigcup_{m=1}^{\infty} \Phi^{m}_k$ and $\bigcup_{m=1}^{\infty} \Phi^{m}_k$ is closed under addition and any $\phi_i$ for $i \leq k$.
The proof of $(iii)$ is straightforward.
\end{proof}

As expected, $\Phi^m_k$ provides a constructive representation for some set-theoretic ordinals. The following theorem shows that it is the set of all ordinals below $\bm{\phi}_k^m(\bm{0})$.

\begin{theorem}\label{MeaningOfPhi}
Let $k \in \mathbb{N}$ and $m \in \mathbb{N}^{\geq 1}$ be natural numbers. Then:
\begin{itemize}
    \item[$(i)$]
$[\![\alpha]\!] \subset  \bm{\phi}^m_k(\bm{0})$, for any $\alpha \in \Phi^m_k$.
     \item[$(ii)$]
For any $\bm{\alpha} \subset \bm{\phi}^m_k(\bm{0})$, there is $\alpha \in \Phi^m_k$ such that $[\![\alpha]\!]=\bm{\alpha}$.  
\end{itemize}
\end{theorem}
\begin{proof}
First, recall that $[\![\alpha + \beta]\!]=[\![\alpha]\!]\bm{+}[\![\beta]\!]$ and $[\![\phi_i(\alpha)]\!]=\bm{\phi}_i([\![\alpha]\!])$, for any $\alpha, \beta \in \Omega$ and any $i \in \mathbb{N}$. The latter is proved in Theorem \ref{TheInterpretationOfPhi}. For $(i)$, we use induction on $m$. For $m=1$, by Theorem \ref{TheInterpretationOfPhi} and Corollary \ref{ClosureUnderAdditionAndPhi}, it is clear that the set 
$X=\{\alpha \in \Phi_k \mid [\![\alpha]\!] \subset \bm{\phi}_k(\bm{0})\}$ 
contains $0$ and is closed under addition and all $\phi_i$'s for $i < k$. Therefore, by the definition of $\Phi_k$, we have $X=\Phi_k$ which completes the proof. 
For the induction step, to prove the claim for $\Phi^{m+1}_k$, by Theorem \ref{TheInterpretationOfPhi} and Corollary \ref{ClosureUnderAdditionAndPhi} again, we know that the set 
$Y=\{\alpha \in \Phi^{m+1}_k \mid [\![\alpha]\!] \subset \bm{\phi}^{m+1}_k(\bm{0})\}$ 
contains $0$ and is closed under addition and all $\phi_i$'s for $i < k$. 
Moreover, if $\alpha = \phi_k(\beta)$ with $\beta \in \Phi^m_k$, by the induction hypothesis, we have 
$[\![\beta]\!] \subset \bm{\phi}^m_k(\bm{0})$, which implies 
$[\![\alpha]\!] = \bm{\phi}_k([\![\beta]\!]) \subset \bm{\phi}_k(\bm{\phi}^m_k(\bm{0})) = \bm{\phi}^{m+1}_k(\bm{0})$, using the strict monotonicity of $\bm{\phi}_k$ from Theorem \ref{PropertiesOfRealVeblen}. Therefore, using the definition of $\Phi^{m+1}_k$, we reach $Y=\Phi^{m+1}_k$ which completes the proof.

For $(ii)$, we use transfinite induction on $\bm{\alpha}$. If $\bm{\alpha}=\bm{0}$, as $[\![0]\!]=\bm{0}$ and $0 \in \Phi^m_k$, there is nothing to prove. If $\bm{\alpha} \neq \bm{0}$, by Theorem \ref{PropertiesOfRealVeblen}, $\bm{\alpha}$ has a representation in the form
$\bm{\alpha} = \bm{\phi}_{n_1}(\bm{\beta_1}) \bm{+ \cdots +} \bm{\phi}_{n_m}(\bm{\beta_m})$ where $\bm{\beta_i} \subset \bm{\phi}_{n_i}(\bm{\beta}_i)$, for any $1 \leq i \leq m$. We prove the following claim:\\

\noindent \textbf{Claim.} For any $1 \leq i \leq m$, either $n_i < k$ and $\bm{\beta_i} \subset \bm{\phi}^m_{k}(\bm{0})$ or $n_i=k$ and $\bm{\beta_i} \subset \bm{\phi}^{m-1}_{k}(\bm{0})$. The second case can take place if $m \geq 2$.\\

\noindent To prove the claim, if there is any $1 \leq i \leq m$ such that $n_i \geq k+1$, then by the monotonicity in Theorem \ref{PropertiesOfRealVeblen}, we have $\bm{\alpha} \supseteq \bm{\phi}_{n_i}(\bm{\beta}_i) \supseteq \bm{\phi}_{k+1}(\bm{0})$ which contradicts the assumption $\bm{\alpha} \subset \bm{\phi}^m_{k}(\bm{0})$.
Therefore, $n_i \leq k$, for any $1 \leq i \leq m$. Now, for any fixed $1 \leq i \leq m$, if $n_i < k$, then $\bm{\beta_i} \subset \bm{\phi}_{n_i}(\bm{\beta}_i) \subseteq \bm{\alpha} \subset \bm{\phi}^m_{k}(\bm{0})$.
If $n_i=k$ then $\bm{\beta_i} \subset \bm{\phi}^{m-1}_{k}(\bm{0})$ because otherwise, we have $\bm{\beta_i} \supseteq \bm{\phi}^{m-1}_{k}(\bm{0})$ which by the monotonicity of $\bm{\phi}_k$ implies $\bm{\alpha} \supseteq \bm{\phi}_k(\bm{\beta_i}) \supseteq \bm{\phi}^{m}_{k}(\bm{0})$ which is in contradiction with the assumption $\bm{\alpha} \subset \bm{\phi}^m_k(\bm{0})$. This completes the proof of the claim.

Now, note that $\bm{\beta_i} \subset \bm{\alpha}$ for every $1 \leq i \leq m$. Therefore, using the above claim and the induction hypothesis on the $\bm{\beta_i}$'s, we see that for each $1 \leq i \leq m$, either $n_i < k$ and there exists $\beta_i \in \Phi^{m}_k$ such that $[\![\beta_i]\!] = \bm{\beta_i}$, or $n_i = k$ and there exists $\beta_i \in \Phi^{m-1}_k$ such that $[\![\beta_i]\!] = \bm{\beta_i}$. In either case, $\phi_{n_i}(\beta_i) \in \Phi^m_k$. Define 
$\alpha = \phi_{n_1}(\beta_1) + \cdots + \phi_{n_m}(\beta_m)$.
Hence, $\alpha \in \Phi^m_k$. Since $[\![\alpha]\!] = \bm{\alpha}$, the proof is complete.
\end{proof}

Theorem \ref{MeaningOfPhi} shows that the class $\Phi^m_k$ provides a constructive counterpart to the set of all set-theoretic ordinals below $\bm{\phi}^m_k(\bm{0})$. In particular, for $m=1$, this means that $\Phi_k$ is the constructive counterpart of the set of all ordinals below $\bm{\phi}_k(\bm{0})$ and, hence, $\Phi_{\omega}$ corresponds to the ordinals below $\bm{\phi}_{\bm{\omega}}(\bm{0})$.

For $k=0$ and $m=3$, the class $\Phi^m_k$ provides a constructive counterpart of the set of all set-theoretic ordinals below
$
\bm{\phi}^3_0(\bm{0})=\bm{\omega}^{\bm{\omega}^{\bm{\omega}^{\bm{0}}}}=\bm{\omega}^{\bm{\omega}}.
$
To obtain a counterpart for ordinals below $\bm{\omega}^k$, we need to introduce another family of classes of constructive ordinals:

\begin{definition}\label{def:Psi-mk}
Let $k \geq 1$ be a natural number. Define the set $\Psi_{\omega} \subseteq \Omega$ (resp.\ $\Psi_k$) of constructive ordinals as the closure of
$\{0\} \cup \{\omega^p \mid p \in \mathbb{N}\}$ (resp.\ $\{0\} \cup \{\omega^p \mid p < k\}$) under addition. Moreover, for any $m \geq 1$, define $\Psi^m_k$ as the set
\[
\Psi^m_k = \{0\} \cup \left\{\sum_{i=1}^m \omega^{p_i} c_i + \sum_{j=0}^{k-1} \omega^{k-1-j} d_j \,\middle|\, \forall i \in \{1, \ldots, m\} \, (p_i < k)\right\}
\]
of ordinals, where $c_i$'s and $d_j$'s are finite ordinals.
\end{definition}
In words, an element of $\Psi_k$ is either $0$ or a sum of powers of $\omega$, each with an exponent less than $k$.  
An element of $\Psi^m_k$ is either $0$ or a sum of $m+k$ terms, where each term is a power of $\omega$ with a finite ordinal coefficient and an exponent below $k$.  
For the first $m$ terms, the exponents of $\omega$ may be any ordinals less than $k$, while the remaining $k$ terms form a strictly decreasing sequence of powers whose largest exponent is at most $k-1$.  
The motivation for this definition is to obtain a covering family of subclasses of $\Psi_k$, where each ordinal is expressed as a sum of a \emph{bounded} number of $\omega$-terms.

\begin{lemma}\label{MonotonicityOfPsi}
Let $k, m \geq 1$ be natural numbers. Then,
\begin{itemize}
    \item[$(i)$] 
$\Psi^m_k \subseteq \Psi^{m+1}_k$ and $\Psi_k=\bigcup_{m=1}^{\infty}\Psi^m_{k}$. 
    \item[$(ii)$] 
$\Psi_k \subseteq \Psi_{k+1}$ and
$\Psi_{\omega}=\Phi^3_0=\bigcup_{k=1}^{\infty}\Psi_{k}$. Hence, $\Psi_{\omega} \subseteq \Phi_1$.
\end{itemize}
\end{lemma}
\begin{proof}
The proof is straightforward.
\end{proof}

The following theorem shows that the class $\Psi_k$ provides a constructive representation for the ordinals below $\bm{\omega}^k$.

\begin{theorem}\label{MeaningOfPsi}
Let $k \in \mathbb{N}^{\geq 1}$ be a natural number. Then:
\begin{itemize}
    \item[$(i)$]
$[\![\alpha]\!] \subset \bm{\omega}^{\bm{k}}$, for any $\alpha \in \Psi_k$.
     \item[$(ii)$]
For any $\bm{\alpha} \subset \bm{\omega}^{k}$, there is $\alpha \in \Psi_k$ such that $[\![\alpha]\!]=\bm{\alpha}$.  
\end{itemize}
\end{theorem}
\begin{proof}
First, recall that $[\![\alpha + \beta]\!] = [\![\alpha]\!] \bm{+} [\![\beta]\!]$ and $[\![\omega^{\alpha}]\!] = \bm{\omega}^{[\![\alpha]\!]}$, for any $\alpha, \beta \in \Omega$.  
For $(i)$, consider the set
$
X = \{\alpha \in \Psi_k \mid [\![\alpha]\!] \subset \bm{\omega}^k\}.
$
Since $[\![0]\!] = \bm{0}$, it is clear that $0 \in X$. Moreover, as $\bm{\omega}^k = \bm{\phi}_0(\bm{k})$, by Theorem~\ref{PropertiesOfRealVeblen}, the set $X$ is closed under addition. Hence, by the definition of $\Psi_k$, we obtain $X = \Psi_k$, which completes the proof of $(i)$.  
For $(ii)$, if $\bm{\alpha} = \bm{0}$ it suffices to take $\alpha = 0$. Otherwise, by Theorem~\ref{PropertiesOfRealVeblen}, $\bm{\alpha}$ can be represented as
$
\bm{\alpha} = \bm{\omega}^{p_1} \bm{+ \cdots +} \bm{\omega}^{p_m},
$
where $p_i < k$ for all $1 \leq i \leq m$. Now, define $\alpha = \omega^{p_1} + \cdots + \omega^{p_m}$. Clearly, $\alpha \in \Psi_k$ and $[\![\alpha]\!] = \bm{\alpha}$.
\end{proof}

\subsection{Unique representation}

So far, we have provided a constructive representation for some set-theoretic ordinals. The next natural point to consider is the uniqueness of such representations. Unfortunately, as already observed, a set-theoretic ordinal can have many constructive representations. For instance, 
\[
[\![\phi_0(0)+\phi_0(\phi_0(0))]\!]=[\![1+\omega]\!]=\bm{\omega}=[\![\omega]\!]=[\![\phi_0(\phi_0(0))]\!],
\]
while $\phi_0(0)+\phi_0(\phi_0(0))=1+\omega \neq \omega=\phi_0(\phi_0(0))$. However, it is possible to prove another uniqueness result, which states that if an ordinal is constructible by addition and $\phi_k$'s from $0$, then this construction is unique. To prove this uniqueness theorem, we begin with an easy lemma.

\begin{lemma}\label{phiIsNonZero}
Let $\alpha \in \Omega$ be an ordinal and $k \in \mathbb{N}$ a natural number. Then, $\phi_k(\alpha)$ is a limit ordinal iff either $k \neq 0$ or $\alpha \neq 0$. Consequently, $\phi_k(\alpha) \neq 0$.
\end{lemma}
\begin{proof}
If $k=0$ and $\alpha=0$, we have $\phi_k(\alpha)=1$, which is not a limit.  
Otherwise, if $k=0$ and $\alpha = \alpha' + 1$ is a successor, $\phi_0(\alpha) = \phi_0(\alpha' + 1) = \phi_0(\alpha') \cdot \omega$ is a limit, and if $\alpha = \langle \alpha_i \rangle$ is a limit, then $\phi_0(\alpha) = \langle \phi_0(\alpha_i) \rangle$ is also a limit.  
For $k > 0$, there is $l \in \mathbb{N}$ such that $k = l + 1$. By looking into the definition, it is clear that $\phi_{l+1}(\alpha)$ is always a limit. This proves the first part. For the second part, by the first part, $\phi_k(\alpha)$ is either a limit, or $k = 0$ and $\alpha = 0$, which implies $\phi_k(\alpha) = 1$. In both cases, $\phi_k(\alpha) \neq 0$.
\end{proof}

The following lemma is the seed of the uniqueness of construction we claimed earlier. The proof involves extensive case-checking. However, since the claim only holds for constructive ordinals as opposed to set-theoretic ones, and hence may be counter-intuitive at first glance, we will provide a detailed and complete proof.

\begin{lemma}\label{UniquenessLemma}
For any $\alpha, \beta, \gamma, \delta \in \Omega$ and any $k, l \in \mathbb{N}$, if $\alpha+\phi_k(\beta)=\gamma+\phi_l(\delta)$, then $k=l$, $\alpha=\gamma$, and $\beta=\delta$. 
\end{lemma}
\begin{proof}

By induction on $\rho \in \Omega$, we prove that for any $\alpha, \beta, \gamma, \delta \in \Omega$ and any $k, l \in \mathbb{N}$, if $\rho = \alpha + \phi_k(\beta) = \gamma + \phi_l(\delta)$, then $k = l$, $\alpha = \gamma$, and $\beta = \delta$.  
Assume the claim holds for all ordinals below $\rho$, and let $\rho = \alpha + \phi_k(\beta) = \gamma + \phi_l(\delta)$.  
To prove $k=l$, $\alpha=\gamma$, and $\beta=\delta$, there are four cases to consider, depending on whether $k$ or $l$ is zero or not:

\item[] I. 
If $k = l = 0$, we need to prove that $\alpha = \gamma$ and $\beta = \delta$.  
There are nine cases to consider, depending on whether $\beta$ and $\delta$ are zero, successors, or limits:

\item[] $\bullet$ If $\beta=\delta=0$, then $\alpha+1=\gamma+1$ which implies that $\alpha=\gamma$.

\item[] $\bullet$ If $\beta = 0$ and $\delta$ is non-zero, then since $\rho = \alpha + 1$, we know that $\rho$ is a successor.  
However, as $\rho = \gamma + \phi_0(\delta)$ and $\delta$ is non-zero, by Lemma \ref{phiIsNonZero}, $\phi_0(\delta)$ and hence $\rho$ is a limit, which is a contradiction.

\item[] $\bullet$ If $\beta = \beta' + 1$ is a successor and $\delta = 0$, the proof is similar to the previous case.

\item[] $\bullet$ If $\beta = \beta' + 1$ and $\delta = \delta' + 1$ are successors, we have $\rho = \alpha + \phi_0(\beta' + 1) = \alpha + \phi_0(\beta') \cdot \omega = \langle \alpha + \phi_0(\beta')(i + 1) \rangle_i$.  
Similarly, $\rho = \langle \gamma + \phi_0(\delta')(i + 1) \rangle_i$.  
This is an equality between two sequences. Therefore, for $i = 0$, we reach $\alpha + \phi_0(\beta') = \gamma + \phi_0(\delta')$.  
Since $\alpha + \phi_0(\beta') \prec \rho$, by the induction hypothesis, we have $\alpha = \gamma$ and $\beta' = \delta'$, which implies $\beta = \delta$.

\item[] $\bullet$ If $\beta = \beta' + 1$ is a successor and $\delta = \langle \delta_i \rangle$ is a limit, then $\rho = \alpha + \phi_0(\beta') \cdot \omega = \langle \alpha + \phi_0(\beta')(i + 1) \rangle_i$.  
On the other hand, $\rho = \langle \gamma + \phi_0(\delta_i) \rangle_i$.  
For $i = 0$, we reach $\alpha + \phi_0(\beta') = \gamma + \phi_0(\delta_0)$.  
Since $\alpha + \phi_0(\beta') \prec \rho$, by the induction hypothesis, we have $\alpha = \gamma$.  
Then, for $i = 1$, we have $\alpha + \phi_0(\beta') + \phi_0(\beta') = \gamma + \phi_0(\delta_1)$.  
Again, by the induction hypothesis, we get $\alpha + \phi_0(\beta') = \gamma$.  
Therefore, $\alpha = \alpha + \phi_0(\beta')$. However, by Lemma \ref{phiIsNonZero}, we have $\phi_0(\beta') \neq 0$. Therefore, by Lemma \ref{AdditionProp}, we reach $\alpha \prec \alpha + \phi_0(\beta')$, which is a contradiction.

\item[] $\bullet$ If $\beta = \langle \beta_i \rangle$ is a limit and $\delta$ is either zero or a successor, the case is similar to the previous ones.  
If $\delta = \langle \delta_i \rangle$ is a limit, then  
$\rho = \langle \alpha + \phi_0(\beta_i) \rangle_i = \langle \gamma + \phi_0(\delta_i) \rangle_i$.  
Hence, $\alpha + \phi_0(\beta_i) = \gamma + \phi_0(\delta_i)$, for any $i \in \mathbb{N}$.  
Therefore, as $\alpha + \phi_0(\beta_i) \prec \rho$, by the induction hypothesis, we have $\alpha = \gamma$ and $\beta_i = \delta_i$, for any $i \in \mathbb{N}$.  
Hence, $\langle \beta_i \rangle = \langle \delta_i \rangle$.\\

\noindent II. If $k=0$ and $l=n+1$, for some $n \in \mathbb{N}$, we must reach a contradiction.  
There are again nine cases to consider depending if $\beta$ and $\delta$ are zero, a successor or a limit:

\item[] $\bullet$ If $\beta=0$, then $\rho=\alpha+\phi_0(0)$ is a successor. However, by Lemma \ref{phiIsNonZero}, $\phi_{n+1}(\delta)$ and hence $\rho=\gamma+\phi_{n+1}(\delta)$ is a limit, which is a contradiction.

\item[] $\bullet$ If $\beta=\beta'+1$ is a successor and $\delta=0$, we have $\alpha+\phi_0(\beta'+1)=\gamma+\phi_{n+1}(0)$. Therefore,  
$\langle \alpha+\phi_0(\beta')(i+1) \rangle_i = \langle \gamma + \phi_n^i(0) \rangle_i$.  
Then, for $i=1$, we have  
$\alpha + \phi_0(\beta') + \phi_0(\beta') = \gamma + \phi_n(0)$,  
which by the induction hypothesis implies $n=0$ and $\beta' = 0$.  
For $i=2$, we have  
$\alpha + \phi_0(\beta') + \phi_0(\beta') + \phi_0(\beta') = \gamma + \phi_n(\phi_n(0))$,  
which by the induction hypothesis implies $\beta' = \phi_n(0)$.  
Therefore, $0 = \beta' = \phi_0(0) = 1$, which is impossible.

\item[] $\bullet$ If $\beta=\beta'+1$ and $\delta=\delta'+1$ are successors, we have  
$\alpha+\phi_0(\beta'+1) = \gamma + \phi_{n+1}(\delta'+1)$.   
Therefore,  
$\langle \alpha+\phi_0(\beta')(i+1) \rangle_i = \langle \gamma + \phi_n^i(\phi_{n+1}(\delta') + 1) \rangle_i$.  
For $i=0$, we have  
$\alpha + \phi_0(\beta') = \gamma + \phi_{n+1}(\delta') + 1$.  
Therefore, $\beta' = 0$ by the induction hypothesis.  
For $i=1$, we have  
$\alpha + \phi_0(\beta') + \phi_0(\beta') = \gamma + \phi_n(\phi_{n+1}(\delta') + 1)$,  
which by the induction hypothesis implies $\beta' = \phi_{n+1}(\delta') + 1$, which is a contradiction.

\item[] $\bullet$ If $\beta=\beta'+1$ is a successor and $\delta=\langle \delta_i \rangle$ is a limit, we have  
$\alpha+\phi_0(\beta'+1) = \gamma + \phi_{n+1}(\langle \delta_i \rangle)$.  
Therefore,  
$\langle \alpha+\phi_0(\beta')(i+1) \rangle_i = \langle \gamma + \phi_{n+1}(\delta_i) \rangle_i$.  
For $i=0$, we have  
$\alpha + \phi_0(\beta') = \gamma + \phi_{n+1}(\delta_0)$.  
By the induction hypothesis, this implies $n+1=0$ which is a contradiction.

\item[] $\bullet$ If $\beta=\langle \beta_i \rangle$ is a limit and $\delta=0$, we have $\alpha+\phi_0(\langle \beta_i \rangle)=\gamma+\phi_{n+1}(0)$. Therefore, $\langle \alpha+\phi_0(\beta_i) \rangle_i=\langle \gamma+\phi^i_{n}(0) \rangle_i$. For $i=0$, we have $\alpha+\phi_0(\beta_0)=\gamma$. For $i=1$, we have $\alpha+\phi_0(\beta_1)=\gamma+\phi_{n}(0)$ which by the induction hypothesis implies $\alpha=\gamma$. Therefore, $\alpha=\alpha+\phi_0(\beta_0)$. However, by Lemma \ref{phiIsNonZero}, we have $\phi_0(\beta_0) \neq 0$. Therefore, by Lemma \ref{AdditionProp}, we reach $\alpha \prec \alpha + \phi_0(\beta_0)$, which is a contradiction.

\item[] $\bullet$ If $\beta=\langle \beta_i \rangle$ is a limit and $\delta=\delta'+1$ is a successor, we have $\alpha+\phi_0(\langle \beta_i \rangle)=\gamma+\phi_{n+1}(\delta'+1)$. Therefore, $\langle \alpha+\phi_0(\beta_i) \rangle_i=\langle \gamma+\phi^i_{n}(\phi_{n+1}(\delta')+1) \rangle_i$. For $i=0$, we have $\alpha+\phi_0(\beta_0) = \gamma+\phi_{n+1}(\delta')+1$. Therefore, by the induction hypothesis, $\alpha = \gamma+\phi_{n+1}(\delta')$. For $i=1$, we get $\alpha+\phi_0(\beta_1) = \gamma+\phi_n(\phi_{n+1}(\delta')+1)$. Therefore, by the induction hypothesis, $\alpha = \gamma$. Hence, $\alpha = \alpha+\phi_{n+1}(\delta')$ which is impossible with a similar reasoning as in the previous case.

\item[] $\bullet$ If $\beta=\langle \beta_i \rangle$ and $\delta=\langle \delta_i \rangle$ are limits, we have $\alpha+\phi_0(\langle \beta_i \rangle)=\gamma+\phi_{n+1}(\langle \delta_i \rangle)$. Therefore, $\langle \alpha+\phi_0(\beta_i) \rangle_i=\langle \gamma+\phi_{n+1}(\delta_i) \rangle_i$. For $i=0$, we have $\alpha+\phi_0(\beta_0) = \gamma+\phi_{n+1}(\delta_0)$ which by the induction hypothesis implies $n+1=0$ which is impossible.\\

\noindent III. The case $k>0$ and $l=0$ is identical to the case II by symmetry.\\

\noindent IV. If $k=m+1$ and $l=n+1$, for some $m, n \in \mathbb{N}$, there are again nine cases to consider depending if $\beta$ and $\delta$ are zero, a successor or a limit:  

\item[] $\bullet$ If $\beta=\delta=0$, we have $\alpha+\phi_{m+1}(0)=\gamma+\phi_{n+1}(0)$ which implies $\langle \alpha+\phi_m^i(0) \rangle_i=\langle \gamma+\phi_n^i(0) \rangle_i$. For $i=1$, we have $\alpha+\phi_m(0)=\gamma+\phi_n(0)$. As $\alpha+\phi_m(0) \prec \rho$, by the induction hypothesis we have $\alpha=\gamma$ and $m=n$ which implies $k=l$.

\item[] $\bullet$ If $\beta=0$ and $\delta=\delta'+1$ is a successor, we have $\alpha+\phi_{m+1}(0)=\gamma+\phi_{n+1}(\delta'+1)$ which implies $\langle \alpha+\phi_m^i(0) \rangle_i=\langle \gamma+\phi_n^i(\phi_{n+1}(\delta')+1) \rangle_i$.
For $i=1$, we reach $\alpha+\phi_m(0)=\gamma+\phi_n(\phi_{n+1}(\delta')+1)$ which by the induction hypothesis implies $0=\phi_{n+1}(\delta')+1$ which is a contradiction.

\item[] $\bullet$ If $\beta = 0$ and $\delta = \langle \delta_i \rangle_i$ is a limit, then we have  
$\alpha + \phi_{m+1}(0) = \gamma + \phi_{n+1}(\langle \delta_i \rangle)$,  
which implies that  
$\langle \alpha + \phi_m^i(0) \rangle_i = \langle \gamma + \phi_{n+1}(\delta_i) \rangle_i$.  
For $i = 0$, this implies $\alpha = \gamma + \phi_{n+1}(\delta_0)$. 
For $i = 1$, we obtain  
$\alpha + \phi_m(0) = \gamma + \phi_{n+1}(\delta_1)$.  
By the induction hypothesis, it follows that $\alpha = \gamma$, which further implies  
$\alpha = \alpha + \phi_{n+1}(\delta_0)$.  
However, this is a contradiction, using a similar argument we employed in previous cases.

\item[] $\bullet$ If $\beta = \beta' + 1$ is a successor and $\delta = 0$, the case is similar to the previous ones, by symmetry between $\beta$ and $\delta$.

\item[] $\bullet$ If $\beta = \beta' + 1$ and $\delta = \delta' + 1$ are successors, we have  
$\alpha + \phi_{m+1}(\beta' + 1) = \gamma + \phi_{n+1}(\delta' + 1)$, which implies  
$\langle \alpha + \phi_m^i(\phi_{m+1}(\beta') + 1) \rangle_i = \langle \gamma + \phi_n^i(\phi_{n+1}(\delta') + 1) \rangle_i$.  
For $i = 0$, we have  
$\alpha + \phi_{m+1}(\beta') + 1 = \gamma + \phi_{n+1}(\delta') + 1$, which implies  
$\alpha + \phi_{m+1}(\beta') = \gamma + \phi_{n+1}(\delta')$.  
By the induction hypothesis, we reach $\alpha = \gamma$, $m + 1 = n + 1$, and $\beta' = \delta'$,  
which implies $\beta = \delta$.

\item[] $\bullet$ If $\beta = \beta' + 1$ is a successor and $\delta = \langle \delta_i \rangle$ is a limit, we have  
$\alpha + \phi_{m+1}(\beta' + 1) = \gamma + \phi_{n+1}(\langle \delta_i \rangle)$,  
which implies  
$\langle \alpha + \phi_m^i(\phi_{m+1}(\beta') + 1) \rangle_i = \langle \gamma + \phi_{n+1}(\delta_i) \rangle_i$.  
For $i = 0$, we have  
$\alpha + \phi_{m+1}(\beta') + 1 = \gamma + \phi_{n+1}(\delta_0)$.  
As the left-hand side is a successor and the right-hand side is a limit by Lemma \ref{phiIsNonZero}, this is a contradiction.

\item[] $\bullet$ If $\beta = \langle \beta_i \rangle$ is a limit and $\delta$ is either $0$ or a successor, the case is similar to the previous ones, using the symmetry between $\beta$ and $\delta$.

\item[] $\bullet$ If $\beta = \langle \beta_i \rangle$ and $\delta = \langle \delta_i \rangle$ are limits, we have  
$\alpha + \phi_{m+1}(\langle \beta_i \rangle) = \gamma + \phi_{n+1}(\langle \delta_i \rangle)$,  
which implies $\langle \alpha + \phi_{m+1}(\beta_i) \rangle_i = \langle \gamma + \phi_{n+1}(\delta_i) \rangle_i$.  
Therefore, for any $i \in \mathbb{N}$, we have  
$\alpha + \phi_{m+1}(\beta_i) = \gamma + \phi_{n+1}(\delta_i)$,  
which by the induction hypothesis implies $\alpha = \gamma$, $m + 1 = n + 1$, and $\beta_i = \delta_i$, for any $i \in \mathbb{N}$.  
Hence, $\beta = \delta$.
\end{proof}

Using the previous lemmas, we are now ready to prove the following uniqueness theorem. We will use this theorem later for the finitary representation of ordinals in $\Phi_{\omega}$.

\begin{theorem}[Unique Representation]\label{UniquenessTheorem}
Let $r, s \geq 1$ be natural numbers, $\{k_i\}_{i=1}^r$ and $\{l_i\}_{i=1}^s$ be sequences of natural numbers, and $\{\alpha_i\}_{i=1}^r$ and $\{\beta_i\}_{i=1}^s$ be sequences of ordinals. Then:
\begin{itemize}
    \item[$(i)$] 
$\sum_{i=1}^r \phi_{k_i}(\alpha_i) \neq 0$.
    \item[$(ii)$] 
If $\sum_{i=1}^r \phi_{k_i}(\alpha_i) = \sum_{i=1}^s \phi_{l_i}(\beta_i)$, then $r = s$ and for every $1 \leq i \leq r = s$, we have  
$k_i = l_i$ and $\alpha_i = \beta_i$.
\end{itemize}
\end{theorem}
\begin{proof}
For $(i)$, if $\sum_{i=1}^r \phi_{k_i}(\alpha_i) = 0$, then by Lemma~\ref{AdditionProp}, $\phi_{k_i}(\alpha_i) = 0$ for any $1 \leq i \leq r$. As $r \geq 1$, this implies the existence of an $i$ such that $\phi_{k_i}(\alpha_i) = 0$, which is impossible by Lemma~\ref{phiIsNonZero}.

For $(ii)$, by induction on $\rho \in \Omega$, we prove that for any $r, s \in \mathbb{N}$, any sequences $\{k_i\}_{i=1}^r$ and $\{l_i\}_{i=1}^s$ of numbers, and any sequences $\{\alpha_i\}_{i=1}^r$ and $\{\beta_i\}_{i=1}^s$ of ordinals, if  $\rho = \sum_{i=1}^r \phi_{k_i}(\alpha_i) = \sum_{i=1}^s \phi_{l_i}(\beta_i)$,  
then $r = s$ and for any $1 \leq i \leq r = s$, we have $k_i = l_i$ and $\alpha_i = \beta_i$. Assume the claim holds for every ordinal below $\rho$ and let  
$\rho = \sum_{i=1}^r \phi_{k_i}(\alpha_i) = \sum_{i=1}^s \phi_{l_i}(\beta_i)$.  
There are four cases to consider, depending on whether $r$ and $s$ are equal to one or greater than one:

If $r, s > 1$, then, using Lemma \ref{UniquenessLemma}, we have  $k_r = l_s$, $\alpha_r = \beta_s$, and $\sum_{i=1}^{r-1} \phi_{k_i}(\alpha_i) = \sum_{i=1}^{s-1} \phi_{l_i}(\beta_i)$.  
Since $\phi_{k_r}(\alpha_r)\neq 0$ (Lemma~\ref{phiIsNonZero}), from Lemma \ref{AdditionProp} it follows that
$\sum_{i=1}^{r-1} \phi_{k_i}(\alpha_i) \prec \rho$ so  
the induction hypothesis implies that $r-1 = s-1$ and for every $1 \leq i \leq r-1$, we have $k_i = l_i$ and $\alpha_i = \beta_i$. Using $r = s$, it follows that $\alpha_r = \beta_r$ and $k_r = l_r$, completing the proof.

If $r = 1$ and $s > 1$, then by Lemma \ref{AdditionProp},  
$\phi_{k_1}(\alpha_1) = 0 + \phi_{k_1}(\alpha_1) = \sum_{i=1}^s \phi_{l_i}(\beta_i)$.  
By Lemma \ref{UniquenessLemma}, we have  
$k_1 = l_s$,  
$\alpha_1 = \beta_s$,  
and  
$0 = \sum_{i=1}^{s-1} \phi_{l_i}(\beta_i)$ which is impossible by part $(i)$. 

The case $r > 1$ and $s = 1$ is similar to the previous one.
If $r = s = 1$, the claim follows from Lemma \ref{UniquenessLemma}.
\end{proof}

\begin{remark}
Theorem~\ref{UniquenessTheorem} may appear counter-intuitive, since for set-theoretic ordinals, uniqueness holds only for representations in \emph{normal form}, as presented in Theorem~\ref{PropertiesOfRealVeblen}. For instance, $\bm{n} + \bm{\omega} = \bm{\omega}$ for any $n \in \mathbb{N}$, which contradicts the kind of uniqueness we claim. In fact, the requirement of normal form is precisely intended to prevent such situations. Despite this counter-intuitiveness, the uniqueness claim does hold for constructive ordinals, which, as trees, are inherently more rigid. The key point is that a limit constructive ordinal $\alpha$ is defined as a fixed sequence of ordinals $\langle \alpha_i \rangle$, rather than as an ordinal that is neither zero nor a successor. Having such sequences explicitly present makes equality between constructive ordinals significantly stricter. As an example of this stricter notion of equality, recall that the identity $n + \omega = \omega$ fails if $n \neq 0$.
\end{remark}

\subsection{Consequences of unique representation}

In this subsection, we make use of the uniqueness of representation (Theorem~\ref{UniquenessTheorem}) to establish several properties of the classes $\Phi^m_k$ and $\Psi^m_k$.

\begin{corollary}\label{Summand}
Let $\alpha, \beta \in \Phi_{\omega}$ be ordinals and $k \in \mathbb{N}$ and $m \in \mathbb{N}^{\geq 1}$ be natural numbers. Then:
\begin{itemize}
    \item[$(i)$] 
If $\alpha+\beta \in \Phi^m_k$ then $\alpha, \beta \in \Phi^m_k$.
    \item[$(ii)$] 
If $\phi_l(\alpha) \in \Phi^m_k$ then either $l=k$, $m \geq 2$ and $\alpha \in \Phi^{m-1}_k$, or $l<k$ and $\alpha \in \Phi^{m}_k$.  
\end{itemize}
\end{corollary}
\begin{proof}
The proof is straightforward.
\end{proof}

\begin{corollary}\label{Properness}
\begin{itemize}
    \item[$(i)$] 
Let $k, l, n \in \mathbb{N}$ and $m \in \mathbb{N}^{\geq 1}$ be natural numbers such that either $l > k$, or $l = k$ and $n \geq m$. Then, $\phi_l^n(0) \notin \Phi_k^m$.
    \item[$(ii)$] 
We have $\Phi_k^m \subsetneq \Phi_k^{m+1}$, $\Psi_{\omega} \subsetneq \Phi_1$, and $\Phi_k \subsetneq \Phi_{k+1}$ for any $k \in \mathbb{N}$ and any $m \in \mathbb{N}^{\geq 1}$.
\end{itemize} 
\end{corollary}
\begin{proof}
For $(i)$, assume $\phi^n_l(0) \in \Phi^m_k$ to reach a contradiction. By Corollary~\ref{Summand}, since $n \geq m \geq 1$ and $\phi^n_l(0)=\phi_l(\phi_l^{n-1}(0)) \in \Phi^m_k$, we must have $l \leq k$. However, by assumption, either $l > k$, or $l = k$ and $n \geq m$. Hence, we must be in the second case, i.e., $l = k$ and $n \geq m$. Thus, $\phi^n_k(0) \in \Phi^m_k$. 
Now, as $n \geq m$, by repeated application of Corollary~\ref{Summand}, we obtain $\phi^{n-m+1}_k(0) \in \Phi^1_k = \Phi_k$. 
If $k = 0$, since $n \geq m$, by Lemma~\ref{phiIsNonZero} we have $\phi^{n-m+1}_0(0) = \phi_0(\phi^{n-m}_0(0)) \notin \Phi_0 = \{0\}$, which yields a contradiction.
If $k > 0$, then by Lemma~\ref{MonotonicityOfPhi} we know that $\Phi_k = \bigcup_{r=1}^{\infty} \Phi_{k-1}^r$. Hence, there exists $r \geq 1$ such that $\phi^{n-m+1}_k(0) \in \Phi_{k-1}^r$. Since $n - m + 1 \geq 1$, we have $\phi^{n-m+1}_k(0)=\phi_k(\phi^{n-m}_k(0)) \in \Phi_{k-1}^r$. Therefore, by Corollary~\ref{Summand}, we obtain $k \leq k - 1$, which is a contradiction.

For $(ii)$, to show that $\Phi_k^m \subsetneq \Phi^{m+1}_k$, note that by part~$(i)$ we have $\phi_k^m(0) \notin \Phi^m_k$, while it is clear that $\phi_k^m(0) \in \Phi^{m+1}_k$. 
To show that $\Psi_{\omega} \subsetneq \Phi_1$, observe that by Lemma~\ref{MonotonicityOfPsi}, we have $\Psi_{\omega} = \Phi_0^3$. Therefore, by part~$(i)$, $\phi_0^3(0) \notin \Phi_0^3=\Psi_{\omega}$, while it is clear that $\phi_0^3(0) \in \Phi_1$.
Finally, to show that $\Phi_k \subsetneq \Phi_{k+1}$, observe that, by part~$(i)$, we have $\phi_k(0) \notin \Phi^1_k = \Phi_k$, while trivially $\phi_k(0) \in \Phi_{k+1}$.
\end{proof}

\begin{remark}\label{LengthBoundForPsi}
As another consequence of Theorem~\ref{UniquenessTheorem}, we can show that any non-zero $\alpha \in \Psi_{\omega}$ has a unique representation in the form $\alpha = \sum_{i=1}^l \omega^{p_i} c_i$, where $l \geq 1$, the $c_i$'s are non-zero finite ordinals, and $p_i \neq p_{i+1}$ for any $1 \leq i < l$. We call this representation the \emph{normal form} of $\alpha \in \Psi_{\omega}$. Moreover, using this uniqueness result, one can easily see that if $\alpha \in \Psi^m_k$ and its normal form is $\alpha = \sum_{i=1}^l \omega^{p_i} c_i$, then $l \leq m + k$. 
\end{remark}

Using Corollary \ref{Summand}, we are now ready to prove that $\Phi^m_k$ and $\Psi^m_k$ are downsets.

\begin{lemma}\label{PhiIsDownset}
The following hold:
\begin{itemize}
    \item[$(i)$] 
    $\Phi^m_k$ is a downset for any $k \geq 0$ and $m \geq 1$. Hence, $\Phi_k$ and $\Phi_{\omega}$ are downsets for all $k \geq 0$.
    \item[$(ii)$]
    $\Psi^m_k$ is a downset for any $k \geq 1$ and $m \geq 1$. Hence, $\Psi_k$ and $\Psi_{\omega}$ are downsets for all $k \geq 1$.
\end{itemize}
\end{lemma}
\begin{proof}
First, note that since the union of downsets is itself a downset, by Lemma~\ref{MonotonicityOfPhi} and Lemma~\ref{MonotonicityOfPsi}, it suffices to prove the first parts of $(i)$ and $(ii)$.
Recall the predecessor function~$p$ defined in  Section~\ref{sec:functions-const-ordinals}.
For~$(i)$, it is enough to prove that for any $\alpha \in \Phi_{\omega}$, if $\alpha \in \Phi^m_k$, then $p(\alpha, n) \in \Phi^m_k$ for any $n \in \mathbb{N}$. Define $X$ as the subset of $\Phi_{\omega}$ consisting of those ordinals $\alpha$ such that for any $k \in \mathbb{N}$ and any $m \in \mathbb{N}^{\geq 1}$, if $\alpha \in \Phi^m_k$, then $p(\alpha, n) \in \Phi^m_k$ for all $n \in \mathbb{N}$. We show that $X = \Phi_{\omega}$. To that end, it is enough to prove that $X$ contains $0$ and is closed under addition and $\phi_l$ for any $l \in \mathbb{N}$.

First, it is clear that $0 \in X$, as $p(0, n) = 0 \in \Phi^m_k$, for any $k \in \mathbb{N}$ and any $m \in \mathbb{N}^{\geq 1}$. For the closure under addition, let $\alpha, \beta \in X$. If $\beta = 0$, then $\alpha + \beta = \alpha \in X$. Hence, we assume that $\beta \neq 0$. Now, let $k \in \mathbb{N}$ and $m \in \mathbb{N}^{\geq 1}$ be arbitrary natural numbers and assume that $\alpha + \beta \in \Phi^m_k$ to show that $p(\alpha+\beta, n) \in \Phi^m_k$, for any $n \in \mathbb{N}$. As $\alpha, \beta \in X \subseteq \Phi_{\omega}$ and $\alpha + \beta \in \Phi^m_k$, by Corollary~\ref{Summand}, we have $\alpha, \beta \in \Phi^m_k$. As $\beta \neq 0$, we have $p(\alpha + \beta, n) = \alpha + p(\beta, n)$. On the other hand, using $\beta \in X$ and $\beta \in \Phi^m_k$, we obtain $p(\beta, n) \in \Phi^m_k$. As $\Phi^m_k$ is closed under addition and $\alpha \in \Phi^m_k$, we conclude that $p(\alpha + \beta, n) = \alpha + p(\beta, n) \in \Phi^m_k$.

For the closure of $X$ under $\phi_l$, let $\alpha \in X$ and $k \in \mathbb{N}$ and $m \in \mathbb{N}^{\geq 1}$ be arbitrary natural numbers. We assume $\phi_l(\alpha) \in \Phi^m_k$ and aim to prove that $p(\phi_l(\alpha), n) \in \Phi^m_k$ for arbitrary $n \in \mathbb{N}$. By Corollary~\ref{Summand}, since $\alpha \in X \subseteq \Phi_{\omega}$ and $\phi_l(\alpha) \in \Phi^m_k$, either $l = k$, $m \geq 2$ and $\alpha \in \Phi^{m-1}_k$, or $l < k$ and $\alpha \in \Phi^m_k$. We now consider four cases:

\item[]  
I. Assume that $l=0$ and $k=0$. Hence, $\alpha \in \Phi^{m-1}_0$. If $\alpha=0$, then $\phi_0(0)=1$ and hence $p(\phi_0(0))=0 \in \Phi^m_k$. If $\alpha=\alpha'+1$ is a successor, $\phi_0(\alpha)=\langle \phi_0(\alpha') \cdot (i+1) \rangle$. As $\alpha \in X$ and $\alpha \in \Phi^{m-1}_0$, we have $\alpha'=p(\alpha, 0) \in \Phi^{m-1}_0$. Therefore, for any $n \in \mathbb{N}$, we have $p(\phi_0(\alpha), n)=\phi_0(\alpha') \cdot (n+1) \in \Phi^m_0$. If $\alpha=\langle \alpha_i \rangle$ is a limit, as $\alpha \in X$ and $\alpha \in \Phi_0^{m-1}$, we have $\alpha_n=p(\alpha, n) \in \Phi_0^{m-1}$, for any $n \in \mathbb{N}$. Therefore, $p(\phi_0(\alpha), n)=\phi_0(\alpha_n) \in \Phi^m_0$, for any $n \in \mathbb{N}$.

\item[] 
II. Assume that $l=0$ and $k>0$. Hence, $\alpha \in \Phi^{m}_k$. The case $\alpha=0$ is similar to the one above. If $\alpha=\alpha'+1$ is a successor, $\phi_0(\alpha)=\langle \phi_0(\alpha') \cdot (i+1) \rangle$. As $\alpha \in X$ and $\alpha \in \Phi^{m}_k$, we have $\alpha'=p(\alpha, 0) \in \Phi^{m}_k$. As $l=0<k$, we reach $p(\phi_0(\alpha), n)=\phi_0(\alpha') \cdot (n+1) \in \Phi^m_k$, for any $n \in \mathbb{N}$.
If $\alpha=\langle \alpha_i \rangle$ is a limit, as $\alpha \in X$ and $\alpha \in \Phi_k^{m}$, we have $\alpha_n=p(\alpha, n) \in \Phi_k^{m}$, for any $n \in \mathbb{N}$. Therefore, $p(\phi_0(\alpha), n)=\phi_0(\alpha_n) \in \Phi^m_k$, for any $n \in \mathbb{N}$.

\item[] 
III. Assume that $l \neq 0$ and $l=k$. Hence, $\alpha \in \Phi^{m-1}_k$. If $\alpha=0$, then $\phi_{k}(0)=\langle \phi^{(i)}_{k-1}(0) \rangle$. As $\phi_{k-1}^{(n)}(0) \in \Phi^m_k$, for any $n \in \mathbb{N}$, there is nothing to prove. If $\alpha=\alpha'+1$ is a successor, as $\alpha \in X$ and $\alpha \in \Phi^{m-1}_k$, then $\alpha'=p(\alpha, 0) \in \Phi^{m-1}_k$. Thus, $p(\phi_k(\alpha), n)=\phi_{k-1}^{(n)}(\phi_k(\alpha')+1) \in \Phi^m_k$, for any $n \in \mathbb{N}$. If $\alpha=\langle \alpha_i \rangle$ is a limit, as $\alpha \in X$ and $\alpha \in \Phi^{m-1}_k$, we have $\alpha_n \in \Phi^{m-1}_k$, for any $n \in \mathbb{N}$. Therefore, $p(\phi_k(\alpha), n)=\phi_k(\alpha_n) \in \Phi^{m}_k$, for any $n \in \mathbb{N}$.

\item[]
IV. Assume that $l \neq 0$ and $l<k$. Hence, $\alpha \in \Phi^{m}_k$. If $\alpha=0$, then $\phi_{l}(0)=\langle \phi^{(i)}_{l-1}(0) \rangle$. As $l<k$ and hence $\phi_{l-1}^{(n)}(0) \in \Phi^m_k$, for any $n \in \mathbb{N}$, there is nothing to prove. If $\alpha=\alpha'+1$ is a successor, as $\alpha \in X$ and $\alpha \in \Phi^{m}_k$, then $\alpha'=p(\alpha, 0) \in \Phi^{m}_k$. Thus, as $l < k$, we reach $p(\phi_l(\alpha), n)=\phi_{l-1}^{(n)}(\phi_l(\alpha')+1) \in \Phi^m_k$, for any $n \in \mathbb{N}$. If $\alpha=\langle \alpha_i \rangle$ is a limit, as $\alpha \in X$ and $\alpha \in \Phi^{m}_k$, we have $\alpha_n \in \Phi^{m}_k$, for any $n \in \mathbb{N}$. Therefore, as $l < k$, we reach $p(\phi_l(\alpha), n)=\phi_l(\alpha_n) \in \Phi^{m}_k$, for any $n \in \mathbb{N}$. This completes the proof of $(i)$.

For $(ii)$, it is enough to prove that if $\alpha \in \Psi^m_k$ then $p(\alpha, n) \in \Psi^m_k$, for any $n \in \mathbb{N}$. As $\alpha \in \Psi^m_k$, either $\alpha=0$ where there is nothing to prove, or it can be written in the form $\alpha=\sum_{i=1}^m \omega^{p_i} c_i + \sum_{j=0}^{k-1} \omega^{k-1-j}d_j$, for some $p_i < k$ and some finite ordinals $c_i$ and $d_j$. There are two cases to consider: either there is $0 \leq j \leq k-1$ such that $d_j \neq 0$ or $\alpha=\sum_{i=1}^m \omega^{p_i} c_i$. In the first case, let $0 \leq l \leq j$ be the greatest number such that $d_l \neq 0$. Therefore, $\alpha=\sum_{i=1}^m \omega^{p_i} c_i + \sum_{j=0}^{l} \omega^{k-1-j}d_j$. If $l=k-1$, we have $\alpha=\sum_{i=1}^m \omega^{p_i} c_i + \sum_{j=0}^{l-1} \omega^{k-1-j}d_j+d_l$ which implies that 
\[
p(\alpha, n)=\sum_{i=1}^m \omega^{p_i} c_i + \sum_{j=0}^{l-1} \omega^{k-1-j}d_j+d_l-1.
\]
which is clearly in $\Psi^m_k$. If $l<k-1$, then we have:
\[
p(\alpha, n)=\sum_{i=1}^m \omega^{p_i} c_i + \sum_{j=0}^{l-1} \omega^{k-1-j}d_j+\omega^{k-1-l}(d_l-1)+\omega^{k-1-l-1}(n+1)
\]
which is clearly in $\Psi^m_k$. In the second case, we have $\alpha=\sum_{i=1}^m \omega^{p_i} c_i$. Let $1 \leq l \leq m$ be the greatest number such that $c_l \neq 0$. Therefore, $\alpha=\sum_{i=1}^l \omega^{p_i} c_i$. If $p_l=0$ then $\alpha=\sum_{i=1}^{l-1} \omega^{p_i} c_i+c_l$ and hence 
\[
p(\alpha, n)=\sum_{i=1}^{l-1} \omega^{p_i} c_i + (c_l-1)
\]
which is clearly in $\Psi^m_k$, as $l \leq m$. If $p_l \geq 1$, then we have:
\[
p(\alpha, n)=\sum_{i=1}^{l-1} \omega^{p_i} c_i + \omega^{p_l}(c_l-1)+\omega^{p_l-1}(n+1)
\]
which is clearly in $\Psi^m_k$ as $p(\alpha, n)$ is the addition of $\sum_{i=1}^{l-1} \omega^{p_i} c_i + \omega^{p_l}(c_l-1)$ with $l \leq m$ summands and $\omega^{p_l-1}(n+1)$, where $p_l-1 < k$.
\end{proof}

\subsection{Dichotomy theorem}

In this subsection, we establish a dichotomy theorem stating that, for any $\alpha \in \Phi_\omega$, there exists $r \in \mathbb{N}$ such that either $\alpha = o(r)$ or $o(r) + \omega \preceq \alpha$. For that purpose, we first need to prove two lemmas.

\begin{lemma}\label{NonZeroPoint}
For any $\langle \alpha_i \rangle \in \Phi_{\omega}$, there is $i \in \mathbb{N}$ such that $\alpha_i \neq 0$.    
\end{lemma}
\begin{proof}
Let $X$ be the set of all ordinals $\alpha \in \Phi_{\omega}$ such that either $\alpha$ is $0$, a successor, or a limit $\langle \alpha_i \rangle \in \Phi_{\omega}$ with $\alpha_i \neq 0$ for some $i \in \mathbb{N}$. To prove $X = \Phi_{\omega}$, it is enough to show that $X$ contains $0$ and is closed under $\phi_k$'s and addition. Clearly, $0 \in X$. 

To prove the closure under addition, assume $\alpha, \beta \in X$. There are three cases to consider, depending on whether $\beta$ is zero, a successor, or a limit.  
If $\beta = 0$, then $\alpha + \beta = \alpha \in X$, so there is nothing to prove.  
If $\beta$ is a successor, then $\alpha + \beta$ is also a successor, and hence $\alpha + \beta \in X$.  
If $\beta = \langle \beta_i \rangle$ is a limit, then $\alpha + \beta = \langle \alpha + \beta_i \rangle$. Since $\beta \in X$, there exists $i \in \mathbb{N}$ such that $\beta_i \neq 0$. Therefore, by Lemma \ref{AdditionProp}, we have $\alpha + \beta_i \neq 0$, which implies $\alpha + \beta \in X$.

To prove the closure under $\phi_k$, we establish the stronger claim that $\phi_k(\alpha) \in X$ for any $\alpha \in \Omega$. We consider two separate cases: either $k = 0$ or $k \neq 0$. 
If $k = 0$, we analyze three subcases depending on whether $\alpha$ is zero, a successor, or a limit.  
If $\alpha = 0$, then $\phi_0(\alpha) = 1$, which implies $\phi_0(\alpha) \in X$.  
If $\alpha = \alpha' + 1$ is a successor, then $\phi_0(\alpha' + 1) = \phi_0(\alpha') \cdot \omega = \langle \phi_0(\alpha') \cdot (i + 1) \rangle$. Taking $i = 0$, by Lemma \ref{phiIsNonZero}, we have $\phi_0(\alpha') \neq 0$, which shows that $\phi_0(\alpha) = \phi_0(\alpha' + 1) \in X$.  
If $\alpha = \langle \alpha_i \rangle$ is a limit, then $\phi_0(\alpha) = \langle \phi_0(\alpha_i) \rangle$. Again, by Lemma \ref{phiIsNonZero}, we have $\phi_0(\alpha_0) \neq 0$, which implies that $\phi_0(\alpha) \in X$.

For $k > 0$, there exists $l \in \mathbb{N}$ such that $k = l + 1$. We must prove that $\phi_{l+1}(\alpha) \in X$ for any $\alpha \in \Omega$. We consider three cases depending on whether $\alpha$ is zero, a successor, or a limit. 
If $\alpha = 0$, we have $\phi_{l+1}(0) = \langle \phi_l^{(i)}(0) \rangle$. For $i = 1$, by Lemma \ref{phiIsNonZero}, we have $\phi_l(0) \neq 0$, which implies $\phi_{l+1}(0) \in X$.  
If $\alpha = \alpha' + 1$ is a successor, then $\phi_{l+1}(\alpha) = \langle \phi_l^{(i)}(\phi_{l+1}(\alpha') + 1) \rangle$. Pick $i = 0$. Since $\phi_{l+1}(\alpha') + 1 \neq 0$, we have $\phi_{l+1}(\alpha) \in X$.  
If $\alpha = \langle \alpha_i \rangle$ is a limit, then $\phi_{l+1}(\alpha) = \langle \phi_{l+1}(\alpha_i) \rangle$. For $i = 0$, since $\phi_{l+1}(\alpha_0) \neq 0$, by Lemma \ref{phiIsNonZero}, we have $\phi_{l+1}(\alpha) \in X$.
\end{proof}

\begin{lemma}\label{FiniteOrInfiniteLemma}
Let $\alpha \in \Phi_{\omega}$ and $k \in \mathbb{N}$. 
If $k \neq 0$ or $\alpha \neq 0$ then $\omega \preceq \phi_{k}(\alpha)$.
\end{lemma}
\begin{proof}
First, we prove the weaker statement that if $\alpha \neq 0$ then $\omega \preceq \phi_{0}(\alpha)$. For that purpose, we use induction on $\alpha$. For $\alpha = 0$, there is nothing to prove.  
If $\alpha = \alpha' + 1$ is a successor, there are two cases to consider: either $\alpha' = 0$ or $\alpha' \neq 0$.  
If $\alpha' = 0$, then $\alpha = 1$ and hence $\phi_0(\alpha) = \phi_0(1) = \omega$. Thus, $\omega \preceq \phi_0(\alpha)$.  
If $\alpha' \neq 0$, then by the induction hypothesis, we have  
\[
\omega \preceq \phi_0(\alpha') \prec \langle \phi_0(\alpha')(i + 1) \rangle_i = \phi_0(\alpha') \cdot \omega = \phi_0(\alpha' + 1)=\phi_0(\alpha).
\]
If $\alpha = \langle \alpha_i \rangle$ is a limit, then since $\alpha \in \Phi_{\omega}$, by Lemma \ref{NonZeroPoint}, there exists $i \in \mathbb{N}$ such that $\alpha_i \neq 0$. Therefore, by the induction hypothesis, we have  
\[
\omega \preceq \phi_0(\alpha_i) \prec \langle \phi_0(\alpha_i) \rangle_i = \phi_0(\alpha).
\]
This completes the proof of the weaker statement. 

To prove the main statement, we use induction on $k$. For $k=0$, the claim follows from the weaker statement. For the induction step, we assume the claim holds for $\phi_k$ and, by induction on $\alpha$, show that $\omega \preceq \phi_{k+1}(\alpha)$.  
If $\alpha = 0$, by the induction hypothesis for $\phi_k$ and the fact that $\phi_k(0) \neq 0$ from Lemma \ref{phiIsNonZero}, we have  
$
\omega \preceq \phi_k(\phi_k(0)) = \phi_k^2(0),
$  
which implies  
$\omega \preceq \langle \phi_k^{i}(0) \rangle_i = \phi_{k+1}(0)$.  
If $\alpha = \alpha' + 1$ is a successor, as $\phi_{k+1}(\alpha' + 1)=\langle \phi_k^i(\phi_{k+1}(\alpha') + 1) \rangle$, we reach $\phi_{k+1}(\alpha') + 1 \prec \phi_{k+1}(\alpha' + 1) $.
Now, by the induction hypothesis for $\alpha$, we have
\[
\omega \preceq \phi_{k+1}(\alpha') \prec \phi_{k+1}(\alpha') + 1 \prec \phi_{k+1}(\alpha' + 1) = \phi_{k+1}(\alpha).
\]  
If $\alpha = \langle \alpha_i \rangle$ is a limit, by the induction hypothesis for $\alpha$, we have
\[
\omega \preceq \phi_{k+1}(\alpha_0) \prec \langle \phi_{k+1}(\alpha_i) \rangle=\phi_{k+1}(\alpha),
\]  
which completes the proof.
\end{proof}

\begin{theorem}[Dichotomy Theorem]\label{FiniteOrInfiniteThm}
For any $\alpha \in \Phi_{\omega}$, there is $r \in \mathbb{N}$ such that either $\alpha=o(r)$ or $o(r)+\omega \preceq \alpha$. 
\end{theorem}
\begin{proof}
Let $X$ be the set of all ordinals $\alpha \in \Phi_{\omega}$ such that either there is $r \in \mathbb{N}$ with $\alpha = o(r)$ or $o(r) + \omega \preceq \alpha$. To prove $X = \Phi_{\omega}$, it is enough to prove that $X$ contains $0$ and is closed under addition and $\phi_k$'s. Clearly, $0 = o(0) \in X$.  
To prove the closure under $\phi_k$, let $\alpha \in X$. If $\alpha = 0$ and $k = 0$, then $\phi_k(\alpha) = o(1)$, which implies $\phi_k(\alpha) \in X$. Otherwise, by Lemma \ref{FiniteOrInfiniteLemma}, we have $\omega \preceq \phi_k(\alpha)$. Therefore, $\phi_k(\alpha) \in X$. 

For the closure under addition, if $\alpha, \beta \in X$, then there exists $r \in \mathbb{N}$ such that either $\alpha = o(r)$ or $o(r) + \omega \preceq \alpha$. Similarly, there exists $s \in \mathbb{N}$ such that either $\beta = o(s)$ or $o(s) + \omega \preceq \beta$.  
If $o(r) + \omega \preceq \alpha$, then since $\alpha \preceq \alpha + \beta$, by Lemma \ref{AdditionProp}, we get $o(r) + \omega \preceq \alpha + \beta$, which implies $\alpha + \beta \in X$. Therefore, we assume $\alpha = o(r)$. Now, if $\beta = o(s)$, then $\alpha + \beta = o(r + s)$, which implies $\alpha + \beta \in X$. Otherwise, 
assume that $o(s) + \omega \preceq \beta$
and note that $\alpha + \beta = o(r) + \beta$. By Lemma \ref{AdditionProp}, we have  
$
o(r) + o(s) + \omega \preceq o(r) + \beta.
$  
Hence,  
$
o(r + s) + \omega \preceq \alpha + \beta,
$  
which implies $\alpha + \beta \in X$. 
\end{proof}

\section{Main Theorem}
\label{sec:main-theorem}
In this section, we present our main theorem together with its corollaries. We also outline the strategy of the proof, which will be carried out in detail in the subsequent sections of the paper.

For any $\alpha \in \Omega$ and $k \geq 0$, recall that the downset $\mathsf{D}_{\alpha}$ is defined as
$
\mathsf{D}_{\alpha} := \{ \beta \in \Omega \mid \beta \prec \alpha \}.
$
Also recall that the set $\Phi_{\omega}$ (resp. $\Phi_k$) is defined in Definition~\ref{Phi-sets} as the smallest set of constructive ordinals containing $0$ and closed under addition and the Veblen functions $\phi_i$ (Definition~\ref{dfn:VeblenHierarchy}), for all $i \in \mathbb{N}$ (resp. all $i < k$).
\begin{definition}\label{definitionOfl}
A downset $\mathsf{A} \subseteq \Phi_{\omega}$ of ordinals is called \emph{bounded} if there exists $k \geq 1$ such that $\mathsf{A} \subseteq \Phi_k$. For any bounded downset $\mathsf{A}$, define $l(\mathsf{A}) = 0$ if $\mathsf{A} \subseteq \Psi_{\omega}$. Otherwise, define $l(\mathsf{A})$ to be the least $k \geq 1$ such that $\mathsf{A} \subseteq \Phi_k$.
\end{definition}

\begin{example}\label{ExamplesOfLevels}
For the first example, observe that $\Psi_{\omega} \subseteq \Phi_1$ is clearly bounded. Moreover, it is trivial that $l(\Psi_{\omega}) = 0$. As a second example, $\Phi_k$ is trivially bounded for any $k \in \mathbb{N}$.  
To compute $l(\Phi_k)$, note that if $k = 0$, then $\Phi_0 = \{0\} \subseteq \Psi_{\omega}$. Hence, $l(\Phi_0) = 0$.  
For $k \geq 1$, by Corollary~\ref{Properness}, we have $\Phi_k \nsubseteq \Phi_{k-1}$ and $\Phi_1 \nsubseteq \Psi_{\omega}$. It follows that $\Phi_k \nsubseteq \Psi_{\omega}$, and therefore $l(\Phi_k) = k$. Moreover, note that $\Phi_{\omega}$ is unbounded; otherwise, there exists $k \in \mathbb{N}$ such that $\Phi_{k+1} \subseteq \Phi_{\omega} \subseteq \Phi_k$, which is impossible by Corollary~\ref{Properness}.
As a third example, we show that for any $\alpha \preceq \omega^{\omega}$, the downset $\mathsf{D}_{\alpha}$ is bounded and $l(\mathsf{D}_{\alpha}) = 0$. Note that by the definition of $\omega^{\omega} = \langle \omega^{i+1} \rangle_i$, if $\beta \prec \alpha \preceq \omega^{\omega}$, by the properties of $\prec$, then $\beta \prec \omega^k$ for some $k \geq 1$. Since $\omega^k \in \Psi_{\omega}$, Lemma~\ref{PhiIsDownset} yields $\beta \in \Psi_{\omega}$. Thus, $\mathsf{D}_{\alpha} \subseteq \Psi_{\omega}$, and so $\mathsf{D}_{\alpha}$ is bounded and $l(\mathsf{D}_{\alpha}) = 0$.
As concrete examples we will use later, we have $l(\mathsf{D}_{\omega^k}) = l(\mathsf{D}_{\omega^\omega}) = 0$ for any $k \geq 1$.    
\end{example}

The main theorem of this paper is the characterization of $\PredFuncClass{\mathsf{A}}$ for any downset $\mathsf{A} \subseteq \Phi_{\omega}$ such that $\mathsf{A} \nsubseteq \mathsf{D}_{\omega}$.

\begin{theorem}[Main Theorem]\label{mainth}
Let $\mathsf{A} \subseteq \Phi_{\omega}$ be a downset of ordinals with $\mathsf{A} \nsubseteq \mathsf{D}_{\omega}$. Then:
\begin{itemize}
    \item[$(i)$] 
If $\mathsf{A}$ is bounded, then $\PredFuncClass{\mathsf{A}} = \GrzClass{l(\mathsf{A})+2}$.
    \item[$(ii)$] 
If $\mathsf{A}$ is unbounded, then $\PredFuncClass{\mathsf{A}} = \PR$.  
\end{itemize}
\end{theorem}

We will prove part~$(i)$ of Theorem~\ref{mainth} later. However, part~$(ii)$ follows as a consequence of part~$(i)$, and hence we provide its proof now.

\begin{proof}[Proof of part $(ii)$]
First, note that since $\mathsf{A} \nsubseteq \mathsf{D}_{\omega}$, there exists $\alpha \in \mathsf{A}$ such that $\alpha \nprec \omega$. By the dichotomy theorem (Theorem~\ref{FiniteOrInfiniteThm}), since $\alpha \in \mathsf{A} \subseteq \Phi_{\omega}$, there exists $r \in \mathbb{N}$ such that either $\alpha = o(r)$ or $o(r) + \omega \preceq \alpha$. The first case is impossible, as $\alpha \nprec \omega$. Therefore, $o(r) + \omega \preceq \alpha$. Since $\alpha \in \mathsf{A}$ and $\mathsf{A}$ is a downset, we conclude $o(r) + \omega \in \mathsf{A}$.

Second, define $\mathsf{A}_k = \mathsf{A} \cap \Phi_k$ for any $k \geq 1$. First, by Lemma~\ref{MonotonicityOfPhi}, $\Phi_{\omega} = \bigcup_{k \in \mathbb{N}} \Phi_k$. Since $\Phi_0 = \{0\} \subseteq \Phi_1$, we also have $\Phi_{\omega} = \bigcup_{k \geq 1} \Phi_k$. Therefore, as $\mathsf{A} \subseteq \Phi_{\omega}$, it follows that $\mathsf{A} = \bigcup_{k \geq 1} \mathsf{A}_k$. Second, by Lemma~\ref{PhiIsDownset}, each $\Phi_k$ is a downset, and hence each $\mathsf{A}_k$ is a bounded downset. Third, since $o(r) + \omega \in \Phi_k$ for any $k \geq 1$, we have $o(r) + \omega \in \mathsf{A}_k$ for all $k \geq 1$. Hence, as $o(r) + \omega \nprec \omega$, it follows that $\mathsf{A}_k \nsubseteq \mathsf{D}_{\omega}$. Therefore, we can apply part~$(i)$ of Theorem~\ref{mainth} to each $\mathsf{A}_k$.

Furthermore, note that the set $\{l(\mathsf{A}_k) \mid k \geq 1\}$ is unbounded; for otherwise, there exists $N \in \mathbb{N}$ such that $l(\mathsf{A}_k) \leq N$ for all $k \geq 1$. Then $\mathsf{A}_k \subseteq \Phi_N$ by the definition of $l$, which implies $\mathsf{A}= \bigcup_{k \geq 1} \mathsf{A}_k \subseteq \Phi_N$, contradicting the assumption that $\mathsf{A}$ is unbounded. 

Now, since $\mathsf{A}_k \subseteq \mathsf{A}_{k+1}$ and $\mathsf{A} = \bigcup_{k \geq 1} \mathsf{A}_k$, by Lemma~\ref{lem:PredInLimit} and part~$(i)$ of Theorem~\ref{mainth}, we have 
\[
\PredFuncClass{\mathsf{A}} = \bigcup_{k \geq 1} \PredFuncClass{\mathsf{A}_k} = \bigcup_{k \geq 1} \mathcal{E}_{l(\mathsf{A}_k)+2}.
\]
Since the set $\{l(\mathsf{A}_k) \mid k \geq 1\}$ is unbounded, we have $\bigcup_{k \geq 1} \mathcal{E}_{l(\mathsf{A}_k)+2} = \PR$, which implies $\PredFuncClass{\mathsf{A}} = \PR$.
\end{proof}

\begin{remark}
The characterization provided in Theorem~\ref{mainth} is tight for any downset $\mathsf{A} \subseteq \Phi_{\omega}$ of ordinals. The only aspect that requires further clarification is the condition $\mathsf{A} \nsubseteq \mathsf{D}_{\omega}$. This condition serves as a natural restriction, ensuring that $\mathsf{A}$ is sufficiently large to support a meaningful recursion. Indeed, if $\mathsf{A} \subseteq \mathsf{D}_{\omega}$, then predicative ordinal recursion operates only up to a fixed finite ordinal, effectively reducing predicative ordinal recursion to safe composition. In this case, the class loses its predicative ordinal recursion scheme entirely and becomes essentially the closure of basic functions under safe composition, and hence too weak to be of significant interest.
\end{remark}

Some concrete instances of Theorem~\ref{mainth} are particularly noteworthy. First, we can apply the characterization to the classes $\Psi_{\omega}$, $\Phi_k$ (for $k \geq 1$), and $\Phi_{\omega}$:

\begin{corollary}\label{MaincorI}
$\PredFuncClass{\Psi_{\omega}}=\GrzClass{2}$, 
    $\PredFuncClass{\Phi_{k}}=\GrzClass{k+2}$, for any $k \geq 1$, and 
    $\PredFuncClass{\Phi_{\omega}}=\PR$.
\end{corollary}
\begin{proof}
First, note that none of the classes $\Psi_{\omega}$, $\Phi_k$ (for $k \geq 1$), and $\Phi_{\omega}$ is a subset of $\mathsf{D}_{\omega}$, since they all contain $\omega$. Then, it is enough to use Example~\ref{ExamplesOfLevels} and Theorem~\ref{mainth}.
\end{proof}

\begin{remark}
As consequences of Theorem~\ref{MeaningOfPhi} and Theorem~\ref{MeaningOfPsi}, we showed that $\Psi_{\omega}$, $\Phi_{k}$ (for $k \geq 1$), and $\Phi_{\omega}$ are the constructive counterparts of the classes of all set-theoretic ordinals below $\bm{\omega}^{\omega}$, $\bm{\phi}_k(\bm{0})$, and $\bm{\phi}_{\bm{\omega}}(\bm{0})$, respectively. Therefore, one can intuitively interpret Corollary~\ref{MaincorI} as stating that predicative ordinal recursion up to ordinals below $\bm{\omega}^{\omega}$, $\bm{\phi}_k(\bm{0})$, or $\bm{\phi}_{\bm{\omega}}(\bm{0})$ corresponds to the classes $\GrzClass{2}$, $\GrzClass{k}$, and $\PR$, respectively, if $k \geq 1$. For instance, to construct all and only elementary functions (i.e., the functions in $\mathcal{E}_3$), one must use predicative ordinal recursion up to the ordinals below $\bm{\phi}_1(\bm{0}) = \bm{\epsilon_0}$.

As a second remark, note that Corollary~\ref{MaincorI} can alternatively be interpreted as a characterization of the Grzegorczyk hierarchy via predicative ordinal recursion on the ordinals within the corresponding classes. This yields a structural and machine-independent characterization of the Grzegorczyk hierarchy, extending the Bellantoni–Cook characterization of $\mathcal{E}_2$ to all levels of the hierarchy.
\end{remark}

Another interesting special case of Theorem~\ref{mainth} arises when the downset $\mathsf{A}$ is restricted to either $\mathsf{D}_{\omega^{\omega}}$ or $\mathsf{D}_{\omega^k}$ for $k \geq 2$. In these cases, we have:

\begin{corollary}\label{MaincorII}
$\PredFuncClass{\omega^k}=\PredFuncClass{\omega^{\omega}}=\GrzClass{2}$, for any $k \geq 2$.
\end{corollary}
\begin{proof}
First, note that for any $k \geq 2$, neither of the classes $\mathsf{D}_{\omega^k}$ nor $\mathsf{D}_{\omega^{\omega}}$ is a subset of $\mathsf{D}_{\omega}$, since both include $\omega$. Then, it is enough to use Example~\ref{ExamplesOfLevels} and Theorem~\ref{mainth}.
\end{proof}

Corollary~\ref{MaincorII} states that if we extend predicative recursion from the set of natural numbers to any fixed finite product thereof, or even to the union of all such finite products, the resulting functions still remain computable in linear space.

Now, let us explain our strategy to prove part~$(i)$ of Theorem~\ref{mainth}. To that end, given any bounded downset $\mathsf{A} \subseteq \Phi_{\omega}$ satisfying $\mathsf{A} \nsubseteq \mathsf{D}_{\omega}$, we need to establish two claims: $\mathcal{E}_{l(\mathsf{A})+2} \subseteq \PredFuncClass{\mathsf{A}}$ and $\PredFuncClass{\mathsf{A}} \subseteq \mathcal{E}_{l(\mathsf{A})+2}$. For each of these claims, we follow a strategy that will be outlined in the following in a rough manner.

For the first claim, i.e., $\mathcal{E}_{l(\mathsf{A})+2} \subseteq \PredFuncClass{\mathsf{A}}$, using the assumption $\mathsf{A} \nsubseteq \mathsf{D}_{\omega}$ and Theorem~\ref{FiniteOrInfiniteThm}, we can easily deduce that there exists $r \in \mathbb{N}$ such that $o(r)+\omega \in \mathsf{A}$. Therefore, by Corollary~\ref{cor:E2-subset-predr}, all linear-space computable functions (i.e., the class $\mathcal{E}_2$) are contained in $\PredFuncClass{\mathsf{A}}$.
Hence, if $l(\mathsf{A}) = 0$, the claim follows immediately. For the case $l(\mathsf{A}) \geq 1$, since $\mathcal{E}_2 \subseteq \PredFuncClass{\mathsf{A}}$, the strategy is to show that any function in $\mathcal{E}_{l(\mathsf{A})+2}$ is computable in space that is linear in a function from $\PredFuncClass{\mathsf{A}}$, and then to use the closure of $\PredFuncClass{\mathsf{A}}$ under composition to obtain the desired function in $\PredFuncClass{\mathsf{A}}$.

To this end, recall that each function in $\mathcal{E}_{l(\mathsf{A})+2}$ is computable in space bounded by a function in $\mathcal{E}_{l(\mathsf{A})+2}$. When $l(\mathsf{A}) \geq 2$ (resp. $l(\mathsf{A}) = 1$), since any function in $\mathcal{E}_{l(\mathsf{A})+2}$ is bounded by some number of iterations of $h_{l(\mathsf{A})+1}$ (resp. by the function $e(n) = 2^n$), it follows that these functions are computable in space bounded by some number of iterations of $h_{l(\mathsf{A})+1}$ (resp. $e$).
Moreover, recall from Example~\ref{ex:grz-def} that $G_{\alpha} \in \PredFuncClass{\mathsf{A}}$ for every $\alpha \in \mathsf{A}$. Therefore, it suffices to show that $h_{l(\mathsf{A})+1}$ is essentially bounded by $G_{\phi_{l(\mathsf{A})-1}(0)}$ if $l(\mathsf{A}) \geq 2$, and that $G_{\phi_{l(\mathsf{A})-1}(0)}$ (resp. $e$) is bounded by $G_{\alpha}$ for some $\alpha \in \mathsf{A}$, if $l(\mathsf{A}) \geq 2$ (resp. $l(\mathsf{A})=1$). This strategy will be carried out in Section~\ref{sec:ReductionToG}.

For the reverse direction $\PredFuncClass{\mathsf{A}} \subseteq \mathcal{E}_{l(\mathsf{A})+2}$, we consider two cases. If $l(\mathsf{A}) = 0$, then $\mathsf{A} \subseteq \Psi_{\omega} = \bigcup_{k \geq 1} \bigcup_{m \geq 1} \Psi^m_k$, and if $l(\mathsf{A}) \geq 1$, then $\mathsf{A} \subseteq \Phi_{l(\mathsf{A})} = \bigcup_{m \geq 1} \Phi^m_{l(\mathsf{A}) - 1}$, by Definition~\ref{definitionOfl}, Lemma~\ref{MonotonicityOfPsi}, and Lemma~\ref{MonotonicityOfPhi}. In both cases, since $\PredFuncClass{\mathsf{A}}$ is continuous in $\mathsf{A}$ by Lemma~\ref{lem:PredInLimit}, it suffices to show that $\PredFuncClass{\Psi^m_k} \subseteq \mathcal{E}_2$ and $\PredFuncClass{\Phi^m_l} \subseteq \mathcal{E}_{l+3}$ for all $k, m \geq 1$ and $l \geq 0$.

For this purpose, we first introduce a suitable notation system for ordinals, encoding them as strings over the alphabet 
$\Sigma = \{0,1,(,),;,\bot\}$, which enables the simulation of the functions in 
$\mathcal{C}_{\Psi^m_k}$ (resp. $\mathcal{C}_{\Phi^m_l}$) via the functions in 
$\GrzClassSig{2}$ (resp. $\GrzClassSig{l+3}$), i.e., the Grzegorczyk hierarchy over $\Sigma$.  
Then, for numeral functions in $\PredFuncClass{\Psi^m_k}$ (resp. $\PredFuncClass{\Phi^m_l} \subseteq \mathcal{E}_{l+3}$), we can return to the usual Grzegorczyk hierarchy by Remark~\ref{rem:equiv-E-sig-E}.

Up to this encoding, note that the basic functions in $\mathcal{C}_{\Psi^m_k}$ (resp. $\mathcal{C}_{\Phi^m_l}$) are already available at $\GrzClassSig{2}$, and each $\GrzClassSig{j}$ is closed under composition. Therefore, the main task is to simulate predicative ordinal recursion over the ordinals in $\Psi^m_k$ (resp. $\Phi^m_l$) using length-bounded primitive recursion available in $\GrzClassSig{2}$ (resp. $\GrzClassSig{l+3}$).
To explain, let $f$ be defined by predicative ordinal recursion:
\begin{align*}
    f(0, \bar \NormalOrdB, \bar \SafeOrdA; \bar \SafeNumbA)
        &= g(\bar \NormalOrdB, \bar \SafeOrdA; \bar \SafeNumbA), \\
    f(\NormalOrdA + 1, \bar \NormalOrdB, \bar \SafeOrdA; \bar \SafeNumbA) 
        &= h_{\mathsf{suc}}(\NormalOrdA, \bar \NormalOrdB, \bar \SafeOrdA; 
            f(\NormalOrdA, \bar \NormalOrdB, \bar \SafeOrdA; \bar \SafeNumbA), \bar \SafeNumbA), \\
    f(\langle \NormalOrdA_i \rangle_i, \bar \NormalOrdB, \bar \SafeOrdA; \bar \SafeNumbA)
        &= h_{\mathsf{lim}}(\langle \NormalOrdA_i \rangle_i, \bar \NormalOrdB, \bar \SafeOrdA;
            f(\NormalOrdA_{q(\bar \SafeOrdA;)}, \bar \NormalOrdB, \bar \SafeOrdA; \bar \SafeNumbA),
            \bar \SafeNumbA).
\end{align*}
In computing $f(\mu, \bar \NormalOrdB, \bar \SafeOrdA; \bar \SafeNumbA)$, observe that the ordinal recursion effectively proceeds along the sequence of $q(\bar n)$-predecessors of $\mu$, starting from $0$ and culminating at $\mu$. Therefore, it suffices to compute
$
f(R(i, \mu, q(\bar n)), \bar \NormalOrdB, \bar \SafeOrdA; \bar \SafeNumbA)
$ 
as a function of $i \in \mathbb{N}$, which can be carried out via primitive recursion on $i$, using $R(i, \mu, q(\bar n))$ to determine the ordinal at each step, and $g$, $h_{\mathsf{suc}}$, and $h_{\mathsf{lim}}$ for the base and recursive cases. Finally, setting $i = L_{\mu}(q(\bar n))$ yields $f(\mu, \bar \NormalOrdB, \bar \SafeOrdA; \bar \SafeNumbA)$, thus completing the simulation of ordinal recursion by primitive recursion.
Now, to show that $f$ is in $\GrzClassSig{2}$ (resp. $\GrzClassSig{l+3}$), it remains to verify that $R$ and $L$ belong to $\GrzClassSig{2}$ (resp. $\GrzClassSig{l+3}$). Furthermore, by bounding the lengths of the values of $f$ with a function in $\mathcal{E}_2$ (resp. $\mathcal{E}_{l+3}$), we can apply length-bounded primitive recursion (Theorem~\ref{the:Esigma-closure}) to conclude that $f$ is in $\GrzClassSig{2}$ (resp. $\GrzClassSig{l+3}$). We pursue this line of argument in Section~\ref{sec:UpperBoundLength}.

\section{Simulation of the Grzegorczyk Hierarchy}
\label{sec:ReductionToG}
In this section, we prove that $\mathcal{E}_{l(\mathsf{A})+2} \subseteq \PredFuncClass{\mathsf{A}}$ for any bounded downset $\mathsf{A} \subseteq \Phi_{\omega}$ satisfying $\mathsf{A} \nsubseteq \mathsf{D}_{\omega}$. 
As explained in Section~\ref{sec:main-theorem}, a key component of our strategy is to establish an upper bound for $h_i$ in terms of $G_{\phi_{i-2}(0)}$ for all $i \geq 3$. To this end, we must analyze the behavior of the slow-growing hierarchy on the ordinals in $\Phi_{\omega}$. Recall that the ordinals in $\Phi_{\omega}$ are generated from $0$ using ordinal addition and the $\phi_k$ functions. Moreover, the functions in the slow-growing hierarchy are additive with respect to the ordinal argument. Therefore, to understand the behavior of the slow-growing hierarchy on $\Phi_{\omega}$, it suffices to compute $G_{\phi_k(\alpha)}(n)$ in terms of $G_{\alpha}(n)$.

In this section, we will also obtain an \emph{upper bound} for the length hierarchy on ordinals in $\Phi_{\omega}$. 
Since the functions in the length hierarchy, like those in the slow-growing hierarchy, commute with addition on the ordinal argument, it suffices to bound $L_{\phi_k(\alpha)}(n)$ in terms of $L_{\alpha}(n)$. As these two tasks are structurally similar, we treat them together in this section, although the upper bound on the length hierarchy will only be needed later in Section~\ref{sec:UpperBoundLength}.

To accomplish these tasks, we introduce two families of functions $\{H_k\}_{k \in \mathbb{N}}$ and $\{Q_k\}_{k \in \mathbb{N}}$ to simulate the computation of $G_{\phi_k(\alpha)}(n)$ (resp.~$L_{\phi_k(\alpha)}(n)$) in terms of $G_{\alpha}(n)$ (resp. $L_{\alpha}(n)$).
To this end, consider the following numeral functions:
\begin{align*}
	&H_0(x,y) := (x+2)^y,\quad\quad
	H_{k+1}(x,0) := H_k^{(x)}(x,0)+1\\
	&H_{k+1}(x,y+1) := H_k^{(x)}(x, H_{k+1}(x,y)+1)+1\\[1em]
	&Q_0(x,y) := (x+1)^y,\quad\quad
	Q_{k+1}(x,0) := Q_k^{(x)}(x,0)\\
	&Q_{k+1}(x,y+1) := Q_k^{(x)}(x, Q_{k+1}(x,y)+1)
\end{align*}
Note that $Q_k^{(x)}$ and $H_k^{(x)}$ denotes $x$ many iterations on the \emph{second} argument of $Q_k$ and $H_k$, respectively. Later, in Lemma~\ref{lem:ub-g-and-l-via-q-h}, we will see that 
$G_{\Veb{k}(\beta)}(x) = Q_k(x, G_{\beta}(x))$ and 
$L_{\Veb{k}(\beta)}(x) \leq H_k(x, L_{\beta}(x))$ 
for any $\beta \in \Omega$ and $k \in \mathbb{N}$.

We begin by establishing some basic properties of these functions.

\begin{lemma}\label{ExpansiveLemma}
For any $k \in \mathbb{N}$, we have:
\begin{itemize}
 \item[$(i)$] 
$Q_{0}(0, y)=1$ and $Q_{k+1}(0, y)=y$, for any $x, y \in \mathbb{N}$.
    \item[$(ii)$] 
$Q_{k}(x, y) \geq y$, for any $x, y \in \mathbb{N}$ such that $x \neq 0$.
\item[$(iii)$] 
$H_k(x, y) \geq y$, for any $x, y \in \mathbb{N}$.
\end{itemize}
As a consequence of $(i)$ and $(ii)$, we have $Q_k(x, y) \geq y$, for any $x, y \in \mathbb{N}$ and any $k \geq 1$.
\end{lemma}
\begin{proof}
For $(i)$, the first part is clear as $Q_0(0, y)=1^y=1$. For the second part, we use induction on $y$. For $y=0$, we have $Q_{k+1}(0, 0)=Q_k^{(0)}(0, 0)=0$ and for the induction step:
\[
Q_{k+1}(0, y+1)=Q_{k+1}^{(0)}(0, Q_{k+1}(0, y)+1)= Q_{k+1}(0, y)+1 = y+1
\]

For $(ii)$, we prove the claim by induction on $k$. For $k=0$, since $x>0$, then $Q_0(x, y)=(x+1)^y\geq 2^y\geq y$, for any $y \in \mathbb{N}$. For the inductive step, we assume the claim for $k$ and use an induction on $y$ to show $Q_{k+1}(x, y) \geq y$. For $y=0$, as $0 \leq Q_{k+1}(x, 0)$, there is nothing to prove. For the inductive step, by the definition of $Q_{k+1}(x, y+1)$ and the induction hypothesis for $k$ and $y$, we have:  
\[
Q_{k+1}(x, y+1)=Q_{k}^{(x)}(x, Q_{k+1}(x, y)+1) \geq Q_{k+1}(x, y)+1 \geq y+1
\]
This completes the proof of $(ii)$. The proof for $(iii)$ is similar to $(ii)$. 
\end{proof}

\begin{lemma}\label{MonotinicityOfQ}
$Q_k$ and $H_k$ are monotone in both of their arguments, for any $k \in \mathbb{N}$.
\end{lemma}
\begin{proof}
We only prove the claim for $Q_k$. The other is similar. To prove the monotonicity of $Q_k$, we use an induction on $k$. For $k=0$ the claim is clear, as $Q_0(x, y)=(x+1)^y$ is monotone. For the inductive step, we assume the monotonicity of $Q_k$ to prove the monotonicity of $Q_{k+1}$. For that purpose,
it is enough to prove $Q_{k+1}(x, y) \leq Q_{k+1}(x, y+1)$ and
$Q_{k+1}(x, y) \leq Q_{k+1}(x+1, y)$, for any $x, y \in \mathbb{N}$. First, observe that $Q_{k}^{(x)}(x, z) \geq z$, for any $z \in \mathbb{N}$. The reason is that if $x=0$, then 
$Q_{k}^{(x)}(x, z) =z$, by definition and if $x>0$, by Lemma \ref{ExpansiveLemma}, we have $Q_{k}^{(x)}(x, z) \geq z$.
Now, using $z=Q_{k+1}(x, y)+1$, we get:
\[
Q_{k+1}(x, y+1)=Q_{k}^{(x)}(x, Q_{k+1}(x, y)+1) \geq Q_{k+1}(x, y)+1 \geq Q_{k+1}(x, y). 
\]
To prove $Q_{k+1}(x, y) \leq Q_{k+1}(x+1, y)$, we use an induction on $y$. For $y=0$, by Lemma \ref{ExpansiveLemma} and the monotonicity of $Q_{k}$, we have:
\[
Q_{k+1}(x+1, 0)=Q_{k}^{(x+1)}(x+1, 0) \geq Q_{k}^{(x)}(x+1, 0) \geq Q_{k}^{(x)}(x, 0)=Q_{k+1}(x, 0).
\]
For the induction step, using Lemma \ref{ExpansiveLemma}, the monotonicity of $Q_{k}$ and the induction hypothesis, we have:
\[
Q_{k+1}(x+1, y+1)=Q_{k}^{(x+1)}(x+1, Q_{k+1}(x+1, y)+1)
\]
\[
\geq Q_{k}^{(x)}(x+1, Q_{k+1}(x+1, y)+1) \geq Q_{k}^{(x)}(x, Q_{k+1}(x, y)+1)=Q_{k+1}(x, y+1). \qedhere
\]
\end{proof}

\begin{lemma}\label{kMonotonicityOfQ}
Let $k\in \mathbb{N}$ be a natural number. Then:
\begin{itemize}
    \item[$(i)$] 
For any $x, y\in \mathbb{N}$, if $x>0$ then $Q_k(x, y) \leq Q_{k+1}(x, y)$.
    \item[$(ii)$] 
For any $x, y\in \mathbb{N}$, we have $H_k(x, y) \leq H_{k+1}(x, y)$.   
\end{itemize}
\end{lemma}
\begin{proof}
We only prove $(i)$. Part $(ii)$ is similar. For $(i)$, we use induction on $k$. For $k=0$, we use induction on $y$. For $y=0$, by $x \geq 1$ and using Lemma \ref{ExpansiveLemma}, we have
\[
Q_1(x, 0)=Q_0^{(x)}(x, 0)=Q_0^{(x-1)}(x, Q_0(x, 0)) \geq Q_0(x, 0). 
\]
For the induction step, first, notice that by definition, $Q_0(x, y)(x+1)=Q_0(x, y+1)$. Now, by $x \geq 1$, Lemmas \ref{ExpansiveLemma} and \ref{MonotinicityOfQ} and the induction hypothesis, we have
\[
Q_1(x, y+1)=Q_0^{(x)}(x, Q_1(x, y)+1) \geq Q_0(x, Q_1(x, y)+1) \geq Q_0(x, Q_0(x, y)+1)
\]
\[
= (x+1)^{Q_0(x, y)+1} \geq  Q_0(x, y)(x+1)=Q_0(x, y+1).
\]
This completes the case $k=0$. For the induction step, we assume the claim for $k$ and prove it for $k+1$. We prove by induction on $y$. For $y=0$, by Lemma \ref{MonotinicityOfQ} and the induction hypothesis, we have:
\[
Q_{k+2}(x, 0)=Q_{k+1}^{(x)}(x, 0) \geq Q_{k}^{(x)}(x, 0)=Q_{k+1}(x, 0).
\]
For the induction step for $y$, by Lemma \ref{MonotinicityOfQ} and the induction hypothesis, we have:
\[
Q_{k+2}(x, y+1)=Q_{k+1}^{(x)}(x, Q_{k+2}(x, y)+1) 
\geq 
Q_{k+1}^{(x)}(x, Q_{k+1}(x, y)+1) 
\]
\[
\geq Q_{k}^{(x)}(x, Q_{k+1}(x, y)+1)=Q_{k+1}(x, y+1).
\qedhere
\]
\end{proof}

The following lemma provides the connection between $\phi_k(\beta)$ and $\beta$ that we have been seeking:

\begin{lemma}\label{lem:ub-g-and-l-via-q-h}
$G_{\Veb{k}(\beta)}(x) = Q_k(x,G_{\beta}(x))$ and 
$L_{\Veb{k}(\beta)}(x) \leq H_k(x,L_{\beta}(x))$, for any $\beta \in \Omega$ and $k \in \mathbb{N}$. Consequently, $G_{\alpha}$ is monotone, for any $\alpha \in \Phi_{\omega}$.
\end{lemma}
\begin{proof}
We only prove $L_{\Veb{k}(\beta)}(x) \leq H_k(x,L_{\beta}(x))$. The equality for $G$ is similar and can also be found in the proof of~\cite[Prop.~VIII.8.27]{odifreddi1999crtv2}. 
We prove the claim by induction on $k$. For $k=0$, as $L_{\omega}(n)=n+2 \geq 2$, by Lemma \ref{lem:g-properties}, we have:
\[
L_{\phi_0(\beta)}(n)=L_{\omega^\beta}(n) \leq L_{\omega}(n)^{L_{\beta}(n)}=(n+2)^{L_{\beta}(n)}=H_0(n, L_{\beta}(n)).
\]
For the induction step, let $k=l+1$ and assume the claim for $l$. Now, by induction on $\beta$, we show that $L_{\Veb{k}(\beta)}(x) \leq H_k(x,L_{\beta}(x))$. For $\beta=0$, by the induction hypothesis for $l$ and the monotonicity of $H_l$ proved in Lemma \ref{MonotinicityOfQ}, we have
\[
L_{\phi_{l+1}(0)}(n)=L_{\langle \phi_l^i(0) \rangle_i}(n)= L_{\phi_l^n(0)}(n)+1 \leq H_l^{(n)}(n, L_{0}(n))+1= H_{l+1}(n, 0).
\]
For the successor $\beta=\gamma+1$, by the induction hypothesis for $l$ and $\gamma$ and the monotonicity of $H_l$ proved in Lemma \ref{MonotinicityOfQ}, we have
\[
L_{\phi_{l+1}(\gamma+1)}(n)=L_{\langle \phi_l^i(\phi_{l+1}(\gamma)+1) \rangle_i}(n)= L_{\phi_l^n(\phi_{l+1}(\gamma)+1)}(n)+1
\]
\[
\leq H_l^{(n)}(n, L_{\phi_{l+1}(\gamma)+1}(n))+1 \leq H_l^{(n)}(n, H_{l+1}(n, L_{\gamma}(n))+1)+1 
\]
\[
{=} H_{l+1}(n, L_{\gamma}(n)+1)=H_{l+1}(n, L_{\gamma+1}(n)).
\]
For the limit $\beta=\langle \gamma_i \rangle$, 
by the induction hypothesis for $\gamma_n$ and the expansiveness of $H_l$, we have
\[
L_{\phi_{l+1}(\langle \gamma_i \rangle)}(n)=L_{\langle \phi_{l+1}(\gamma_i)\rangle}(n)= L_{\phi_{l+1}(\gamma_n)}(n)+1 \leq H_{l+1}(n, L_{\gamma_n}(n))+1 
\]
\[
\leq
H^{(n)}_l(n, H_{l+1}(n, L_{\gamma_n}(n))+1)+1 
= H_{l+1}(n, L_{\gamma_n}(n)+1)
= H_{l+1}(n, L_{\langle \gamma_i \rangle}(n)).
\]
This completes the proof of the first part. For the second part, i.e., the monotonicity of $G_{\alpha}$ for any $\alpha \in \Phi_{\omega}$, let $X$ be the set of ordinals $\alpha \in \Phi_{\omega}$ for which $G_{\alpha}$ is monotone. Since the constant zero function and each $Q_i$ (for $i \in \mathbb{N}$) are monotone (Lemma \ref{MonotinicityOfQ}), and the sum of monotone functions is monotone, it follows from Lemma~\ref{lem:g-properties} and the first part that $X$ contains $0$ and is closed under the operations $\phi_i$ and addition. Therefore, $X = \Phi_{\omega}$, which completes the proof of the second part. 
\end{proof}

\begin{remark}
Note that $G_{\alpha}$ is not necessarily monotone for arbitrary $\alpha \in \Omega$. For instance, consider $\alpha = \langle 1, 0, 0, \ldots \rangle$, which is constantly zero except for the first element, which is one. In this case, we have $G_{\alpha}(0) = G_{1}(0) = 1$, while $G_{\alpha}(n+1) = G_{0}(n+1) = 0$ for any $n \in \mathbb{N}$.
\end{remark}

In the next lemma, we show the connection between iterating $Q_k$ with respect to its second argument and $Q_{k+1}$. Moreover, we demonstrate the relationship between the two arguments of $Q_k$ in a specific situation that we will need later.

\begin{lemma}\label{LemmaIteratedQ}
Let $k$ and $i$ be natural numbers. Then:
\begin{itemize}
    \item[$(i)$]
$Q_k^{2i}(2, 2) \leq Q_{k+1}(2, i)$. 
\item[$(ii)$]
$Q_{k+1}(2, i) \leq Q_{k+1}(3i+2, 0)$.    
\end{itemize}
\end{lemma}
\begin{proof}
For $(i)$, we use induction on $i$. For $i=0$, using the monotonicity of $Q_k(x, y)$ in $k$ (Lemma \ref{kMonotonicityOfQ}) and in $x$ and $y$ (Lemma \ref{MonotinicityOfQ}), we have 
\[
Q_k^{0}(2, 2)=2 \leq 3=Q_0^2(2, 0) \leq Q_k^2(2, 0)=Q_{k+1}(2, 0).
\]
For the inductive step, by the induction hypothesis and the monotonicity of $Q_k$ (Lemma \ref{MonotinicityOfQ}), we have
\[
Q_k^{2i+2}(2, 2)=Q_k^2(2, Q_k^{2i}(2, 2)) \leq Q_k^2(2, Q_{k+1}(2, i))
\]
\[
\leq Q_k^2(2, Q_{k+1}(2, i)+1)=Q_{k+1}(2, i+1),
\]
which completes the proof.

For $(ii)$, since $Q_{k}^{j}(2, 0) \leq Q^j_k(j, 0) = Q_{k+1}(j, 0)$ for any $j \geq 2$ by Lemma~\ref{MonotinicityOfQ}, it is enough to prove that $Q_{k+1}(2, i) \leq Q_{k}^{3i+2}(2, 0)$. To prove this claim, we use induction on $i$. 
For $i=0$, by definition, we have
$
Q_{k+1}(2, 0) = Q_k^{2}(2, 0).
$
For the inductive step, first note that by the monotonicity of $Q_k(x, y)$ in $k$ (Lemma~\ref{kMonotonicityOfQ}), we have
\[
y + 1 \leq 3^y = Q_0(2, y) \leq Q_k(2, y),
\]
for any $y \in \mathbb{N}$. Then, using the monotonicity of $Q_k$ (Lemma~\ref{MonotinicityOfQ}) and the induction hypothesis, we obtain
\[
Q_{k+1}(2, i+1) = Q_k^{2}(2, Q_{k+1}(2, i) + 1) \leq Q_k^{2}(2, Q_k(2, Q_{k+1}(2, i))) 
\]
\[
\leq Q_k^{3}(2, Q_k^{3i+2}(2, 0)) = Q_k^{3(i+1)+2}(2, 0),
\]
which completes the proof.
\end{proof}

Using the machinery developed thus far, we are now ready to provide the upper bound for $h_k$ in terms of $G_{\phi_{k-2}(0)}$ that we have been seeking. First, we use $Q_{k-2}(2, n)$ instead.

\begin{lemma}\label{BoundForh}
$h_{k}(n) \leq Q_{k-2}(2, n)$, for any $k \geq 3$ and $n \in \mathbb{N}$.
\end{lemma}
\begin{proof}
We prove the claim by induction on $k$. For $k=3$, recall that $h_1(n)=n^2+2$. Thus, $h_1(n) \geq 2$ for any $n \in \mathbb{N}$, and $h_1(n) \leq n^3$ for any $n \geq 2$. Therefore, $h_2(n)=h_1^{n}(2) \leq 3^{3^{n}}=Q_0^2(2, n)$, for any $n \in \mathbb{N}$. Now, using the monotonicity of $h_2$ and Lemma~\ref{LemmaIteratedQ}, we reach
\[
h_{3}(n)=h_2^{n}(2) \leq Q^{2n}_{0}(2, 2) \leq Q_{1}(2, n).
\]
For the inductive step, assume the claim for $k$. Then, by the induction hypothesis, the monotonicity of $h_k$, the expansiveness of $Q_k(2, -)$ (Lemma~\ref{ExpansiveLemma}), and Lemma~\ref{LemmaIteratedQ}, we have
\[
h_{k+1}(n)=h_k^{n}(2) \leq Q^{n}_{k-2}(2, 2) \leq Q^{2n}_{k-2}(2, 2) \leq Q_{k-1}(2, n),
\]
which completes the proof.
\end{proof}

\begin{corollary}\label{lem:Q-bounds-h}
$h_{k}(n) \leq G_{\phi_{k-2}(0)}(3n+2)$, for any $k \geq 3$ and $n \in \mathbb{N}$.   
\end{corollary}
\begin{proof}
By Lemma~\ref{BoundForh}, we have $h_k(n) \leq Q_{k-2}(2, n)$. Then, by Lemma~\ref{LemmaIteratedQ}, we obtain $h_k(n) \leq Q_{k-2}(3n+2, 0)$. Finally, since $G_{\phi_{k-2}(0)}(3n+2) = Q_{k-2}(3n+2, 0)$ by Lemma~\ref{lem:ub-g-and-l-via-q-h}, we obtain the desired bound.
\end{proof}

As the second part of the strategy for proving $\mathcal{E}_{l(\mathsf{A})+2} \subseteq \PredFuncClass{\mathsf{A}}$, as explained in Section~\ref{sec:main-theorem}, we need to bound $G_{\phi_{l(\mathsf{A})-1}(0)}$ (resp.\ $e(n)=2^n$) by $G_{\alpha}$ for some $\alpha \in \mathsf{A}$, if $l(\mathsf{A}) \geq 2$ (resp.\ $l(\mathsf{A})=1$). For that purpose, we first prove the following lemma.

\begin{lemma}\label{ExistenceOfWitness}
Let $k \geq 1$ be a natural number. Then:
\begin{itemize}
 \item[$(i)$] 
If $\alpha \in \Phi_{1}-\mathsf{D}_{\omega}$, then $G_{\alpha}(n) \geq n+1$, for any $n \in \mathbb{N}$.
    \item[$(ii)$] 
If $\alpha \in \Phi_{1}-\Psi_{\omega}$, then $G_{\alpha}(n) \geq 2^n$, for any $n \in \mathbb{N}$.
    \item[$(iii)$] 
If $\alpha \in \Phi_{k+1}-\Phi_k$, then $G_{\alpha}(n) \geq G_{\phi_k(0)}(n)$, for any $n \geq 1$. 
\end{itemize}
\end{lemma}
\begin{proof}
For $(i)$, as $\alpha \in \Phi_{1} - \mathsf{D}_{\omega}$, we have $\alpha \neq 0$. Therefore, there exist $\{\beta_i\}_{i=1}^m$ in $\Phi_{1}$ such that $\alpha = \phi_{0}(\beta_1) + \cdots + \phi_{0}(\beta_m)$. Note that there exists $1 \leq i \leq m$ such that $\beta_i \neq 0$, because otherwise, $\alpha \prec \omega$. Pick $1 \leq i \leq m$ such that $\beta_i \neq 0$. As $\beta_i \in \Phi_1$, there exist $\{\gamma_j\}_{j=1}^r$ in $\Phi_1$ such that $\beta_i = \phi_{0}(\gamma_1) + \cdots + \phi_{0}(\gamma_r)$. Therefore, by Lemma~\ref{lem:g-properties} and Lemma~\ref{lem:ub-g-and-l-via-q-h}, we have:
\[
G_{\beta_i}(n) = \sum_{j=1}^r G_{\phi_{0}(\gamma_j)}(n) = \sum_{j=1}^r G_{\omega^{\gamma_j}}(n) = \sum_{j=1}^r (n+1)^{G_{\gamma_j}(n)} \geq 1.
\]
Therefore, by Lemma~\ref{lem:g-properties} and Lemma~\ref{lem:ub-g-and-l-via-q-h}, we reach
\[
G_{\alpha}(n) \geq G_{\phi_{0}(\beta_i)}(n) = G_{\omega^{\beta_i}}(n) = (n+1)^{G_{\beta_i}(n)} \geq (n+1).
\]

For $(ii)$, as $\alpha \in \Phi_{1} - \Psi_{\omega}$, we have $\alpha \neq 0$. Therefore, there exist $\{\beta_i\}_{i=1}^m$ in $\Phi_{1}$ such that $\alpha = \phi_{0}(\beta_1) + \cdots + \phi_{0}(\beta_m)$. Again, note that there exists $1 \leq i \leq m$ such that $\beta_i \notin \mathsf{D}_{\omega}$, because otherwise, $\alpha \in \Psi_{\omega}$. Pick $1 \leq i \leq m$ such that $\beta_i \notin \mathsf{D}_{\omega}$. Therefore, using $(i)$, Lemma \ref{lem:g-properties} and Lemma \ref{lem:ub-g-and-l-via-q-h}, we reach
\[
G_{\alpha}(n) \geq G_{\phi_{0}(\beta_i)}(n) = G_{\omega^{\beta_i}}(n) = (n+1)^{G_{\beta_i}(n)} \geq (n+1)^{(n+1)} \geq 2^n.
\]

For $(iii)$, let $X \subseteq \Phi_{k+1}$ be the set of all ordinals $\alpha$ such that either $\alpha \in \Phi_k$ or $G_{\alpha}(n) \geq G_{\phi_k(0)}(n)$ for all $n \geq 1$. To prove the claim, it suffices to show that $X = \Phi_{k+1}$. For this, we show that $X$ contains $0$ and is closed under addition and under $\phi_i$ for any $i \leq k$.
Clearly, $0 \in \Phi_k$ and hence $0 \in X$. For addition, let $\alpha, \beta \in X$. If both $\alpha, \beta \in \Phi_k$, then $\alpha + \beta \in \Phi_k$. Otherwise, assume that at least one of $\alpha$ or $\beta$ is outside $\Phi_k$. Since $\alpha, \beta \in X$, either $G_{\alpha}(n) \geq G_{\phi_k(0)}(n)$ for all $n \geq 1$, or $G_{\beta}(n) \geq G_{\phi_k(0)}(n)$ for all $n \geq 1$. Using the fact that
$
G_{\alpha + \beta}(n) = G_{\alpha}(n) + G_{\beta}(n),
$ from Lemma \ref{lem:g-properties}, it follows that in either case
$G_{\alpha + \beta}(n) \geq G_{\phi_k(0)}(n)$ for all $n \geq 1$,
and thus $\alpha + \beta \in X$.
For closure under $\phi_i$ for any $i \leq k$, let $\alpha \in X$. There are two cases: either $i=k$ or $i < k$. If $i = k$, then we have the following by Lemma \ref{lem:ub-g-and-l-via-q-h} and the monotonicity of $Q_k$ (Lemma \ref{MonotinicityOfQ}), from which we conclude $\phi_k(\alpha) \in X$.
\[
G_{\phi_k(\alpha)}(n) = Q_k(n, G_{\alpha}(n)) \geq Q_k(n, 0) = Q_k(n, G_0(n)) = G_{\phi_k(0)}(n).
\]
If $i < k$, and $\alpha \in \Phi_k$, then clearly $\phi_i(\alpha) \in \Phi_k$
and thus
$\phi_i(\alpha) \in X$. Otherwise, if $G_{\alpha}(n) \geq G_{\phi_k(0)}(n)$ for all $n \geq 1$, then by the expansiveness of $Q_i(n, -)$ for $n \geq 1$ (Lemma \ref{ExpansiveLemma}), we have
\[
G_{\phi_i(\alpha)}(n) = Q_i(n, G_{\alpha}(n)) \geq G_{\alpha}(n) \geq G_{\phi_k(0)}(n),
\]
which establishes $\phi_i(\alpha) \in X$.
\end{proof}

Finally, we are in a position to prove the direction of part~$(i)$ of Theorem~\ref{mainth} that we have been seeking in this section.

\begin{theorem}
Let $\mathsf{A} \subseteq \Phi_{\omega}$ be a bounded downset of ordinals such that $\mathsf{A} \nsubseteq \mathsf{D}_{\omega}$. Then, $\GrzClass{l(\mathsf{A})+2} \subseteq \PredFuncClass{\mathsf{A}}$.
\end{theorem}
\begin{proof}
First, observe that, as $\mathsf{A} \subseteq \Phi_{\omega}$, by Theorem~\ref{FiniteOrInfiniteThm}, either every element in $\mathsf{A}$ is of the form $o(r)$ for some $r \in \mathbb{N}$, or there exists an element $\alpha \in \mathsf{A}$ such that $o(r) + \omega \preceq \alpha$. The first case implies that $\mathsf{A} \subseteq \mathsf{D}_{\omega}$, which contradicts the assumption. Thus, as $\mathsf{A}$ is a downset, there exists $r \in \mathbb{N}$ such that $o(r) + \omega \in \mathsf{A}$. Therefore, $\mathcal{E}_2 \subseteq \PredFuncClass{\mathsf{A}}$, by Corollary~\ref{cor:E2-subset-predr}.

To prove $\GrzClass{l(\mathsf{A})+2} \subseteq \PredFuncClass{\mathsf{A}}$, we consider three cases. The case $l(\mathsf{A}) = 0$ has already been established. We postpone the proof for the case $l(\mathsf{A}) = 1$ to the end, as it follows with only minor adjustments from the last case $l(\mathsf{A}) \geq 2$.

For $l(\mathsf{A}) \geq 2$, let $f(n_1, \ldots, n_l) \in \GrzClass{l(\mathsf{A})+2}$ to show that $f \in \PredFuncClass{\mathsf{A}}$. First, by Theorem~\ref{the:elem-characterizatons}, there exists a function $T_f \in \mathcal{E}_{l(\mathsf{A})+2}$ such that $f(\bar{n})$ is computable in time $T_f(\RepLen{\bar{n}})$ for any $\bar{n} \in \mathbb{N}$. 
By Lemma~\ref{lem:grz-properties}, there exists $M \in \mathbb{N}$ such that
\[
T_f(\bar{m}) \leq h_{l(\mathsf{A})+1}^M\left( \max_j m_j \right)
\quad \text{for all } \bar{m} \in \mathbb{N}.
\]
Since $h_{l(\mathsf{A})+1}$ is monotone and $\RepLen{n_j} \leq n_j + 1$ for each $1 \leq j \leq l$, we can conclude that $f(\bar{n})$ is computable in time at most $h_{l(\mathsf{A})+1}^M(l + \sum_j n_j)$.

Let $U$ be a deterministic Turing machine that computes $f(\bar{n})$ within at most $h_{l(\mathsf{A})+1}^M(l + \sum_j n_j)$ steps. Define a function $g(i, \bar{n})$ that returns the output of $U(\bar{n})$ if the machine halts within $\RepLen{i}$ steps, and $0$ otherwise. Since we use binary representations for numbers (including for the input $i$), it follows that
\[
f(\bar{n}) = g(i, \bar{n}) \quad \text{for all } i \geq 2^{h_{l(\mathsf{A})+1}^M(l + \sum_j n_j)}.
\]
It is clear that $g$ is computable in linear space: it suffices to simulate $U$ for at most $\RepLen{i}$ steps, which requires only $O(\RepLen{i})$ cells. Thus, $g \in \mathcal{E}_2 \subseteq \PredFuncClass{\mathsf{A}}$.
    
To conclude that $f \in \PredFuncClass{\mathsf{A}}$, it now suffices to find a function $I(\bar n) \in \PredFuncClass{\mathsf{A}}$ such that $I(\bar n) \geq 2^{h_{l(\mathsf{A})+1}^M(l+\sum_j n_j)}$. Then, since $f(\bar n) = g(I(\bar n), \bar n)$ and both $g, I \in \PredFuncClass{\mathsf{A}}$, and $\PredFuncClass{\mathsf{A}}$ is closed under composition, we obtain $f \in \PredFuncClass{\mathsf{A}}$. To find such an $I$, we show that:\\
    
\noindent \textbf{Claim.} There are monotone functions $F, E \in \PredFuncClass{\mathsf{A}}$ such that $F(n) \geq G_{\phi_{l(\mathsf{A})-1}(0)}(n)$ and $E(n) \geq 2^n$, for any $n \in \mathbb{N}$.\\
    
\noindent Using this claim and Corollary~\ref{lem:Q-bounds-h}, since addition and the constant functions are in $\mathcal{E}_2 \subseteq \PredFuncClass{\mathsf{A}}$, and using the monotonicity of $F$ and $E$, as well as the closure of $\PredFuncClass{\mathsf{A}}$ under composition, we can find the desired function $I \in \PredFuncClass{\mathsf{A}}$.

To prove the claim, first observe that, as $l(\mathsf{A}) \geq 2$, by definition, $\mathsf{A} \subseteq \Phi_{l(\mathsf{A})}$ but $\mathsf{A} \nsubseteq \Phi_{l(\mathsf{A})-1}$. Thus, there exists $\alpha \in \mathsf{A}$ such that $\alpha \in \Phi_{l(\mathsf{A})} - \Phi_{l(\mathsf{A})-1}$. Therefore, Lemma~\ref{ExistenceOfWitness} yields $G_{\phi_{l(\mathsf{A})-1}(0)}(n) \leq G_{\alpha}(n)$ for all $n \geq 1$. This implies that $G_{\phi_{l(\mathsf{A})-1}(0)}(n) \leq G_{\alpha}(n) + G_{\phi_{l(\mathsf{A})-1}(0)}(0)$ for any $n \in \mathbb{N}$. Define $F(n) = G_{\alpha}(n) + G_{\phi_{l(\mathsf{A})-1}(0)}(0)$. By Lemma~\ref{lem:ub-g-and-l-via-q-h}, the function $F$ is monotone. Moreover, since $\alpha \in \mathsf{A}$, we have $G_{\alpha} \in \PredFuncClass{\mathsf{A}}$, by Example \ref{Exam:GInPred}. Therefore, as the addition operation and any constant function belong to $\mathcal{E}_2 \subseteq \PredFuncClass{\mathsf{A}}$, and $\PredFuncClass{\mathsf{A}}$ is closed under composition, we conclude that $F \in \PredFuncClass{\mathsf{A}}$.
    
To find $E$, as $l(\mathsf{A}) - 1 \geq 1$, by Lemma~\ref{lem:ub-g-and-l-via-q-h} and Lemma~\ref{kMonotonicityOfQ}, if $n \geq 3$, we have:
\[
F(n) \geq G_{\phi_{l(\mathsf{A})-1}(0)}(n) 
= Q_{l(\mathsf{A})-1}(n, 0) \geq Q_1(n, 0) =
\]
\[
Q_0^{(n)}(n, 0) \geq (n+1)^{(n+1)} \geq 2^n.
\]
Therefore, if we define $E(n) = F(n) + 4$, it follows that $E(n) \geq 2^n$ for any $n \in \mathbb{N}$. Since $F$ is monotone, $E$ is also monotone. Moreover, using a similar argument as in the case of $F$, we conclude that $E \in \PredFuncClass{\mathsf{A}}$. This completes the proof of the claim, and hence the case $l(\mathsf{A}) \geq 2$.

Finally, if $l(\mathsf{A})=1$, let $f \in \mathcal{E}_{l(\mathsf{A})+2}=\mathcal{E}_3$ to show that $f \in \PredFuncClass{\mathsf{A}}$. By Theorem~\ref{the:elem-characterizatons}, there exists a function $T_f \in \mathcal{E}_{3}$ such that $f(\bar{n})$ is computable in time $T_f(\RepLen{\bar{n}})$ for any $\bar{n} \in \mathbb{N}$. 
By Lemma~\ref{lem:grz-properties}, there exists $M \in \mathbb{N}$ such that
\[
T_f(\bar{m}) \leq 2^{2^{\iddots^{\max_j m_j}}}
\quad \text{for all } \bar{m} \in \mathbb{N}.
\]
where the number of $2$'s is $M$. Then, with the same type of argument as above, it is enough to find a function $E(n) \geq 2^n$ such that $E \in \PredFuncClass{\mathsf{A}}$. For that purpose, as $l(\mathsf{A})=1$, by definition, $\mathsf{A} \subseteq \Phi_{1}$ but $\mathsf{A} \nsubseteq \Psi_{\omega}$. Therefore, there is $\alpha \in \mathsf{A}$ such that $\alpha \in \Phi_{1}-\Psi_{\omega}$. Hence, by Lemma \ref{ExistenceOfWitness}, we have $G_{\alpha}(n) \geq 2^n$, for any $n \in \mathbb{N}$. Therefore, it is enough to set $E=G_{\alpha}$. As $\alpha \in \mathsf{A}$, we have $G_{\alpha} \in \PredFuncClass{\mathsf{A}}$, by Example \ref{Exam:GInPred}. Moreover, as $\alpha \in \Phi_{\omega}$, by Lemma \ref{lem:ub-g-and-l-via-q-h}, $G_{\alpha}$ is monotone which completes the proof.
\end{proof}


\section{Simulation of $\PredFuncClass{\mathsf{A}}$}
\label{sec:UpperBoundLength}
In this section, we follow the strategy outlined in Section~\ref{sec:main-theorem} to prove that $\PredFuncClass{\mathsf{A}} \subseteq \mathcal{E}_{l(\mathsf{A})+2}$, where $\mathsf{A} \subseteq \Phi_{\omega}$ is a bounded downset of ordinals. As discussed in Section~\ref{sec:main-theorem}, this requires three main ingredients. First, in Section~\ref{subsection-notation-system}, we present a notation system for the ordinals in $\Phi_{\omega}$. Second, in Section~\ref{subsec:l-r-comp}, up to the encoding of ordinals, we locate the functions $R(i, \mu, n)$ and $L_{\mu}(n)$ within the Grzegorczyk hierarchy. Third, in Section~\ref{subsec:bounding-lemma}, we establish an upper bound for predicative ordinal functions in terms of functions in the Grzegorczyk hierarchy. Finally, in Subsection~\ref{subsecSimulation}, we combine these ingredients to prove the claim that $\PredFuncClass{\mathsf{A}} \subseteq \mathcal{E}_{l(\mathsf{A})+2}$. 

\subsection{A Notation System for $\Phi_{\omega}$}
\label{subsection-notation-system}

Let $\Sigma := \{0, 1, (, ), ;, \bot \}$ be the alphabet. We provide a notation system for the ordinals in $\Phi_{\omega}$ in terms of strings over $\Sigma$. 
First, for $n \in \mathbb{N}$, let $\IntermCode{n}$ be the binary representation of $n$ as a string over $\Sigma$. Then, recursively define the following notation (or code) for the ordinals in $\Phi_{\omega}$ as strings in $\Sigma^\ast$, setting $\IntermCode{0} \Def ()$, and
\[
\IntermCode{\sum_{i=1}^m \Veb{k_i}(\beta_i)}
\Def
(\IntermCode{k_1};\IntermCode{\beta_1})\cdots(\IntermCode{k_m};\IntermCode{\beta_m}),
\]
where $k_i \in \mathbb{N}$ and $\beta_i \in \Phi_\omega$, for any $1 \leq i \leq m$. 
Note that, by Definition~\ref{Phi-sets} and Theorem~\ref{UniquenessTheorem}, the zero ordinal cannot be written in the form $\sum_{i=1}^m \Veb{k_i}(\beta_i)$, and for any non-zero ordinal in $\Phi_{\omega}$, such a representation exists and is unique. Hence, the coding function $\IntermCode{-} \colon \Phi_{\omega} \to \Sigma^*$ is well-defined. 

We define $\RepLen{\alpha}$ as the length of the string $\IntermCode{\alpha} \in \Sigma^*$, for any $\alpha \in \Phi_{\omega}$. Moreover, recall that $\RepLen{n}$ denotes the length of the binary representation of $n$, for any $n \in \mathbb{N}$.

\begin{definition}\label{Dfn: NumeralAndOrdinalFunc}
Let $X \subseteq \Omega$ be a set of ordinals. A function is called a \emph{numeral function with ordinals in $X$} if it takes as input only natural numbers and ordinals from $X$, and returns a natural number. Similarly, a function is called an \emph{ordinal function with ordinals in $X$} if it takes as input natural numbers and ordinals from $X$, and returns an output in $X$.
\end{definition}

\begin{remark}
Observe that any downset $\mathsf{A} \subseteq \Omega$ is closed under the predecessor function $p$. Therefore, the restriction of $p$ to $\mathsf{A}$, denoted by $p|_{\mathsf{A}}$ is an ordinal function with ordinals in $\mathsf{A}$.
\end{remark}

As both natural numbers and ordinals in $\Phi_{\omega}$ are encoded as strings over $\Sigma$, any numeral or ordinal function with ordinals in $\Phi_{\omega}$ naturally has a coded version over $\Sigma^*$. More precisely:

\begin{definition}\label{def:coded-func}
Given a numeral or ordinal function $f(\bar \alpha, \bar n)$ with ordinals in $X \subseteq \Phi_{\omega}$ (Definition \ref{Dfn: NumeralAndOrdinalFunc}), define its coded version $\CodedFunc{f}$ over $\Sigma^*$ by:
\[
\CodedFunc{f}(\bar{u}, \bar{v}) :=
\begin{cases}
\IntermCode{f(\bar{\alpha}, \bar{n})} & \exists \bar{\alpha} \in X\, \exists \bar{n} \in \mathbb{N}\, (\bar{u} = \IntermCode{\bar{\alpha}} \wedge \bar{v} = \IntermCode{\bar{n}}), \\
\bot & \text{otherwise.}
\end{cases}
\]
\end{definition} 

\begin{lemma}\label{lem:representable-n-geq-3}
Let $\mathsf{A}$ be either $\Phi_{\omega}$, $\Phi_k$, or $\Phi^m_k$ for some $k \geq 0$ and $m \geq 1$. Then:
\begin{itemize}
\item[$(i)$]
Deciding whether a given string is the code of an ordinal in $\mathsf{A}$, and if it is, testing whether the ordinal is zero, a successor or a limit is possible in polynomial time.
\item[$(ii)$] 
The function $\CodedFunc{p|_{\mathsf{A}}}: \Sigma^* \times \Sigma^* \to \Sigma^*$ is in $\GrzClassSig{3}$, where $p|_{\mathsf{A}}$ is the restriction of the ordinal predecessor function to $\mathsf{A}$.
\item[$(iii)$]
There is an expansive polynomial function $P$ with natural coefficients 
such that
$\RepLen{p(\alpha, n)} \leq P(\RepLen{\alpha}, {n})$, for any $\alpha \in \Phi_\omega$ and any $n \in \mathbb{N}$.
\end{itemize}
\end{lemma}
\begin{proof}
For~$(i)$, checking whether a given string $w \in \Sigma^*$ is the code of an ordinal in $\Phi_{\omega}$ can be done by a recursive algorithm: first, we check whether $w = ()$; if not, we verify whether $w = (u_1; v_1; \cdots ; u_l; v_l)$, where each $u_i$ is the binary code of a natural number $k_i$ and each $v_i$ is the code of an ordinal. It is clear that this algorithm runs in polynomial time. For $\Phi_k$, one also needs to add another check to each recursive call to ensure that $k_i < k$ for any $1 \leq i \leq l$. For $\Phi_k^m$, if $m=1$, since $\Phi^{1}_k = \Phi_k$, we already covered this case. If $m > 1$, one also needs to add another check to each recursive call to ensure that for any $1 \leq i \leq l$, either $k_i = k$ and $v_i$ is the code of an ordinal in $\Phi^{m-1}_k$, or $k_i < k$ and $v_i$ is the code of an ordinal in $\Phi^{m}_k$. It is clear that these modified algorithms also run in polynomial time.

For the second task, if the string is indeed the code of an ordinal $\alpha \in \mathsf{A}$, determining whether $\alpha$ is zero, a successor, or a limit ordinal is straightforward. An ordinal is zero if and only if its code is $()$, and by Theorem~\ref{UniquenessTheorem}, it is a successor if and only if $\IntermCode{\alpha}$ ends with $(\lceil 0 \rceil;())$. Otherwise, it is a limit ordinal. Using~$(i)$, all of these conditions can be checked in polynomial time.

For~$(ii)$, using~$(i)$, we present an algorithm to compute predecessors that runs in $\mathcal{E}_3$ time. First, we check whether the given string is the code of an ordinal in $\mathsf{A}$. If it is not, we output $\bot$. Otherwise, we determine whether the input ordinal is zero, a successor, or a limit. In the first case, the output is $()$. In the successor case, computing $\IntermCode{\alpha}$ from $\IntermCode{\alpha+1}$ can be done by removing the rightmost $(\lceil 0 \rceil;())$ from the code $\IntermCode{\alpha+1}$. Therefore, the computation in these cases can be performed in time $\RepLenCode{\alpha}^{O(1)}$.
The challenging part concerns the predecessor of limit ordinals. To address this case, we can easily use the definition of ordinal addition and Definition \ref{dfn:VeblenHierarchy} to see that, for any limit ordinal $\alpha = \langle\alpha_i\rangle \in \Phi_\omega$, we have:
	\[
		\alpha_n := 
		\begin{cases}
			\Veb{0}(\gamma)\cdot (n+1) & \text{if } \alpha = \omega^{\gamma+1} = \Veb{0}(\gamma+1)\\
			\Veb{0}(\gamma_n)
			& \text{if } \alpha = \omega^\gamma = \Veb{0}(\gamma), \text{ for } \gamma=\langle \gamma_i \rangle \text{ a limit}\\
			\Veb{k}^{(n)}(0)  & \text{if } \alpha = \Veb{k+1}(0)\\
			\Veb{k}^{(n)}(\Veb{k+1}(\gamma)+1) & \text{if } \alpha=\Veb{k+1}(\gamma+1)\\
			\Veb{k+1}(\gamma_n) & \text{if } \alpha=\Veb{k+1}(\gamma), \text{ for } \gamma=\langle \gamma_i \rangle \text{ a limit}\\
			\beta + \gamma_n & 	\text{if } \alpha = \beta + \gamma, \text{ for } \gamma=\langle \gamma_i \rangle \text{ a limit}\\
		\end{cases}
	\]
This provides a recursive algorithm for computing $\CodedFunc{p|_{\mathsf{A}}}(\lceil \alpha \rceil, \IntermCode{n})$. To explain, let $\alpha = \sum_{j=1}^r \phi_{k_j}(\gamma_j)$. The main task is to compute the expression $\CodedFunc{p|_{\mathsf{A}}}(\lceil \phi_{k_r}(\gamma_r) \rceil, \IntermCode{n})$. Observe that when $r = 1$, this is exactly the required value, while for $r > 1$, the desired code is obtained by appending $\lceil \sum_{j=1}^{r-1} \phi_{k_j}(\gamma_j) \rceil$ to the left side of $\CodedFunc{p|_{\mathsf{A}}}(\lceil \phi_{k_r}(\gamma_r) \rceil, \IntermCode{n})$.
Now, to compute $\CodedFunc{p|_{\mathsf{A}}}(\lceil \phi_k(\gamma) \rceil, \IntermCode{n})$, let $k = k_r$ and $\gamma = \gamma_r$. Observe that $\lceil \phi_k(\gamma) \rceil$ can be directly accessed from $\lceil \alpha \rceil$ by extracting its last segment. There are two cases to consider.
If $\gamma$ is zero or a successor, then computing $\CodedFunc{p|_{\mathsf{A}}}(\lceil \phi_k(\gamma) \rceil, \IntermCode{n})$ is straightforward and can be done in time $O(\RepLen{\alpha} n)$. This is because, in these cases, inspection of the explicit expression for $\alpha_n$ listed above reveals that its length is bounded by $O(\RepLen{\alpha} n)$.  
If instead $\gamma$ is a limit ordinal, the computation recurses into $\lceil \gamma \rceil$. Each step of the recursion just described takes $O(\RepLen{\alpha} n)$ time, and the recursion depth is $O(\RepLen{\alpha})$. Thus, the algorithm in this case runs in time $O(\RepLen{\alpha}^2 n)$.  
Therefore, in terms of the representation length $\RepLenCode{n}$, the computation takes $\RepLenCode{\alpha}^{O(1)} 2^{O(\RepLenCode{n})}$ steps. Hence, $\CodedFunc{p|_{\mathsf{A}}}$ is computable in $\mathcal{E}_3$ time and thus belongs to the class $\GrzClassSig{3}$.

For $(iii)$, to define the polynomial $P$, observe from the analysis above that in each recursive step, we add at most $O(\RepLen{\alpha}n)$ symbols. Therefore, we have $\RepLen{p(\alpha, n)} \leq O(\RepLen{\alpha}^2 n)$ for any $\alpha \in \Phi_\omega$ and any $n \in \mathbb{N}$. It is thus sufficient to define $P(x, n) := Cx^2 n + x + n+D$, for sufficiently large constants $C, D \in \mathbb{N}$.
\end{proof}

\begin{remark}
As observed in the proof of Lemma~\ref{lem:representable-n-geq-3}, some cases that arise in the computation of $p(\phi_{k+1}(\gamma), n)$
invoke $n$ iterations of $\phi_k$. As a result, the length of $\CodedFunc{p}(\IntermCode{\alpha}, \IntermCode{n})$ becomes polynomial in $n$, and hence exponential in $\RepLenCode{n}$. Therefore, the function $\CodedFunc{p}$ cannot be computed in less than exponential time, and consequently, $\GrzClassSig{3}$ is the lowest level of the Grzegorczyk hierarchy to include $\CodedFunc{p}$.
\end{remark}

\subsubsection{A linear-space notation system for $\Psi_{\omega}$}
\label{subsec:rep-omega-omega}

A special case of the main objective of this section is to prove that $\PredFuncClass{\Psi_{\omega}} \subseteq \mathcal{E}_2$. 
This requires simulating the corresponding predicative ordinal recursion within the class $\mathcal{E}_2$. Since $\mathcal{E}_2$, the class of linear-space computable functions, is more resource-sensitive than the broader class of elementary functions on which our original ordinal notation system was based, we now require a more refined (i.e., linear-space efficient) notation system for the ordinals in $\Psi_\omega$.

Recall that, by Remark~\ref{LengthBoundForPsi}, any non-zero $\alpha \in \Psi_\omega$
has a unique normal form
$\sum_{i=1}^l \omega^{p_i} c_i$,
where $l \geq 1$,
the
$c_i$'s are
non-zero finite ordinals,
and
$p_i \neq p_{i+1}$ for all 
$1 \leq i < l$.
Therefore, for such ordinals, this presentation allows for a simpler and more compact notation system over the alphabet~$\Sigma$. 
We define this new notation as follows. If $\alpha = 0$, we set $\IntermCode{\alpha}_E$ as the string $()$. If $\alpha \neq 0$ and has the normal form
$\alpha = \sum_{i=1}^l \omega^{p_i} c_i$, define
\[
\IntermCode{\alpha}_E \Def (
\IntermCode{p_{1}}; \IntermCode{c_{1}};\cdots;\IntermCode{p_{l}};\IntermCode{c_l}).
\]
To ensure a uniform notation, we also extend the new notation $\IntermCode{-}_E$ to numbers; that is, if $n \in \NatSet$, we sometimes denote $\IntermCode{n}$ by $\IntermCode{n}_E$. Moreover, define $\RepLen{\alpha}_E$ and $\RepLen{n}_E$ as the lengths of the strings $\IntermCode{\alpha}_E$ and $\IntermCode{n} = \IntermCode{n}_E$, respectively.

\begin{definition}\label{def:coded-func-omega-omega}
Given a numeral or ordinal function $f(\bar \alpha, \bar n)$ with ordinals in $X \subseteq \Psi_{\omega}$, define its coded version $\CodedFunc{f}_E$ over $\Sigma^*$ by:
\[
\CodedFunc{f}_E(\bar{u}, \bar{v}) :=
\begin{cases}
\IntermCode{f(\bar{\alpha}, \bar{n})}_E & \exists \bar{\alpha} \in X\, \exists \bar{n} \in \mathbb{N}\, (\bar{u} = \IntermCode{\bar{\alpha}}_E \wedge \bar{v} = \IntermCode{\bar{n}}_E), \\
\bot & \text{otherwise},
\end{cases}
\]
\end{definition} 

\begin{lemma}\label{lem:representable-n-geq-3III}
Let $\mathsf{A}$ be either $\Psi_{\omega}$, $\Psi_k$, or $\Psi^m_k$ for some $k, m \geq 1$. Then:
\begin{itemize}
\item[$(i)$]
Deciding whether a given string $w \in \Sigma^*$ is of the form $\IntermCode{\alpha}_E$, for some $\alpha \in \mathsf{A}$, and, if it is, testing whether $\alpha$ is zero, a successor, or a limit is possible in linear space.
\item[$(ii)$] 
The function $\CodedFunc{p|_{\mathsf{A}}}_E: \Sigma^* \times \Sigma^* \to \Sigma^*$ is computable in linear space, where $p|_{\mathsf{A}}$ is the restriction of the ordinal predecessor function to $\mathsf{A}$.
\end{itemize}
\end{lemma}
\begin{proof}
For~$(i)$, checking whether a given string $w \in \Sigma^*$ is of the form $\IntermCode{\alpha}_E$, for some $\alpha \in \Psi_{\omega}$, can be done using the following algorithm: first, we check whether $w = ()$; if not, we verify whether $w$ is in the form $(\tau_1; \tau_2; \ldots; \tau_{2l})$ for some $l \geq 1$, such that $\tau_{2i-1}$ is the binary code of a number $p_i$, for any $1 \leq i \leq l$, and $\tau_{2i}$ is the binary code of a non-zero number $c_i$, for any $1 \leq i \leq l$, and, second, whether 
$\tau_{2j-1} \neq \tau_{2j+1}$,
for any  $1 \leq j < l$. It is clear that the algorithm only checks the basic form, whether some strings are over $\{0, 1\}$ and start with $1$, along with verifying some equalities, and hence runs in linear space. For $\Psi_k$, we should also add another check to ensure that $p_i < k$ for any $1 \leq i \leq l$. For $\Psi^m_k$, we also need to check the existence of $1 \leq r \leq l$ such that $r \leq m$ and $p_{j}<p_{j+1}$ for any $r+1 \leq j < l$. It is clear that these checks also run in linear space.

For the second task, if the string is indeed the code of an ordinal $\alpha \in \mathsf{A}$, determining whether $\alpha$ is zero, a successor, or a limit ordinal is straightforward. An ordinal is zero if and only if its code is $()$, and by Theorem~\ref{UniquenessTheorem}, it is a successor if $\tau_{2l-1} = 0$; and it is a limit otherwise. Using the previous part, all of these conditions can be checked in linear space.

For $(ii)$, first, we check in linear space whether $\alpha$ is a code of an ordinal in $\mathsf{A}$. If it is not, we output $\bot$. Otherwise, we determine whether the ordinal is zero, a successor, or a limit. If it is zero, the predecessor is $()$. If it is a successor, then $\IntermCode{\alpha}_E$ must be in the form 
$(\IntermCode{p_{1}}; \IntermCode{c_{1}};\cdots;\IntermCode{p_{l}};\IntermCode{c_l})$, where $p_l = 0$. Then, if $c_l = 1$, we have $\IntermCode{p(\alpha, n)}_E = (\IntermCode{p_{1}}; \IntermCode{c_{1}};\cdots;\IntermCode{p_{l-1}};\IntermCode{c_{l-1}})$ and if 
$c_l \geq 2$, then $\IntermCode{p(\alpha, n)}_E = (\IntermCode{p_{1}}; \IntermCode{c_{1}};\cdots;\IntermCode{p_{l}};\IntermCode{c_l - 1})$.
It is clear that in these cases, the computation of the predecessor is possible in linear space.

If $\alpha$ is a limit ordinal with the normal form $\alpha = \sum_{i=1}^l\omega^{p_i} \cdot c_i$, we have $p_l \geq 1$. Hence
\[
p(\alpha, n) = \sum_{i=1}^{l-1} \omega^{p_i} \cdot c_i + \omega^{p_l} \cdot (c_l - 1) + \omega^{p_l - 1} \cdot (n + 1). \qquad (*)
\]
To compute $\IntermCode{p(\alpha, n)}_E$, it remains to transform $(*)$ into normal form (gathering coefficients of adjacent terms with the same power, and removing a term whose coefficient now is $0$), and output its representation. If $c_l > 1$, then $(*)$ is already the normal form of $p(\alpha, n)$. Hence, we have
\[
\IntermCode{p(\alpha, n)}_E = (
\IntermCode{p_1}; \IntermCode{c_1}; \cdots; \IntermCode{p_l}; \IntermCode{c_l - 1}; \IntermCode{p_l - 1}; \IntermCode{n + 1}).
\]
If $c_l = 1$, we can simplify $p(\alpha, n)$ into
\[
p(\alpha, n) = \sum_{i=1}^{l-1} \omega^{p_i} \cdot c_i + \omega^{p_l - 1} \cdot (n + 1). \qquad (**)
\]
There are two cases to consider, either $p_{l-1} = p_l - 1$ or not. If $p_{l-1} \neq p_l - 1$, $(**)$ is already the normal form of $p(\alpha, n)$. Hence, we have
\[
\IntermCode{p(\alpha, n)}_E = (
\IntermCode{p_1}; \IntermCode{c_1}; \cdots; \IntermCode{p_{l-1}}; \IntermCode{c_{l-1}}; \IntermCode{p_l - 1}; \IntermCode{n + 1}).
\]
Otherwise, the normal form of $p(\alpha, n)$ is
\[
p(\alpha, n) = \sum_{i=1}^{l-2} \omega^{p_i} \cdot c_i + \omega^{p_{l-1}} \cdot (c_{l-1}+ n + 1). 
\]
Therefore, we have
$\IntermCode{p(\alpha, n)}_E = (
\IntermCode{p_1}; \IntermCode{c_1}; \cdots; \IntermCode{p_{l-1}}; \IntermCode{c_{l-1} + n + 1})$. 
It is clear that in computing $\IntermCode{p(\alpha, n)}_E$ in this case, we essentially check some equalities and perform some additions. As these operations run in linear space, $\CodedFunc{p|_{\mathsf{A}}}_E$ is computable in linear space.
\end{proof}

To control the computations that run on the ordinals in 
$\Psi_\omega$, along with their codes, a notion of norm is also required.

\begin{definition}\label{def:sum-norm-tuples}
For any $\alpha \in \Psi_\omega$, if $\alpha = 0$, define $|\alpha| = 0$; and if $\alpha \neq 0$, define $\SumNorm{\alpha} \Def \sum_{i=1}^{l} c_i$, where $\alpha = \sum_{i=1}^l \omega^{p_i} \cdot c_i$ is the normal form of $\alpha$.
\end{definition}

For any natural number $n \in \mathbb{N}$, the length of the binary representation of $n$, i.e., $\RepLen{n}$, is linearly related to $\log(n+1)$. More precisely,
\[
\log(n+1) \leq \RepLen{n} \leq \log(n+1) + 1
\]
for any $n \in \mathbb{N}$.
The same claim holds for ordinals in $\Psi^m_k$, where $k, m \geq 1$:

\begin{lemma}
\label{lem:relate-norms-linear}
Let  $k,m \in \mathbb{N}^{\geq 1}$ be fixed numbers. Then, 
there are $a, b\in \NatSet^{\geq 1}$ such that
$\log(\SumNorm{\alpha}+1)\leq  \RepLen{\alpha}_E$ and $\RepLen{\alpha}_E \leq a \cdot \log(\SumNorm{\alpha}+1)+b$,
for any $\alpha \in \Psi_k^m$.
\end{lemma}
\begin{proof}
To make the proof easier to read, we first examine the conditions under which $a$ and $b$ satisfy the claim, and at the end, we show that the conditions have a solution.
First, if $\alpha = 0$, then $|\alpha| = 0$. Hence, $\log(\SumNorm{\alpha} + 1) = 0 \leq \RepLen{\alpha}_E$, and there is nothing to check. For $\RepLen{\alpha}_E \leq a \cdot \log(\SumNorm{\alpha} + 1) + b$, since $\RepLenCode{0} = 2$, it is enough to set $b \geq 2$. If $\alpha \neq 0$, let 
$\alpha = \sum_{i=1}^l \omega^{p_i} c_i$
be the normal form of $\alpha$. Hence,
$\IntermCode{\alpha}_E = (
\IntermCode{p_{1}}; \IntermCode{c_{1}};\cdots;\allowbreak\IntermCode{p_{l}};\IntermCode{c_l})$, which implies $\RepLen{\alpha}_E = 2l + 1 + 
 \sum_{i} \RepLen{p_i}+
 \sum_{i} \RepLen{c_i}$.
For the first claim, first observe that for any $x, y \in \mathbb{R}^{\geq 2}$, we have $x+y \leq xy$. 
Therefore, as $l \geq 1$, we obtain:
        \[
        \log(\SumNorm{\alpha}+1)
        \leq  
        \log(1+\sum_i c_i)
         \leq 
        \log(\sum_i (c_i+2))
        \leq 
        \log(\prod_i (c_i+2))=
        \]
        \[
        \sum_i \log(c_i + 2)
        \leq \sum_i \log(2(c_i+1))
        = l + \sum_i \log(c_i+1)
        \leq 
        l + \sum_i \RepLen{c_i}
        \leq \RepLen{\alpha}_E.
        \]
For the second claim, as $c_i \leq |\alpha|$ and $p_i < k$ for any $1 \leq i \leq l$, and $l \leq k+m$ (Remark~\ref{LengthBoundForPsi}), we have:
 \[
 \RepLen{\alpha}_E = 2l + 1 + 
 \sum_{i} \RepLen{p_i}+
 \sum_{i} \RepLen{c_i}
 \leq 2l+1+ 
 \sum_{i} (\log(p_i+1)+1)
 +
 \sum_{i} (\log(c_i+1)+1)
 \]
 \[
        \leq 4l+1+ 
        \sum_{i} \log(p_i+1) +
        \sum_{i} \log(c_i+1) 
        \leq C(m,k) + \sum_{i} \log(c_i+1)
\]
\[
\leq C(m,k)+ l\log(\SumNorm{\alpha}+1)
        \leq 
        C(m,k)+ (m+k)\log(\SumNorm{\alpha}+1)
 \]
where $C(m,k) \geq 4l+1 \geq 2$ is a constant in terms of $m$ and $k$. 
Therefore, it is enough to set $a \Def k+m$ and $b \Def C(m,k)$. Notice that $b \geq 2$ as required for the case $\alpha=0$.
\end{proof}

\begin{lemma}\label{lem:omega-k-case}
$\SumNorm{{p(\alpha, n)}} \leq 
\SumNorm{{\alpha}} + n$, for 
any $\alpha \in \Psi_\omega$ and $n \in \NatSet$.
\end{lemma}
\begin{proof}
If $\alpha = 0$, we have $p(\alpha, n) = 0$, and hence there is nothing to prove. If $\alpha \neq 0$, let $\sum_{i=1}^{l} \omega^{p_i} \cdot c_i$ be the normal form of $\alpha$. Note that $c_i \geq 1$ for all $1 \leq i \leq l$.
If $\alpha$ is a successor, then $p_l = 0$, and hence
\[
p(\alpha, n) = \left( \sum_{i=1}^{l-1} \omega^{p_i} \cdot c_i \right) + (c_l - 1).
\]
Therefore,
$\SumNorm{p(\alpha, n)} = \SumNorm{\alpha} - 1 < \SumNorm{\alpha} \leq \SumNorm{\alpha} + n$.
If $\alpha$ is a limit, then $p_l > 0$, and hence
\[
p(\alpha, n) = \left( \sum_{i=1}^{l-1} \omega^{p_i} \cdot c_i \right) + \omega^{p_l} \cdot (c_l - 1) + \omega^{p_l - 1} \cdot (n + 1).
\]
Therefore,
$\SumNorm{p(\alpha, n)} \leq \sum_{i=1}^{l} c_i + n = \SumNorm{\alpha} + n$.
\end{proof}

\subsection{Computing $L_{\mu}(n)$ and $R(i, \mu, n)$}\label{subsec:l-r-comp}

Recall (see Definition~\ref{def:slow-growing-hierarchy}) that $R^b(i, \mu, n)$ is obtained by applying the $n$-predecessor function $i$ times to $\mu$, and that $L_{\mu}(n)$ is the least $l \in \mathbb{N}$ such that $R^b(l, \mu, n)=0$. We called the decreasing sequence $\{R^b(i, \mu, n)\}_{i=0}^{L_{\mu}(n)}$ the sequence of $n$-predecessors of $\mu$, and defined $R(i, \mu, n)=R^b(L_{\mu}(n) \dotminus i, \mu, n)$ as the function that enumerates the sequence of $n$-predecessors of $\mu$ in a forward manner (i.e., from $0$ to $\mu$).

In this subsection, we complete the second step of the strategy outlined in Section~\ref{sec:main-theorem}.  
Up to the encoding of ordinals and numbers as strings over $\Sigma$, we show that the functions $R(i,\mu,n)$ and $L_{\mu}(n)$ lie in the appropriate level of the Grzegorczyk hierarchy. For this purpose, since $R$ is defined in terms of $R^b$ and $L$, it suffices to locate $R^b$ and $L$ within the hierarchy. Observe that $R^b$ is defined by primitive recursion, and $L$ by minimization.
Hence, to place these functions at the suitable level of the hierarchy, we first upper bound them by functions at this level, and then appeal to the closure of the Grzegorczyk classes under bounded recursion and bounded minimization (see Section~\ref{sec:preliminaries}). We derive these required bounds separately for the restrictions of $R^b$ and $L$ to $\Phi^{m}_{k-3}$ (with $k \geq 3$, $m \geq 1$) and to $\Psi_k^m$ (with $k,m \geq 1$) in the following subsubsections.

\subsubsection{The case $\Phi^{m}_{k-3}$, for $k \geq 3$ and $m \geq 1$}

In this subsubsection, we aim to show that $\CodedFunc{R|_{\Phi^m_{k-3}}}$ and $\CodedFunc{L|_{\Phi^m_{k-3}}}$
belong to $\GrzClassSig{k}$. As discussed above, this requires establishing upper bounds for $\RepLenCode{R^b(i,\mu,n)}$ and $L_{\mu}(n)$.  
We begin with the simpler case of $\RepLenCode{R^b(i, \mu, n)}$ and then turn to the more involved case of $L_{\mu}(n)$. 

\begin{lemma}\label{lem:bounds-rec-sequence}
There is a monotone function $M_R \in \mathcal{E}_3$ such that $\RepLenCode{R^b(i, \mu, n)} \allowbreak\leq M_R(\RepLenCode{i}, \RepLenCode{\mu}, \RepLenCode{n})$, for any $\mu \in \Phi_{\omega}$ and any $n, i \in \mathbb{N}$.
\end{lemma}
\begin{proof}
By Lemma~\ref{lem:representable-n-geq-3}, there is a monotone and expansive polynomial $P$ such that $\RepLenCode{p(\mu, n)} \leq P(\RepLenCode{\mu}, n)$, for any $\mu \in \Phi_{\omega}$ and $n \in \mathbb{N}$. 
Define the function $M'$ by $M'(i, x, y)=P^i(x, y)$, where $P^i$ is the result of the $i$-iteration of $P$ on the first argument. Note that $P^0(x, y)=x$ and as $P$ is monotone and expansive, $M'$ is also monotone in all arguments. Moreover, as $P \in \mathcal{E}_2$, we have $M' \in \mathcal{E}_3$, by Lemma \ref{lem:grz-properties}.
Now, by induction on $i \in \mathbb{N}$, we prove 
$\RepLenCode{R^b(i, \mu, n)} \leq M'(i, \RepLenCode{\mu}, n)$. 
For $i=0$, we have
    $\RepLenCode{R^b(0, \mu, n)} = \RepLenCode{\mu}= P^0(\RepLenCode{\NormalOrdA}, n)=M'(0, \RepLenCode{\NormalOrdA}, n)$.
    For the inductive step, assume the claim for $i$. We know that
    $\RepLenCode{R^b(i+1, \mu, n)}
    =
    \RepLenCode{p(R^b(i, \mu, n),n)}
    \leq P(\RepLenCode{R^b(i, \mu, n)}, n)
    \leq P(P^{i}(\RepLenCode{\NormalOrdA}, n), n)
=M'(i+1, \RepLenCode{\NormalOrdA}, n)$.  This establishes $\RepLenCode{R^b(i, \mu, n)} \leq M'(i, \RepLenCode{\mu}, n)$. It remains to obtain the bound in terms of $\RepLenCode{i}$ and $\RepLenCode{n}$ instead of $i$ and $n$.
Define 
$M_R(j, x, y) \Def M'(2^j,\allowbreak x, 2^{y})$ and notice that $M_R$ is monotone and $M_R \in \mathcal{E}_3$. Finally, as $M'$ is monotone, we reach
$\RepLenCode{R^b(i, \mu, n)} \leq M_R(\RepLenCode{i}, \RepLenCode{\mu}, \RepLenCode{n})$, for any $\mu \in \Phi_{\omega}$ and any $n, i \in \mathbb{N}$.
\end{proof}

To provide an upper bound for $L_{\mu}(n)$, we first need to prove some basic properties. First, as $L_{\phi_k(\mu)}(n)$ is controlled by the function $H_k$ and $L_{\mu}(n)$ (Lemma~\ref{lem:ub-g-and-l-via-q-h}), our first task is clearly to locate the functions $\{H_k\}_{k=0}^{\infty}$ in the Grzegorczyk hierarchy.

\begin{lemma}\label{lem:veb-hie-H-func}
$H_i \in \mathcal{E}_k$, for any $k \geq 3$ and $i \leq k - 3$.
\end{lemma}
\begin{proof}
We prove the claim by an induction on $k$. For $k=3$, as $H_0$ is clearly elementary, it is in $\GrzClass{3}$ and hence there is nothing to prove. For the induction step, assume the claim for $k$ to prove it for $k+1$. Let $i \leq k-2$. If $i=0$, similar to the base case, i.e., $k=0$, there is nothing to prove. If $i \geq 1$, let $i' = i-1$. Note that $i' \leq k-3$. Thus, by the induction hypothesis, $H_{i'} \in \GrzClass{k}$. By Lemma \ref{lem:grz-properties}, the function $H_{i'}^{(x)}$ is in $\GrzClass{k+1}$.
    Since $H_i=H_{i'+1}$ is defined by a primitive recursion from $H_{i'}^{(x)} \in \GrzClass{k+1}$, if we show that it is bounded by a function in $\GrzClass{k+1}$, we obtain $H_{i'+1} \in \mathcal{E}_{k+1}$. 
    
    Now, since $H_{i'} \in \mathcal{E}_k$, by Lemma~\ref{lem:grz-properties}, there is a constant $M \in \NatSet$ such that $H_{i'}(x,y) \leq E(\max\{x,y\})$, for all $x,y \in \NatSet$, where $E = h_{k-1}^M$. As $h_{k-1}$ is strictly monotone and expansive, so is $E$. Therefore, by induction on $r \in \mathbb{N}$, it follows that $E^{(r)}(x)$ is expansive and strictly monotone in $x$. Since $E^{(r)}(x)$ is always a natural number, we obtain 
$E^{(r)}(x) + y \leq E^{(r)}(x+y)$,
for all $r, x, y \in \mathbb{N}$.
 Moreover, by an induction on $r \in \mathbb{N}$ and the expansiveness of $E^{(j)}$'s, we can show that for any $r \in \mathbb{N}$, we have $H_{i'}^{(r)}(x, y) \leq E^{(r)}(\max\{x, y\})$, for any $x, y \in \mathbb{N}$.

	Let $F(x,y) := E^{x(y+1)}((y+1)(x+y+1))$. As $E \in \GrzClass{k}$, we have $F \in \GrzClass{k+1}$, by Lemma~\ref{lem:grz-properties}. Therefore, it is enough to show that $H_{i}(x, y) \leq F(x, y)$, for any $x, y \in \mathbb{N}$. To prove this claim, we proceed by induction on $y$.
	For the base case, since $H_{i'}(x,0) \leq E(x)$, we have 
	\[
    H_{i}(x,0) = H_{i'}^{(x)}(x,0)+1
    \leq E^{(x)}(x)+1
    \leq 
    E^{(x)}(x+1)
    = F(x,0).
    \]
	For the induction step, by the induction hypothesis, $H_{i}(x,y) \leq F(x,y)$. Thus, 
\[
H_{i}(x, y) + 1 \leq F(x, y) + 1 = E^{x(y+1)}((y+1)(x + y \textcolor{black}{+ 1})) + 1 
\]
\[
\leq E^{x(y+1)}((y+1)(x + y \textcolor{black}{+ 1})) + (y+1) 
\]
\[
\leq E^{x(y+1)}((y+1)(x + y\textcolor{black}{+ 1}) + (y+1)) 
= E^{x(y+1)}((y+1)(x + y \textcolor{black}{+ 2})).
\]   
    Thus, $ H_{i}(x,y)+1 \leq E^{x(y+1)}((y+1)(x+y\textcolor{black}{+ 2}))$.
	Hence, using the monotonicity of $E$, for any $r \in \mathbb{N}$ we have:
\[
H_{i'}^{(r)}(x, H_{i}(x, y) + 1) \leq E^{(r)}\big(\max\{x, H_{i}(x, y) + 1\}\big)
\leq E^{(r)}(x + H_{i}(x, y) + 1) 
\]
\[
\leq E^{(r)}\big(E^{x(y+1)}((y+1)(x + y \textcolor{black}{+ 2})) + x\big) 
\leq E^{(r)}\big(E^{x(y+1)}((y+1)(x + y \textcolor{black}{+ 2})+x)\big).
\]
Therefore, putting $r=x$, we have:
   \[
H_i(x, y+1) = H_{i'}^{(x)}(x, H_i(x, y) + 1) \textcolor{black}{+1}
\leq E^{(x)}\big(E^{x(y+1)}((y+1)(x + y \textcolor{black}{+ 2}) \textcolor{black}{+ x})\big) 
\textcolor{black}{+ 1}
\]
\[
\leq E^{(x)}\big(E^{x(y+1)}((y+1)(x + y \textcolor{black}{+ 2}) \textcolor{black}{+ x + 1})\big)
\]
\[
\leq E^{x(y+2)}((y+2)(x + y \textcolor{black}{+ 2}))  = F(x, y+1).
\]
Therefore, $H_{i}(x,y) \leq F(x,y)$, for any $x, y \in \mathbb{N}$. This implies $H_i \in \mathcal{E}_{k+1}$.
\end{proof}

Now, to present the promised upper bound functions for $L_{\mu}(n)$, for each $k \geq 0$, define the family $\{ J_m^k : \NatSet^3 \to \NatSet \}_{m \in \NatSet}$ of numeral functions by:
\begin{align*}
&\,J_0^0(x, n, y) \Def y, &\\
&\,J_{m+1}^0(x, n, y) \Def x \cdot (n+2)^{J_m^0(x, n, y)}, &\\
&\,J_0^{k+1}(0, n, y) \Def y, &\\
&J_0^{k+1}(x + 1, n, y) \Def (x + 1) \cdot H_k\big(n, (x + 1) \cdot J_0^{k+1}(x, n, y)\big),& \\
&\,J_{m+1}^{k+1}(x, n, y) \Def x \cdot J_0^{k+1}\big(x, n, H_{k+1}\big(n, J_m^{k+1}(x, n, y)\big)\big).
\end{align*}

In the following, we establish some properties of this family.
\begin{lemma}\label{lem:mon-exp-js}
\begin{itemize}
    \item[$(i)$] 
$J_m^k$ is monotone in all arguments, for any $k, m \in \NatSet$.
    \item[$(ii)$] 
For any $m, k, x, n, y \in \NatSet$, we have
   \[
    x \cdot H_{k+1}(n,J_{m}^{k+1}(x,n,y))
    +
    x \cdot H_{k}(n,J_{m+1}^{k+1}(x,n,y)) \leq  J_{m+1}^{k+1}(x+1,n,y).
\]
\end{itemize}
\end{lemma}
\begin{proof}
For $(i)$, first, recall that for any $k \in \mathbb{N}$, by Lemma~\ref{MonotinicityOfQ}, the function $H_k$ is monotone in both arguments. Now, if $k=0$, by induction on $m$, one can easily prove the monotonicity of $J^0_m$, for any $m \in \mathbb{N}$. For $k > 0$, we prove the claim by induction on $m$. If $m=0$, the monotonicity of $J_0^k$ in $n$ and $y$ is trivial by the monotonicity of $H_{k-1}$. For its monotonicity in $x$, we have to show that $J^k_0(x+1, n, y) \geq J^k_0(x, n, y)$ for any $x, n, y \in \mathbb{N}$. For that purpose, it is enough to use $H_{k-1}(a, b) \geq b$ from Lemma \ref{ExpansiveLemma} to obtain:
\[
J^{k}_0(x+1, n, y)=(x + 1) \cdot H_{k-1}\big(n, (x + 1) \cdot J_0^{k}(x, n, y)\big)
\]
\[
\geq H_{k-1}\big(n, (x + 1) \cdot J_0^{k}(x, n, y)\big) \geq (x+1) \cdot J_0^{k}(x, n, y) \geq J_0^{k}(x, n, y).
\]
This completes the case $m=0$. The induction step is trivial by the monotonicity of $H_k$ and $J_0^k$.

For $(ii)$, by the definition of $J_0^{k+1}$ and the monotonicity of $J_0^{k+1}$, $J_{m}^{k+1}$, and $H_{k+1}$, we have:  
\begin{align*}
    x \cdot H_{k+1}(n,J^{k+1}_{m}(x,n,y))&=
     x \cdot J^{k+1}_0(0, n, H_{k+1}(n,J^{k+1}_{m}(x,n,y)))\\
    &\leq
    x \cdot J^{k+1}_0(x+1,n,H_{k+1}(n,J^{k+1}_{m}(x+1,n,y))) \quad (*)
    \end{align*}
    \noindent and using the definition of $J^{k+1}_{m+1}$, we also have:
    \begin{align*}
    &x \cdot H_{k}(n,J^{k+1}_{m+1}(x,n,y))\\
    &\leq 
    (x+1) \cdot H_{k}(n,(x+1)\cdot J^{k+1}_0(x,n, H_{k+1}(n, J^{k+1}_{m}(x,n,y))))\\
    &= J^{k+1}_0(x+1,n, H_{k+1}(n, J^{k+1}_{m}(x,n,y)))\\ 
    &\leq J^{k+1}_0(x+1,n, H_{k+1}(n, J^{k+1}_{m}(x+1,n,y))). \quad (**)
    \end{align*}
Therefore, using the definition of $J^{k+1}_{m+1}$,   
    the sum of the right sides of $(*)$ and $(**)$ is upper bounded by $J^{k+1}_{m+1}(x+1,n,y)$.
\end{proof}

As $J^k_m$'s are intended to play the role of upper bounds for $L$, we clearly need to locate them in the Grzegorczyk hierarchy.

\begin{lemma}\label{JIsInE}
    $J^{k-3}_{m} \in \GrzClass{k}$, for any $k \geq 3$ and $m \geq 0$.
\end{lemma}
\begin{proof}
If $k=3$, by induction on $m$, it is clear that $J_m^0 \in \mathcal{E}_3$. If $k \geq 4$, we prove the claim by induction on $m$. For $m=0$, we have to prove $J^{k-3}_0 \in \mathcal{E}_k$. As $J^{k-3}_0$ is defined by primitive recursion, using a composition of $H_{k-4}$ and polynomials, and since $H_{k-4} \in \mathcal{E}_{k-1} \subseteq \mathcal{E}_k$ by Lemma~\ref{lem:veb-hie-H-func}, we only need to provide an upper bound for $J_0^{k-3}$ that is also in $\GrzClass{k}$. For that purpose, define
\[
E(x,n,z) \Def (x+1) \cdot H_{k-4}(n,(x+1) \cdot z).
\]
Denote $u$ iterations of $E$ on its third argument by $E^{(u)}$, i.e., $E^{(0)}(x,n,z)=z$ and $E^{(u+1)}(x,n,z)=E(x,n,E^{(u)}(x,n,z))$. As $H_{k-4} \in \GrzClass{k-1}$, we have $E \in \GrzClass{k-1}$, and hence its iteration, i.e., $E^{(u)}(x,n,z)$, is in $\GrzClass{k}$ by Lemma~\ref{lem:grz-properties}. Therefore, it is enough to prove $J^{k-3}_0(x,n,y) \leq E^{(x)}(x,n,y)$, for any $x, n, y \in \mathbb{N}$. 

We prove this claim by induction on $x$. For $x=0$, we clearly have $J^{k-3}_0(0,n,y) = y = E^{(0)}(0,n,y)$. For the induction step, by the induction hypothesis and the monotonicity of $H_{k-4}$ proved in Lemma~\ref{MonotinicityOfQ}, we have
\[
J^{k-3}_0(x+1,n,y) \leq 
(x+1) \cdot H_{k-4}(n,(x+1) \cdot E^{(x)}(x,n,y))
= E^{(x+1)}(x,n,y).
\]
Therefore, for any $x,n,y \in \mathbb{N}$, we have $J^{k-3}_0(x,n,y) \leq E^{(x)}(x,n,y)$. This completes the proof of the case $m=0$ in showing that $J_m^{k-3} \in \mathcal{E}_k$.

For the inductive step, note that $J^{k-3}_{m+1}$ is defined by composition from $J^{k-3}_{0}$, $H_{k-3}$, $J^{k-3}_{m}$, and a polynomial. By the induction hypothesis, we have $J^{k-3}_m \in \mathcal{E}_k$, and we have just proved that $J^{k-3}_0 \in \mathcal{E}_k$. Moreover, by Lemma~\ref{lem:veb-hie-H-func}, $H_{k-3} \in \GrzClass{k}$. Therefore, since $\mathcal{E}_k$ is closed under composition and contains all polynomials, the claim follows for $m+1$.
\end{proof}

Now, we are ready to prove the promised upper bound:

\begin{lemma}\label{lem:bound-on-g}
    Let $k \geq 3$ and $m \geq 1$. Then, 
    $
    L_\NormalOrdA(n) 
    \leq
    J_{m-1}^{k-3}(\RepLenCode{\NormalOrdA}, n, 0)
    $, for any $\NormalOrdA \in \Phi^m_{k-3}$ and any $n \in \mathbb{N}$.
\end{lemma}
\begin{proof}
If $k=3$, we prove the claim by induction on $m$. If $m=1$, then $\NormalOrdA = 0$ and the claim holds as $L_0(n)=0$. For the inductive step, let $m \geq 2$ and assume the claim for $m-1$ to prove it for $m$. If $\mu=0$, again there is nothing to prove. Otherwise, as $\NormalOrdA \in \Phi^{m}_{k-3}=\Phi_0^{m}$, we can write $\mu$ as $\sum_{i=1}^{l} \Veb{0}(\gamma_{i})$, where $\gamma_i \in \Phi^{m-1}_{k-3}=\Phi_0^{m-1}$, for any $1 \leq i \leq l$. As $m \geq 2$, by the induction hypothesis, $L_{\gamma_i}(n) \leq J_{m-2}^{0}(\RepLenCode{\gamma_i}, n, 0)$, for any $1 \leq i \leq l$. 
Therefore, using the monotonicity of $J^0_{m-2}$ from Lemma \ref{lem:mon-exp-js}, we reach $L_{\gamma_i}(n) \leq J_{m-2}^{0}(\RepLenCode{\mu}, n, 0)$, for any $1 \leq i \leq l$.
By Lemmas~\ref{lem:g-properties}
and~\ref{lem:ub-g-and-l-via-q-h}, $L_{\NormalOrdA}(n)$ is upper bounded by
\[
\sum_{i=1}^{l} H_0(n,L_{\gamma_i}(n))
= \sum_{i=1}^{l} (n+2)^{L_{\gamma_i}(n)}
\leq
\RepLenCode{\NormalOrdA} \cdot 
(n+2)^{J^0_{m-2}(\RepLenCode{\NormalOrdA},n,0)}
\leq J^0_{m-1}(\RepLenCode{\NormalOrdA},n,0).
\]
This completes the proof of the case $k=3$. For $k \geq 4$, we prove the claim by induction on $\RepLenCode{\mu}$. If $\mu=0$, again, there is nothing to prove. Otherwise, as $\NormalOrdA \in \Phi^{m}_{k-3}$, we can write $\mu$ as $ \sum_{i \in I} \Veb{k_i}(\gamma_{i})$, where $I$ is non-empty, $k_i \leq k-3$, for any $i \in I$ and if $k_i=k-3$, then $\gamma_i \in \Phi_{k-3}^{m-1}$ and if $k_i < k-3$, we have $\gamma_i \in \Phi^{m}_{k-3}$. Note that as $I$ is non-empty, $\RepLenCode{\gamma_{i}} \leq \RepLenCode{\mu}-1$, for any $i \in I$. Define $J=\{i \in I \mid k_i=k-3\}$ and $K=\{i \in I \mid k_i < k-3\}$.
Therefore, by Lemmas~\ref{lem:g-properties}, \ref{lem:ub-g-and-l-via-q-h}, and \ref{kMonotonicityOfQ}, we have: 
\[
L_{\mu}(n) \leq \sum_{i \in I}H_{k_i}(n, L_{\gamma_i}(n))  \leq \sum_{i \in J} H_{k-3}(n, L_{\gamma_i}(n))+\sum_{i \in K} H_{k-4}(n, L_{\gamma_i}(n)).
\]
As $\RepLenCode{\gamma_{i}} \leq \RepLenCode{\mu}-1$, by the induction hypothesis and the monotonicity of $J^{k-3}_{m-1}$ and $J^{k-3}_{m-2}$ (Lemma \ref{lem:mon-exp-js}), we have $L_{\gamma_i}(n) \leq J_{m-2}^{k-3}(\RepLenCode{\NormalOrdA}-1, n, 0)$ for any $i \in J$ and $L_{\gamma_i}(n) \leq J_{m-1}^{k-3}(\RepLenCode{\NormalOrdA}-1, n, 0)$ for any $i \in K$. Therefore,
\[
L_{\mu}(n) \leq \sum_{i \in J} H_{k-3}(n, J_{m-2}^{k-3}(\RepLenCode{\NormalOrdA}-1, n, 0))+\sum_{i \in K} H_{k-4}(n, J_{m-1}^{k-3}(\RepLenCode{\NormalOrdA}-1, n, 0)).
\]
As the cardinal of $I$ is bounded by $\RepLenCode{\mu}-1$, we reach 
\[
L_{\mu}(n) \leq  (\RepLenCode{\mu}-1) (H_{k-3}(n, J_{m-2}^{k-3}(\RepLenCode{\NormalOrdA}-1, n, 0))+ H_{k-4}(n, J_{m-1}^{k-3}(\RepLenCode{\NormalOrdA}-1, n, 0))).
\]
Finally, by Lemma~\ref{lem:mon-exp-js}, part $(ii)$, we get $L_{\mu}(n) \leq J_{m-1}^{k-3}(\RepLenCode{\NormalOrdA}, n, 0)$.
\end{proof}

Finally, we have gathered all the necessary ingredients to conclude that both 
$\CodedFunc{R|_{\Phi^{m}_{k-3}}}$ and $\CodedFunc{L|_{\Phi^{m}_{k-3}}}$ belong to $\GrzClassSig{k}$.  
For completeness, we also included the function $\CodedFunc{G|_{\Phi^{m}_{k-3}}}$, as computing $G_\mu(n)$ uniformly in $\mu$ may be of independent interest.

\begin{theorem}\label{lem:space-to-comp-seq}
Let $k \geq 3$ and $m \geq 1$. Then the functions $\CodedFunc{R|_{\Phi^m_k}}$, $\CodedFunc{L|_{\Phi^m_k}}$, and $\CodedFunc{G|_{\Phi^m_k}}$ all belong to $\GrzClassSig{k}$.
\end{theorem}
\begin{proof}
Before we begin, recall from Lemma~\ref{lem:representable-n-geq-3} that one can decide in polynomial time whether a string in $\Sigma^\ast$ encodes an ordinal in $\Phi^m_{k-3}$, and, if so, whether the encoded ordinal is zero, a successor, or a limit. Moreover, recall from the same Lemma that $\CodedFunc{p|_{\Phi^m_{k-3}}} \in \GrzClassSig{3}$, and since $k \geq 3$, we have $\CodedFunc{p|_{\Phi^m_{k-3}}} \in \GrzClassSig{k}$.

First, we show that $\CodedFunc{R^b|_{\Phi^m_{k-3}}} \in \GrzClassSig{k}$.  
Define a function $I(w, z)$ that outputs $w$ if there exist $\mu \in \Phi^m_{k-3}$ and $n \in \mathbb{N}$ such that $w=\IntermCode{\mu}$ and $z=\IntermCode{n}$, and outputs $\bot$ otherwise. By the above discussion, $I$ is computable in polynomial time and hence $I \in \GrzClassSig{3} \subseteq \GrzClassSig{k}$.  
Now, by Theorem~\ref{the:Esigma-closure}, it suffices to observe that $\CodedFunc{R^b|_{\Phi^m_{k-3}}}$ is defined by length-bounded recursion from functions in $\GrzClassSig{k}$ using a bounding function in $\GrzClass{k}$. For this purpose, notice the equality:
\[
\CodedFunc{R^b|_{\Phi^m_{k-3}}}(v,w,z)
=
\begin{cases}
   I(w, z) & \text{if } v = \IntermCode{0},\\
   \CodedFunc{p|_{\Phi^m_{k-3}}}(\CodedFunc{R^b|_{\Phi^m_{k-3}}}(\IntermCode{i}, w,z),z) & \text{if } v = \IntermCode{i+1},\\
   \bot & \text{otherwise,}
\end{cases}
\]
which is the encoded version of the recursive definition of $R^b$ and is therefore straightforward to prove. Since $I$ and $\CodedFunc{p|_{\Phi^m_{k-3}}}$ are both in $\GrzClassSig{k}$, the only remaining task is to check the bound. For this, by Lemma~\ref{lem:bounds-rec-sequence}, there exists a monotone function $M_R \in \GrzClass{3} \subseteq \GrzClass{k}$ such that
\[
\RepLenCode{R^b(i, \mu, n)}
\leq M_R(\RepLenCode{i}, \RepLenCode{\mu}, \RepLenCode{n}),
\]
for any $i, n \in \mathbb{N}$ and $\mu \in \Phi_{k-3}^m$. Since the output of $\CodedFunc{R^b|_{\Phi^m_{k-3}}}$ is either $\bot$ or $\IntermCode{R^b(i, \mu, n)}$, we obtain 
\[
\RepLenCode{\CodedFunc{R^b|_{\Phi^m_{k-3}}}(v, w, z)} \leq 1+M_R(\RepLenCode{v}, \RepLenCode{w}, \RepLenCode{z}),
\]
for any $v, w, z \in \Sigma^*$. Therefore,
since $1+M_R \in \GrzClass{k}$, we get the intended upper bound. 
Hence, by closure under length-bounded primitive recursion (Theorem~\ref{the:Esigma-closure}), we conclude that $\CodedFunc{R^b|_{\Phi^m_{k-3}}} \in \GrzClassSig{k}$.

Now we show that $\CodedFunc{L|_{\Phi^m_{k-3}}} \in \GrzClassSig{k}$. For this purpose, by Theorem~\ref{the:Esigma-closure}, it suffices to observe that $\CodedFunc{L|_{\Phi^m_{k-3}}}$ is essentially defined by bounded $\Sigma$-minimization from a function in $\GrzClassSig{k}$ using a bounding function in $\GrzClass{k}$.  
For the bound, define
\[
M_L(x, y) = J_{m-1}^{k-3}(x, 2^{y}, 0),
\]
and note that $M_L \in \GrzClass{k}$ by Lemma~\ref{JIsInE}. Moreover, by Lemma~\ref{lem:bound-on-g} and the monotonicity of $J_{m-1}^{k-3}$ from Lemma~\ref{lem:mon-exp-js}, we have
\[
L_\mu(n) \leq J_{m-1}^{k-3}(\RepLenCode{\mu}, n, 0) \leq J_{m-1}^{k-3}(\RepLenCode{\mu}, 2^{\RepLenCode{n}}, 0) = M_L(\RepLenCode{\mu}, \RepLenCode{n}),
\]
for any $\mu \in \Phi^m_{k-3}$ and $n \in \mathbb{N}$.  
Next, define
\[
L_R(v, w) \Def \Bmin \, i \leq M_L(\RepLenCode{v}, \RepLenCode{w}) \,[\CodedFunc{R^b|_{\Phi^m_{k-3}}}(\IntermCode{i},v,w) = \IntermCode{0}].
\]
We have already shown that $\CodedFunc{R^b|_{\Phi^m_{k-3}}} \in \GrzClassSig{k}$ and $M_L \in \GrzClass{k}$. Therefore, closure under $\Sigma$-bounded minimization (Theorem~\ref{the:Esigma-closure}) implies that $L_R \in \GrzClassSig{k}$.  
Finally, observe the equality
\[
\CodedFunc{L|_{\Phi^m_{k-3}}}(v, w) =
\begin{cases}
    L_R(\IntermCode{\mu},\IntermCode{n}) & \exists \mu \in \Phi^m_{k-3} \,\exists n \in \mathbb{N} \, (v = \IntermCode{\mu} \wedge w = \IntermCode{n}), \\
    \bot & \text{otherwise}
\end{cases}
\]
which is straightforward to verify. By the discussion at the beginning of the proof, deciding between the cases in the previous equality can be done in polynomial time. Therefore, we can conclude that $\CodedFunc{L|_{\Phi^m_{k-3}}} \in \GrzClassSig{k}$.

For $\CodedFunc{R|_{\Phi^m_{k-3}}} \in \GrzClassSig{k}$, we can easily prove the claim using the definition $R(i, \mu, n) = R^b(L_\mu(n) \dotminus i, \mu, n)$, the established results $\CodedFunc{R^b|_{\Phi^m_{k-3}}} \in \GrzClassSig{k}$ and $\CodedFunc{L|_{\Phi^m_{k-3}}} \in \GrzClassSig{k}$, and the facts that deciding whether a string is a code of a number is possible in polynomial time and that $\dotminus$ belongs to $\GrzClass{3} \subseteq \GrzClass{k}$.

Finally, we leverage the above functions to show that $\CodedFunc{G|_{\Phi^m_{k-3}}} \in \GrzClassSig{k}$. The idea is to count the number of successor ordinals in the sequence of $n$-predecessors of the input ordinal as enumerated by $\CodedFunc{R|_{\Phi^m_{k-3}}}$. For this purpose, define a function $F(x,u,v,w)$ as follows: it returns $\bot$ if it is not the case that 
$x = \IntermCode{y}$, $u = \IntermCode{i}$, $v = \IntermCode{\mu}$, and $w = \IntermCode{n}$ for some $y, i, n \in \mathbb{N}$ and $\mu \in \Phi_{k-3}^m$. Otherwise, set
\[
F(\IntermCode{y}, \IntermCode{i}, \IntermCode{\mu}, \IntermCode{n}) =
\begin{cases}
\IntermCode{y+1} & \text{if } \CodedFunc{R|_{\Phi^m_{k-3}}}(\IntermCode{i}, \IntermCode{\mu}, \IntermCode{n}) \text{ is a successor},\\
\IntermCode{y} & \text{otherwise}.
\end{cases}
\]  
Since $\CodedFunc{R|_{\Phi^m_{k-3}}} \in \GrzClassSig{k}$, and since both the test for whether an ordinal is a successor and the computation of the numerical successor function are polynomial-time operations, it follows that $F \in \GrzClassSig{k}$.
Next, define
\[
g(u,v,w) \Def 
\begin{cases}
\IntermCode{0} & \text{if } u = \IntermCode{0},\\
F(g(\IntermCode{i},v,w),u,v,w) & \text{if } u = \IntermCode{i+1},\\
\bot & \text{otherwise}.
\end{cases}
\]  
It is easy to see that $g(\IntermCode{i}, \IntermCode{\mu}, \IntermCode{n})$ counts the number of successor ordinals in the sequence $\{R(j, \mu, n)\}_{j=0}^i$. Therefore, 
$
g(\IntermCode{i}, \IntermCode{\mu}, \IntermCode{n}) \leq L_\mu(n),
$ 
for any $\mu \in \Phi_{k-3}^m$ and any $n, i \in \mathbb{N}$. Since the output of $g$ is either $\bot$ or a number bounded by $L_{\mu}(n) \leq M_L(\RepLenCode{\mu}, \RepLenCode{n})$, we obtain 
\[
\RepLenCode{g(u,v,w)} \leq 1 + \RepLenCode{M_L(\RepLenCode{v}, \RepLenCode{w})} \leq 2+M_L(\RepLenCode{v}, \RepLenCode{w}).
\]  
Since $M_L \in \GrzClass{k}$, we have $2 + M_L \in \GrzClass{k}$. Therefore, $g$ is defined by length-bounded primitive recursion and thus $g \in \GrzClassSig{k}$ by Theorem~\ref{the:Esigma-closure}.  
Finally, observe that
\[
\CodedFunc{G|_{\Phi^m_{k-3}}}(v,w) = g(\CodedFunc{L|_{\Phi^m_{k-3}}}(v,w),v,w).
\]  
Therefore, since $g$ and $\CodedFunc{L|_{\Phi^m_{k-3}}}$ are in $\GrzClassSig{k}$, we conclude that $\CodedFunc{G|_{\Phi^m_{k-3}}} \in \GrzClassSig{k}$.
\end{proof}

\subsubsection{The case $\Psi_k^m$, for $k,m \geq 1$}
 
In this subsubsection, we aim to show that $\CodedFunc{R|_{\Psi^m_k}}_E$ and $\CodedFunc{L|_{\Psi^m_k}}_E$
are computable in linear space and hence belong to $\GrzClassSig{2}$. As before, the main strategy is to establish upper bounds for $\RepLenCode{R(i, \mu, n)}_E$ and $L_{\mu}(n)$. 
Recall from Definition~\ref{def:sum-norm-tuples}
that $\SumNorm{0}=0$ and $\SumNorm{\alpha} = \sum_{i=1}^{l} c_i$, where $\alpha=\sum_{i=1}^l\omega^{p_i}\cdot c_{i}$
is the normal form of the non-zero ordinal $\alpha \in \Psi_{\omega}$.
To obtain the desired upper bounds, we first express them in terms of $|\mu|$, and then, by applying Lemma~\ref{lem:relate-norms-linear}, we reformulate them in terms of $\RepLenCode{\mu}_E$.

\begin{lemma}\label{the:computationOfGForOmegak}
Let $k,m \in {\mathbb{N}^{ \geq 1}}$ be natural numbers. Then:
\begin{itemize}
    \item[$(i)$] 
There is a polynomial $P_L$ with natural coefficients such that $L_\mu(n) \leq P_L(\SumNorm{\mu}, n)$, for any $\mu \in \Psi_k^m$ and any $n \in \mathbb{N}$. 
    \item[$(ii)$] 
There is a polynomial $P_R$ with natural coefficients such that $\SumNorm{R^b(i, \mu, n)}
        \allowbreak\leq P_R(i, \SumNorm{\NormalOrdA}, n)$,
    for any $\mu \in \Psi_k^m$ and any $n, i \in \mathbb{N}$. 
\end{itemize}
\end{lemma}
\begin{proof}
For $(i)$, define 
$P_L(x, n)=x \sum_{i=1}^{m+k} (n+2)^i$. We need to show that $L_\mu(n) \leq P_L(\SumNorm{\mu}, n)$. If $\mu=0$, as $L_{\mu}(n)=0$, there is nothing to prove. For non-zero $\mu \in \Psi_{k}^m$, let $\mu=\sum_{i=1}^{l}\omega^{p_i} \cdot c_i$ be its normal form. Note that for each $1 \leq i \leq l$, by definition  $c_i > 0$ and as $\SumNorm{\mu}=\sum_{i=1}^{l}c_i$, we have $c_i \leq \SumNorm{\mu}$.
By Remark \ref{LengthBoundForPsi}, $l \leq m+k$. Therefore, using Lemma \ref{lem:g-properties} and the fact that $\omega$ and $c_i$'s are non-zero, for any $n \in \mathbb{N}$, we reach: 
\[
L_{\mu}(n) \leq 
\sum_{i=1}^{l} c_i (n+2)^{i} 
\leq \SumNorm{\mu} \sum_{i=1}^{l} (n+2)^i
\leq \SumNorm{\mu} \sum_{i=1}^{m+k} (n+2)^i
=P_L(\SumNorm{\mu}, n).
\]
For $(ii)$, by Lemma~\ref{lem:omega-k-case}, we have $\SumNorm{p(\alpha, n)} \leq \SumNorm{\alpha}+n$, for any $\alpha \in \Psi^m_k$ and any $n \in \mathbb{N}$.
Define $P_R(i,x,n) = x+in$.
By induction on $i$, we prove $\SumNorm{R^b(i, \mu, n)} \leq P_R(i, \SumNorm{\NormalOrdA}, n)$. If $i=0$, we have $R^b(0, \mu, n)= \NormalOrdA$,
and the claim is clear. For the inductive case, assume the claim for $i$. Then, as $R^b(i+1, \mu, n)=p(R^b(i, \mu, n), n)$, we have
\[
\SumNorm{R^b(i+1, \mu, n)}=\SumNorm{p(R^b(i, \mu, n), n)} 
\leq \SumNorm{\mu}+in+n=\SumNorm{\mu}+(i+1)n
\]
which completes the proof. 
\end{proof}

\begin{lemma}\label{lem:poly-to-linear}
    For any polynomial $P(x_1,\ldots,x_l)$ with natural coefficients, there is a linear function $L$ with natural coefficients such that for any $\bar{x} \in \mathbb{N}$,
    $
    \log(P(x_1,\ldots,x_l)+1)
    \leq
    L(\log(x_1+1),\ldots,\log(x_l+1))
    $.
\end{lemma}
\begin{proof}
If $P(\bar{x})$ is the zero polynomial, let $L=0$, and there is nothing to prove. Otherwise, let $P(x_1,\ldots,x_l) \Def \sum_{i = 1}^s \prod_{j=1}^{l} x_{j}^{e_{ij}}$, with $s \geq 1$, where each $e_{ij} \geq 0$. Consider the following inequalities. First, it is clear that $x^e+1 \leq (x+2)^{e+1}$ for any $x, e \in \mathbb{N}$. Second, since $a+b \leq ab$ for any $a, b \in \mathbb{R}^{\geq 2}$, we obtain $\sum_{i=1}^s a_i \leq \prod_{i=1}^s a_i$ for any $\{a_0, \ldots, a_s\} \subseteq \mathbb{R}^{\geq 2}$. Third, as $(x+2) \leq 2(x+1)$, we have $\log(x+2) \leq \log(x+1)+1$. Using these inequalities, we obtain:
\[
   \log(P(x_1,\ldots,x_l)+1) \leq  \log\Big(\sum_{i = 1}^s \prod_{j=1}^{l} (x_{j}^{e_{ij}} + 1)\Big)
    \leq 
    \log\Big(\sum_{i = 1}^s \prod_{j=1}^{l} (x_{j}+2) ^{e_{ij}+1}\Big)
\]
\[
    \leq 
    \log\Big(\prod_{i = 1}^s \prod_{j=1}^{l} (x_{j}+2) ^{e_{ij}+1}\Big)
    =
    \sum_{i = 1}^s \sum_{j=1}^{l} {(e_{ij}+1)} \log(x_{j}+2)
\]
\[
    \leq 
    \sum_{i = 1}^s \sum_{j=1}^{l} {(e_{ij}+1)} (\log(x_{j}+1)+1)
    \leq 
    a \cdot\sum_{j=1}^l 
     \log(x_j+1)
    + b
\]
for sufficiently large constants $a,b \in \NatSet$.
\end{proof}

\begin{lemma}\label{lem:mon-poly-e2-rq}
Let $k,m \in {\mathbb{N}^{ \geq 1}}$ be natural numbers. Then, there are linear functions $M_R$ and $M_L$ with natural coefficients such that $\RepLen{R^b(i, \mu, n)}_E
        \leq M_R(\RepLen{i}_E,\RepLen{\NormalOrdA}_E, 
    \RepLen{n}_E)$ and $\RepLen{L_{\mu}(n)}_E
        \leq M_L(\RepLen{\NormalOrdA}_E, 
    \RepLen{n}_E)$,
    for any $\mu \in \Psi_k^m$ and any $n, i \in \mathbb{N}$.    
\end{lemma}
\begin{proof}
We only cover the case of $M_R$; the other is similar. By Lemma \ref{the:computationOfGForOmegak}, there is a polynomial $P_R$ with natural coefficients such that
$\SumNorm{R^b(i, \mu, n)} \leq P_R(i, \SumNorm{\NormalOrdA}, n)$,  for any 
$\mu \in \Psi_k^m$ and any $n, i \in \mathbb{N}$.
As $P_R$ is a polynomial with natural coefficients, by Lemma~\ref{lem:poly-to-linear}, there is a linear function $L$ with natural coefficients such that:
    \[
    \log(\SumNorm{R(i, \mu, n)}+1) \leq \log(P_R(i, 
    \SumNorm{\NormalOrdA}, 
    n)+1)
    \leq L(\log(i+1), \log(\SumNorm{\NormalOrdA}+1),\log(n+1)).
    \]
By Lemma~\ref{lem:relate-norms-linear}, there are positive natural numbers $a, b \in \mathbb{N}^{\geq 1}$ such that $\log(\SumNorm{\alpha}+1)\leq  \RepLen{\alpha}_E$ and $\RepLen{\alpha}_E \leq a \cdot \log(\SumNorm{\alpha}+1)+b$, for any $\alpha \in \Psi^m_k$. Define $M_R(x, y, z)=aL(x, y, z)+b$. Since $\RepLen{x}_E-1 \leq \log(x+1) \leq \RepLen{x}_E$ for any $x \in \mathbb{N}$, we reach $\RepLen{R^b(i, \mu, n)}_E
        \leq M_R(\RepLenCode{i}_E, \RepLen{\NormalOrdA}_E, \RepLen{n}_E)$,
    for any $\mu \in \Psi_k^m$ and any $n, i \in \mathbb{N}$. 
\end{proof}

\begin{theorem}\label{lem:space-to-comp-seq-linear}
Let $k, m \geq 1$ be natural numbers. Then, $\CodedFunc{R|_{\Psi^m_k}}_E$, $\CodedFunc{L|_{\Psi^m_k}}_E$ and $\CodedFunc{G|_{\Psi^m_k}}_E$ 
 are all in $\GrzClassSig{2}$.
\end{theorem}
\begin{proof}
The proof is similar to that of Theorem~\ref{lem:space-to-comp-seq}, replacing $\GrzClassSig{k}$ with $\GrzClassSig{2}$, i.e., linear-space functions. The main points are as follows. First, by Lemma~\ref{lem:representable-n-geq-3III}, checking whether a string is a code of an ordinal in $\Psi^m_k$ and, if so, whether it is zero, a successor, or a limit, can be done in linear space. Moreover, by the same lemma, the function $\CodedFunc{p|_{\Psi^m_k}}$ is computable in linear space. Second, by Lemma~\ref{lem:mon-poly-e2-rq}, there exist linear functions $M_R$ and $M_L$ with natural coefficients such that 
\[
\RepLen{R^b(i, \mu, n)}_E \leq M_R(\RepLenCode{i}_E, \RepLenCode{\mu}_E, \RepLenCode{n}_E)
\quad \text{and} \quad
L_{\mu}(n) \leq 2^{M_L(\RepLenCode{\mu}_E, \RepLenCode{n}_E)}
\]
for any $\mu \in \Psi^m_k$ and any $n, i \in \mathbb{N}$. These bounds are precisely what we need (see Theorem~\ref{the:Esigma-closure}) to apply the closure of $\GrzClassSig{2}$ under length-bounded recursion and bounded $\Sigma$-minimization.
\end{proof}

\subsection{Bounding lemmas}\label{subsec:bounding-lemma}

In this subsection, we complete the third part of our strategy, as outlined in Section~\ref{sec:main-theorem}, by establishing an upper bound on the length of the outputs of the functions in $\mathcal{C}_{\mathsf{A}}$ for $\mathsf{A} = \Phi^m_{k-3}$ (with $k \geq 3$ and $m \geq 1$) and $\mathsf{A} = \Psi_k^m$ (with $k,m \geq 1$), using functions in $\GrzClass{k}$ and linear functions, respectively.

\begin{lemma}
\label{lem:poly-norm-codes-loglog}
Let $k,m \in \NatSet$ with $k \geq 3$ and $m \geq 1$. Then, for any $f \in \ClassName{\Phi^{m}_{k-3}}$, there is a monotone function $B_f \in \mathcal{E}_k$ such that 
\[
f(\bar\NormalOrdA,\bar\SafeOrdA;\bar\SafeNumbA) \leq
    B_f(
    \RepLenCode{\bar \NormalOrdA},
    \bar \SafeOrdA
    )
    +
    \max_i \SafeNumbA_i,
\]
for any $\bar\NormalOrdA \in \Phi^{m}_{k-3}$ and any $\bar n, \bar a \in \NatSet$.
    Consequently,
    there is a monotone function $E_f \in \mathcal{E}_k$ such that
$\RepLenCode{f(\bar\NormalOrdA,\bar\SafeOrdA;\bar \SafeNumbA)}
    \leq E_f(
    \RepLenCode{\bar \NormalOrdA},
    \RepLenCode{\bar \SafeOrdA},
    \RepLen{\bar \SafeNumbA}
    )
$, for any $\bar\NormalOrdA \in \Phi_{k-3}^{m}$ and any $\bar{n}, \bar{a} \in \NatSet$.
\end{lemma}
\begin{proof}
For the first part of the claim, we proceed by induction on the construction of $f$ as an element of $\ClassName{\Phi^{m}_{k-3}}$.  
For the base case, we define $B_f$ as follows:
    \begin{itemize}
        \item (constant) For the nullary zero function, take $B_f$ as the constant $0$. 
        \item (projections) For any of the projections, take $B_f(\bar x) := \sum_{i} x_i$.
        \item (successor) 
        Take $B_f$ as the constant $1$. 
        \item (predecessor)
        Take $B_f$ as the constant $0$. It works 
        as $\Pred(;a) \leq a$.
        \item (conditional) 
            Take $B_f$ as the constant $0$.
     \end{itemize}
For the safe composition, if $f$ is constructed by:
    $$
    f(\bar \NormalOrdA, \bar \SafeOrdA; \bar \SafeNumbA)
		= h(\bar{\NormalOrdA},\bar s(\bar \NormalOrdA, \bar \SafeOrdA;);\bar t(\bar \NormalOrdA,\bar\SafeOrdA; \bar \SafeNumbA)),
    $$
first, consider the functions $B_h$, $\bar{B}_s$, and $\bar{B}_t$ obtained by using the induction hypothesis on $h$, $\bar{s}$, and $\bar{t}$, respectively. Then, define 
$B(\bar x, \bar y) := B_h(\bar x, \bar{B}_s(\bar x,\bar{y}))$ and set $B_f(\bar x, \bar y):= B(\bar x, \bar y)+\sum_i B_{t_i}(\bar{x},\bar{y})$. Clearly $B_f$ is monotone and in $\mathcal{E}_k$. Moreover, using the induction hypothesis and the monotonicity of $B_h$, we have:
\begin{align*}
h(\bar{\NormalOrdA},\bar s(\bar \NormalOrdA, \bar \SafeOrdA;);\bar t(\bar \NormalOrdA,\bar\SafeOrdA; \bar \SafeNumbA)) 
&\leq B_h\big(\RepLenCode{\bar{\NormalOrdA}},
\bar{s}(\bar{\NormalOrdA}, \bar{n};)) + \max_i t_i(\bar{\NormalOrdA}, \bar{\SafeOrdA}; \bar{\SafeNumbA}) \\
&\leq B_h\big(\RepLenCode{\bar{\NormalOrdA}},
\bar{B}_s(\RepLenCode{\bar{\NormalOrdA}}, 
\bar{n})\big) 
+ \max_i t_i(\bar{\NormalOrdA}, \bar{\SafeOrdA};\bar{\SafeNumbA})\\
&= B\big(\RepLenCode{\bar{\NormalOrdA}}, 
\bar{n}\big) 
+ \max_i t_i(\bar{\NormalOrdA}, \bar{\SafeOrdA}; \bar{\SafeNumbA})\\
&\leq B\big(\RepLenCode{\bar{\NormalOrdA}}, 
\bar{n}\big) 
+ \max_i\big\{B_{t_i}(\RepLenCode{\bar{\NormalOrdA}}, 
\bar{n}) + \max_j a_j\big\} \\
&\leq B_f\big(\RepLenCode{\bar{\NormalOrdA}}, 
\bar{n}\big) + \max_j a_j.
\end{align*}
Therefore, $f(\bar\NormalOrdA,\bar\SafeOrdA;\bar\SafeNumbA) \leq
    B_f(\RepLenCode{\bar \NormalOrdA},
    \bar \SafeOrdA)
    +
    \max_j \SafeNumbA_j$.

For predicative ordinal recursion, assume that $f$ is defined by
\begin{align*}
f(0, \bar \NormalOrdB, \bar \SafeOrdA; \bar \SafeNumbA)
  &= g(\bar \NormalOrdB, \bar \SafeOrdA; \bar \SafeNumbA), \\
f(\NormalOrdA+1, \bar \NormalOrdB, \bar \SafeOrdA; \bar \SafeNumbA) 
  &= h_{\mathsf{suc}}\!\big(\NormalOrdA, \bar \NormalOrdB, \bar \SafeOrdA; 
     f(\NormalOrdA, \bar \NormalOrdB, \bar \SafeOrdA; \bar \SafeNumbA), \bar \SafeNumbA\big), \\
f(\langle \NormalOrdA_i \rangle_i, \bar \NormalOrdB, \bar \SafeOrdA; \bar \SafeNumbA)
  &= h_{\mathsf{lim}}\!\big(\langle \NormalOrdA_i \rangle_i, \bar \NormalOrdB, \bar \SafeOrdA;
     f(\NormalOrdA_{q(\bar \SafeOrdA;)}, \bar \NormalOrdB, \bar \SafeOrdA; \bar \SafeNumbA),
     \bar \SafeNumbA\big).
\end{align*}
By the induction hypothesis, there exists $B_q \in \mathcal{E}_k$ such that 
$q(\bar n;) \leq B_q(\bar n)$. By Lemma~\ref{lem:bounds-rec-sequence}, there are monotone 
functions $M_L, M_R \in \mathcal{E}_k$ such that
\[
\RepLenCode{R^b(i, \mu, n)} \leq M_R(\RepLenCode{i}, \RepLenCode{\mu}, \RepLenCode{n})
\quad\text{and}\quad
L_\mu(n) \leq M_L(\RepLenCode{\mu}, \RepLenCode{n}),
\]
for all $\mu \in \Phi_{k-3}^m$ and $n,i \in \mathbb{N}$. Define
\[
M_{L,q}(x, \bar z) := M_L(x, 1+B_q(\bar z)), 
\quad
M_{R,q}(x, \bar z) := M_R(1+M_{L,q}(x,\bar z), x, 1+B_q(\bar z)).
\]
By monotonicity of $M_L$ and the fact that $\RepLenCode{b} \leq 1+b$ for all $b \in \mathbb{N}$,
we obtain
\[
L_\mu(q(\bar n)) \leq M_{L,q}(\RepLenCode{\mu}, \bar n),
\quad
\RepLenCode{R^b(i, \mu, q(\bar n))} \leq M_{R,q}(\RepLenCode{\mu}, \bar n).
\]
Since this bound is independent of $i$, and since $R(i, \mu, q(\bar n))$ traverses the same sequence 
of $q(\bar n)$-predecessors of $\mu$, the same bound applies to $R(i,\mu,q(\bar n))$.
Define
\[
C(v,\bar x,\bar z) :=
  B_{h_{\mathsf{suc}}}(M_{R,q}(v,\bar z), \bar x, \bar z) +
  B_{h_{\mathsf{lim}}}(M_{R,q}(v,\bar z), \bar x, \bar z),
\]
and set
\[
B_f(v,\bar x,\bar z) :=
  B_g(\bar x,\bar z) + M_{L,q}(v,\bar z)\cdot C(v,\bar x,\bar z).
\]
It is immediate that $B_f$ is monotone and $B_f \in \mathcal{E}_k$.
For simplicity, write $\mu(i) := R(i,\mu,q(\bar n))$. By the induction hypothesis,
\begin{align*}
f(\mu(0), \bar \NormalOrdB, \bar \SafeOrdA; \bar \SafeNumbA)
  &= g(\bar \NormalOrdB, \bar \SafeOrdA; \bar \SafeNumbA)
   \leq B_g(\RepLenCode{\bar \NormalOrdB}, \bar n) + \max_l a_l, \\
f(\NormalOrdA(i+1), \bar \NormalOrdB, \bar \SafeOrdA; \bar \SafeNumbA)
  &\leq B_H(\RepLenCode{\mu'_i}, \RepLenCode{\bar \NormalOrdB}, \bar n) 
   + \max\bigl\{ f(\NormalOrdA(i), \bar \NormalOrdB, \bar \SafeOrdA; \bar \SafeNumbA), \max_l a_l \bigr\},
\end{align*}
where
\[
(\mu'_i,H) =
\begin{cases}
(\mu(i), h_{\mathsf{suc}}), & \text{if $\NormalOrdA(i+1)$ is a successor}, \\
(\mu(i+1), h_{\mathsf{lim}}), & \text{if $\NormalOrdA(i+1)$ is a limit}.
\end{cases}
\]
Since $\mu(j) = R(j,\mu,q(\bar n))$, we have 
$\RepLenCode{\mu(j)} \leq M_{R,q}(\RepLenCode{\mu},\bar n)$ for all $j$, 
and thus
$B_H(\RepLenCode{\mu'_i}, \RepLenCode{\bar \NormalOrdB}, \bar n) 
  \leq C(\RepLenCode{\mu}, \RepLenCode{\bar \NormalOrdB}, \bar n)$.
Therefore,
\begin{align*}
f(\mu(0), \bar \NormalOrdB, \bar \SafeOrdA; \bar \SafeNumbA)
  &\leq B_g(\RepLenCode{\bar \NormalOrdB}, \bar n) + \max_l a_l, \\
f(\NormalOrdA(i+1), \bar \NormalOrdB, \bar \SafeOrdA; \bar \SafeNumbA)
  &\leq C(\RepLenCode{\mu}, \RepLenCode{\bar \NormalOrdB}, \bar n)
       + \max\{ f(\NormalOrdA(i), \bar \NormalOrdB, \bar \SafeOrdA; \bar \SafeNumbA), \max_l a_l\}.
\end{align*}
By induction on $i$, it follows that
\[
f(\NormalOrdA(i), \bar \NormalOrdB, \bar \SafeOrdA; \bar \SafeNumbA)
  \leq B_g(\RepLenCode{\bar \NormalOrdB}, \bar n) 
       + i \, C(\RepLenCode{\mu}, \RepLenCode{\bar \NormalOrdB}, \bar n)
       + \max_l a_l. \qquad (*)
\]
Recall that $L_\alpha(m)$ is the least $j \in \mathbb{N}$ such that $R^b(j,\alpha,m)=0$. 
Hence $R(L_{\alpha}(m), \alpha, m)=R^b(0, \alpha, m)=\alpha$ which implies $R(L_\mu(q(\bar n)), \mu, q(\bar n)) = \mu$.
Since $\mu(L_\mu(q(\bar n)))=
R(L_\mu(q(\bar n)), \mu, q(\bar n))$ by definition, substituting $i = L_\mu(q(\bar n))$ into $(*)$ and using 
$L_\mu(q(\bar n)) \leq M_{L,q}(\RepLenCode{\mu}, \bar n)$, we conclude
\[
f(\NormalOrdA, \bar \NormalOrdB, \bar \SafeOrdA; \bar \SafeNumbA)
  \leq B_g(\RepLenCode{\bar \NormalOrdB}, \bar n)
       + M_{L,q}(\RepLenCode{\mu}, \bar n)\, C(\RepLenCode{\mu}, \RepLenCode{\bar \NormalOrdB}, \bar n)
       + \max_l a_l,
\]
which implies
$f(\NormalOrdA, \bar \NormalOrdB, \bar \SafeOrdA; \bar \SafeNumbA)
  \leq B_f(\RepLenCode{\mu}, \RepLenCode{\bar \NormalOrdB}, \bar n) + \max_l a_l$.

If $f$ is obtained by constant substitution, i.e.,
\[
f(\NormalOrdA_1,\ldots,\NormalOrdA_{i-1},\NormalOrdA_{i+1},\ldots,\NormalOrdA_{k}, \bar \SafeOrdA; \bar \SafeNumbA)
=
g(\NormalOrdA_1,\ldots,\NormalOrdA_{i-1},\alpha,\NormalOrdA_{i+1},\ldots,\NormalOrdA_{k},\bar\SafeOrdA;\bar\SafeNumbA),
\]
then, by the induction hypothesis, there exists monotone
$B_g \in \mathcal{E}_k$
such that
\[
g(\bar\NormalOrdA,\bar\SafeOrdA;\bar\SafeNumbA)
\leq 
B_g(\RepLen{\bar{\NormalOrdA}},\bar{n}) + \max_l a_l.
\]
It is then sufficient to define
\[
B_f(x_1,\ldots,x_{i-1},x_{i+1},\ldots,x_k,\bar z)
\Def 
B_g(x_1,\ldots,x_{i-1}, \RepLen{\alpha}, x_{i+1},\ldots,x_k,\bar z),
\]
which is clearly monotone and belongs to $\mathcal{E}_k$.

The case of structural rules follows similarly, namely by modifying in the obvious way 
the function obtained via the induction hypothesis. This completes the proof of the first part.

For the second part, using the first part, the monotonicity of $B_f$, and the facts that $\RepLenCode{a} \leq 1+a$ and $m \leq 2^{\RepLenCode{m}}$ for every $a, m \in \mathbb{N}$, we reach
\[
\RepLenCode{f(\bar\NormalOrdA,\bar\SafeOrdA;\bar \SafeNumbA)}
    \leq \RepLenCode{B_f(
    \RepLenCode{\bar \NormalOrdA},
    \bar \SafeOrdA
    )
    +
    \max_i \SafeNumbA_i}
    \leq
    1+
    B_f(
    \RepLenCode{\bar \NormalOrdA},
    \bar \SafeOrdA
    )+\max_i \SafeNumbA_i
    \]
    \[
    \leq 1+B_f(
    \RepLenCode{\bar \NormalOrdA},
    2^{\RepLenCode{\bar\SafeOrdA}}
    )
    +\sum_l 2^{\RepLenCode{a_l}}.
    \]
It suffices now to observe that $
E_f(\bar x, \bar y, \bar z) \Def 1 + B_f(\bar x, \bar 2^{\bar y}) + \sum_l 2^{z_l}
$ is monotone, belongs to $\GrzClass{k}$ as $k \geq 3$, and satisfies the required property.
\end{proof}

For $\mathsf{A} = \Psi_k^m$, where $k,m \geq 1$, we first bound the value of the function by a polynomial in terms of the numeral inputs and the norms of the ordinal inputs, rather than the lengths of their encodings. We then apply Lemma~\ref{lem:relate-norms-linear} and Lemma~\ref{lem:poly-to-linear} to connect $\SumNorm{\mu}$ with $\RepLenCode{\mu}_E$ via a logarithmic relationship, which in turn yields a linear bound on the length of the functions.

\begin{lemma}\label{lem:poly-bound-value-omega-k}
Let $k,m \in \mathbb{N}^{\geq 1}$. For any 
$f \in \ClassName{\Psi_k^m}$, 
there is 
a polynomial $P_f$ with natural coefficients such that 
$f(\bar\NormalOrdA,\bar\SafeOrdA;\bar \SafeNumbA)
	\leq P_f(
\SumNorm{\bar\NormalOrdA},
\bar{n}
	)
	+
	\max_i {a_i}$,
for any $\bar\NormalOrdA \in \Psi_k^m$ and any $\bar{n},\bar\SafeNumbA \in \NatSet$.
\end{lemma}
\begin{proof}
The argument is analogous to that of Lemma~\ref{lem:poly-norm-codes-loglog}. 
The only difference is that, in this case, we use Lemma~\ref{the:computationOfGForOmegak} to polynomially bound 
$\SumNorm{R(i, \mu, q(\bar n))}$ and $L_{\mu}(q(\bar n))$ in terms of $\SumNorm{\mu}$ and $\bar n$.
\end{proof}

\begin{corollary}\label{cor:len-code-linear-bound}
Let $k,m \in \mathbb{N}_{\geq 1}$. Then, for every 
$f \in \ClassName{\Psi_k^m}$, there exists a linear function 
$L_f$ with natural coefficients such that
\[
    \RepLen{f(\bar\NormalOrdA,\bar\SafeOrdA;\bar\SafeNumbA)}_E
    \;\leq\;
    L_f\bigl(\RepLen{\bar\NormalOrdA}_E,\,
    \RepLen{\bar\SafeOrdA}_E,\,
    \RepLen{\bar\SafeNumbA}_E\bigr),
\]
for all $\bar\NormalOrdA \in \Psi_k^m$ and all 
$\bar n, \bar\SafeNumbA \in \NatSet$.
\end{corollary}
\begin{proof}
By Lemma~\ref{lem:poly-bound-value-omega-k}, there is a polynomial function 
$P_f$ with natural coefficients such that
$f(\bar\NormalOrdA,\bar\SafeOrdA;\bar \SafeNumbA)
\leq P_f(\SumNorm{\bar\NormalOrdA}, \bar\SafeOrdA, \bar{\SafeNumbA})$,
for any $\bar\NormalOrdA \in \Psi_k^m$ and any $\bar\SafeOrdA, \bar\SafeNumbA \in \NatSet$. Therefore, by Lemma~\ref{lem:poly-to-linear},
there is a linear map $L$ with natural coefficients
such that
\[
\log\bigl(1 + f(\bar\NormalOrdA,\bar\SafeOrdA;\bar \SafeNumbA)\bigr) 
\leq 
\log \bigl(1 + P_f(\SumNorm{\bar\NormalOrdA}, \bar\SafeOrdA, \bar{\SafeNumbA})\bigr)
\leq L(\log(\SumNorm{\bar \NormalOrdA}+1), \log(\bar \SafeOrdA+1), \log(\bar{\SafeNumbA}+1)).
\]
By Lemma~\ref{lem:relate-norms-linear}, 
$\log(\SumNorm{\alpha}+1) \leq \RepLen{\alpha}_E$, for any $\alpha \in \Psi^m_k$. Moreover, note that $\RepLen{x}_E-1 \leq \log(x+1) \leq \RepLen{x}_E$ for any $x \in \mathbb{N}$. Define $L_f(x, y, z) \Def L(x, y, z)+1$. As $L_f$ has natural coefficients, both $L$ and, consequently, $L_f$ are monotone. Therefore:
\[
\RepLen{f(\bar\NormalOrdA,\bar\SafeOrdA;\bar \SafeNumbA)}_E 
\leq 
\log\bigl(1 + f(\bar\NormalOrdA,\bar\SafeOrdA;\bar \SafeNumbA)\bigr) + 1
\leq L_f(\RepLen{\bar\NormalOrdA}_E, \RepLen{\bar\SafeOrdA}_E, \RepLen{\bar{\SafeNumbA}}_E),
\]
which completes the proof.
\end{proof}

\subsection{Proof of $\PredFuncClass{\mathsf{A}} \subseteq \mathcal{E}_{l(\mathsf{A})+2}$}\label{subsecSimulation}

In this subsection, we employ all the ingredients developed so far to show that $\PredFuncClass{\Phi_{k-3}^{m}} \subseteq \mathcal{E}_{k}$, for $k \geq 3$ and $m \geq 1$, and $\PredFuncClass{\Psi_k^m} \subseteq \mathcal{E}_{2}$, for $k,m \geq 1$, and then conclude that $\PredFuncClass{\mathsf{A}} \subseteq \mathcal{E}_{l(\mathsf{A})+2}$, for any bounded $\mathsf{A} \subseteq \Phi_{\omega}$.

\begin{lemma}\label{lem:coded-func-in-grz-space}\
   \begin{itemize}
    \item[$(i)$] 
    Let $k, m \in \mathbb{N}$ with $k \geq 3$ and $m \geq 1$. Then,
    $\CodedFunc{f} \in \GrzClassSig{k}$ for any $f \in \ClassName{\Phi_{k-3}^{m}}$. Therefore, $\PredFuncClass{\Phi_{k-3}^{m}} \subseteq \mathcal{E}_{k}$.
    \item[$(ii)$] 
    Let $k, m \in \mathbb{N}^{\geq 1}$. Then, 
    $\CodedFunc{f}_E \in \GrzClassSig{2}$
    for any 
    $f \in \ClassName{\Psi_k^m}$. Therefore, 
    $\PredFuncClass{\Psi_k^m} \subseteq \mathcal{E}_{2}$.
\end{itemize}
\end{lemma}
\begin{proof}
For $(i)$, we proceed by induction on the construction of $f$ as an element of $\ClassName{\Phi_{k-3}^{m}}$. The claim for the basic functions is immediate. 
For the inductive step, if $f$ is constructed by safe composition, the claim follows directly from the induction hypothesis and the fact that $\GrzClassSig{k}$ is closed under composition. 
If $f$ is defined by the following predicative ordinal recursion:
\begin{align*}
			f(0, \bar \NormalOrdB, \bar \SafeOrdA; \bar \SafeNumbA)
			&= g(\bar \NormalOrdB,
			\bar \SafeOrdA; \bar \SafeNumbA)\\
			f(\NormalOrdA+1,\bar \NormalOrdB, \bar \SafeOrdA; \bar \SafeNumbA) &= h_{\mathsf{suc}}(\NormalOrdA, \bar \NormalOrdB,\bar \SafeOrdA; 
			f(\eta, \bar \NormalOrdB, \bar \SafeOrdA; \bar \SafeNumbA), \bar \SafeNumbA)\\
			f(\langle \NormalOrdA_i \rangle_i, \bar \NormalOrdB,\bar \SafeOrdA;\bar \SafeNumbA)
			&= h_{\mathsf{lim}}(\langle \NormalOrdA_i \rangle_i, \bar \NormalOrdB, \bar \SafeOrdA;
			f(\NormalOrdA_{q(\bar \SafeOrdA;)}, \bar \NormalOrdB, \bar \SafeOrdA; \bar \SafeNumbA),
			\bar \SafeNumbA
			),
		\end{align*}
by the induction hypothesis, the functions $\CodedFunc{q}$, $\CodedFunc{g}$, $\CodedFunc{h_{\mathsf{suc}}}$, and $\CodedFunc{h_{\mathsf{lim}}}$ are all in $\GrzClassSig{k}$.  
To prove that $\CodedFunc{f} \in \GrzClassSig{k}$, we begin with an informal explanation of our strategy, ignoring the encoding of ordinals and numbers. For simplicity, denote $R(i, \mu, q(\bar n))$ by $\NormalOrdA (i)$. Observe that the recursion in the definition of $f(\mu, \bar \nu, \bar n; \bar a)$ actually operates on the sequence $\{\mu(i)\}_{i=0}^{L_{\mu}(q(\bar n))}$ of $q(\bar n)$-predecessors of $\mu$. Hence, it suffices to construct $f(\mu(i), \bar \nu, \bar n; \bar a)$ as a function of $i$ and then substitute $i=L_{\mu}(q(\bar n))$. We already know that the function $\CodedFunc{L|_{\Phi_{k-3}^m}}$ belongs to $\GrzClassSig{k}$. Define 
$
    F(i, \mu, \bar \nu, \bar n, \bar a)$ as $ f(\mu(i), \bar \nu, \bar n; \bar a),
$
and notice that $F$ is constructible by a primitive recursion on $i$, using $g$, $h_{\mathsf{suc}}$, and $h_{\mathsf{lim}}$. Finally, by Lemma~\ref{lem:poly-norm-codes-loglog}, the length of $F(i, \mu, \bar \nu, \bar n, \bar a)$ is bounded by a function in $\GrzClass{k}$ in terms of $\RepLenCode{\mu(i)}$, $\RepLenCode{\bar \nu}$, $\RepLenCode{\bar n}$, and $\RepLenCode{\bar a}$. Since $\RepLenCode{\mu(i)}$ can be bounded in terms of $\RepLenCode{\mu}$, $\RepLenCode{i}$, and $\RepLenCode{\bar n}$ by a function in $\GrzClass{k}$, we may conclude, by applying the closure of $\GrzClassSig{k}$ under length-bounded primitive recursion, that $F \in \GrzClassSig{k}$, which completes the proof.

In the following, we implement this strategy in detail, bringing in the encoding of ordinals and numbers. First, recall from Lemma~\ref{lem:representable-n-geq-3} that deciding whether a string in $\Sigma^*$ encodes an ordinal in $\Phi^m_{k-3}$, and if so, whether it is zero, a successor, or a limit, can be done in polynomial time. Moreover, the same lemma ensures that
$
\CodedFunc{p|_{\Phi^m_{k-3}}} \in \GrzClassSig{3} \subseteq \GrzClassSig{k}$, as $k \geq 3$.
Now, define a function $I(b)$ on $\Sigma^*$ which outputs $b$ if $b=\IntermCode{\alpha}$ for some $\alpha \in \Phi^{m}_{k-3}$, and $\bot$ otherwise. Next, anticipating that $\CodedFunc{h_{\mathsf{suc}}}$ at $\IntermCode{\alpha+1}$ uses $\IntermCode{\alpha}$, and $\CodedFunc{h_{\mathsf{lim}}}$ at $\IntermCode{\langle\alpha_i\rangle}$ uses $\IntermCode{\langle\alpha_i\rangle}$ itself,
define $S(b)$ on $\Sigma^*$ by
\[
S(b) =
\begin{cases}
\CodedFunc{p|_{\Phi^m_{k-3}}}(b, \IntermCode{0}), & \text{if } b=\IntermCode{\alpha} \text{ for some successor ordinal } \alpha \in \Phi^m_{k-3},\\[4pt]
I(b), & \text{if } b=\IntermCode{\alpha} \text{ for some limit ordinal } \alpha \in \Phi^m_{k-3},\\[4pt]
\bot, & \text{otherwise}.
\end{cases}
\]
From the above discussion, it follows that both $I$ and $S$ belong to $\GrzClassSig{3} \subseteq \GrzClassSig{k}$.

Now we define a function responsible for the recursive calls that determine $f(\mu(i+1), \bar \nu, \bar n; \bar a)$ in terms of $f(\mu(i), \bar \nu, \bar n; \bar a)$, for any $i < L_{\mu}(q(\bar n))$. The task is simply to decide whether $\mu(i+1)$ is a successor or a limit, and then apply either $h_{\mathsf{suc}}$ or $h_{\mathsf{lim}}$.

To implement this idea, in the presence of the codings, first define the relation $B(y, u, \bar w)$ on $\Sigma^*$ stating the existence of $i, \bar n \in \mathbb{N}$ and $\mu \in \Phi_{k-3}^m$ such that $y=\IntermCode{i}$, $u=\IntermCode{\mu}$, $\bar w=\IntermCode{\bar n}$ and $i < L_\mu(q(\bar n))$. By  Lemma~\ref{lem:representable-n-geq-3}, Theorem~\ref{lem:space-to-comp-seq}, and the induction hypothesis, the relation $B$ is computable in space $\GrzClass{k}$. Now, define
\[
J(y,u,\bar v,\bar w,x,\bar z) \Def
\begin{cases}
H\bigl(S(\CodedFunc{R|_{\Phi^{m}_{k-3}}}(y,u,\CodedFunc{q}(\bar w))),\bar v,\bar w,x,\bar z\bigr), & B(y, u, \bar w),\\[4pt]
\bot, & \text{otherwise}.
\end{cases}
\]
where $H$ is $\CodedFunc{h_{\mathsf{suc}}}$ if 
$\CodedFunc{R|_{\Phi^m_{k-3}}}(y,u,\CodedFunc{q}(\bar w))$
is a code for a successor ordinal, and $H$ is $\CodedFunc{h_{\mathsf{lim}}}$ otherwise.
This function $J$ is clearly in $\GrzClassSig{k}$ by the induction hypothesis,
Lemma~\ref{lem:representable-n-geq-3}, and Theorem~\ref{lem:space-to-comp-seq}.

Intuitively, $y$ holds the code of the index $i$, $x$ holds the code of the value of the recursive call to $\CodedFunc{f}$ at each step, and $u,\bar v,\bar w,$ and $\bar z$ hold the codes of the inputs of $\CodedFunc{f}$, namely $\lceil \mu \rceil$, $\lceil \bar \nu \rceil$, $\lceil \bar n \rceil$, and $\lceil \bar a \rceil$, respectively. Thus, for $i < L_{\mu}(q(\bar n))$, the function $J$ is responsible for invoking either $\CodedFunc{h_{\mathsf{suc}}}$ or $\CodedFunc{h_{\mathsf{lim}}}$, depending on whether $\mu(i+1)$ is a successor or a limit.

Now we can define $f(\mu(i), \bar \nu, \bar n; \bar a)$ using primitive recursion on $i$. Formally, define
\[
F(y,u,\bar v,\bar w,\bar z)
\Def 
\begin{cases}
    \CodedFunc{g}(\bar v,\bar w,\bar z),
    &\text{if } y = \IntermCode{0},\\[6pt]
    J(y,u,\bar v,\bar w,F(\IntermCode{i},u,\bar v,\bar w,\bar z),\bar z),
    &\text{if } y = \IntermCode{i+1},\\[6pt]
    \bot,
    &\text{otherwise}.
\end{cases}
\]
It is not hard to see that
\[
F(\IntermCode{i}, 
\IntermCode{\NormalOrdA},
\IntermCode{\bar \NormalOrdB},\IntermCode{\bar \SafeOrdA},
\IntermCode{\bar \SafeNumbA})
=
\IntermCode{f(\NormalOrdA(i), \bar \NormalOrdB,\bar \SafeOrdA;\bar \SafeNumbA)}, 
\qquad (*)
\]
for any $i \leq L_{\mu}(q(\bar n))$.

To show that $F \in \GrzClassSig{k}$, it suffices to find an upper bound for the length of $F$. For this purpose, by Theorem~\ref{lem:space-to-comp-seq}, we have $\CodedFunc{R|_{\Phi^m_{k-3}}} \in \GrzClassSig{k}$. Therefore, it is computable in time $\GrzClass{k}$, which means that there exists a monotone function $N \in \GrzClass{k}$ such that 
$\RepLenCode{R(i, \alpha, m)} \leq N(\RepLenCode{i}, \RepLenCode{\alpha}, \RepLenCode{m})$ 
for any $\alpha \in \Phi^m_{k-3}$ and any $m, i \in \mathbb{N}$. Hence, by Lemma~\ref{lem:poly-norm-codes-loglog} for $q$ and the monotonicity of $N$, we reach  
\[
\RepLenCode{R(i, \mu, q(\bar n))} \leq N(\RepLenCode{i}, \RepLenCode{\mu}, E_q(\RepLenCode{\bar n})).
\] 
Now, since the value of $F$ is either $\bot$ or $\IntermCode{f(\NormalOrdA(i), \bar \NormalOrdB, \bar \SafeOrdA; \bar \SafeNumbA)}$, for some $i \leq L_{\mu}(q(\bar n))$, by Lemma~\ref{lem:poly-norm-codes-loglog} we have the following bound:
\[
\RepLenCode{
F(y,u,\bar v, \bar w, \bar z)
}
\leq 
1 + E_f\bigl(N(\RepLenCode{y}, \RepLen{u}, E_q(\RepLen{\bar w})), \RepLen{\bar v}, 
\RepLen{\bar w}, \RepLen{\bar z}\bigr).
\] 
Since $E_f, E_q, N \in \GrzClass{k}$ and $\GrzClassSig{k}$ is closed under length-bounded primitive recursion by Theorem~\ref{the:Esigma-closure}, we conclude that $F \in \GrzClassSig{k}$.

Finally, substituting $i = L_{\mu}(q(\bar n))$ in $(*)$, and noting that $F$ returns $\bot$ for any inputs that are not codes of ordinals in $\Phi^m_{k-3}$ or numbers, we obtain
\[
\CodedFunc{f}(u,\bar v, \bar w, \bar z) = F(\CodedFunc{L|_{\Phi^m_{k-3}}}(u, \CodedFunc{q}(\bar w)), u, \bar v, \bar w, \bar z).
\]
By Theorem~\ref{lem:space-to-comp-seq}, $\CodedFunc{L|_{\Phi^m_{k-3}}} \in \GrzClassSig{k}$. Therefore, we conclude that $\CodedFunc{f} \in \GrzClassSig{k}$.

Finally, since the cases where $f$ is obtained by constant substitution or structural rules are straightforward, we can conclude that for any $f \in \ClassName{\Phi_{k-3}^{m}}$, we have $\CodedFunc{f} \in \GrzClassSig{k}$. This completes the first part of $(i)$. For the second part, since functions in $\PredFuncClass{\Phi^m_{k-3}} \subseteq \ClassName{\Phi_{k-3}^{m}}$ have no ordinal inputs, they are numeral functions whose lifts are in $\GrzClassSig{k}$. Therefore, by Remark~\ref{rem:equiv-E-sig-E}, they must be in $\GrzClass{k}$.

The proof of part $(ii)$ is similar to that of $(i)$. It suffices to use linear-space algorithms and the linear bounds from Lemma~\ref{lem:representable-n-geq-3III}, Theorem~\ref{lem:space-to-comp-seq-linear}, and Corollary~\ref{cor:len-code-linear-bound}.
\end{proof}

\begin{theorem}
Let $\mathsf{A} \subseteq \Phi_{\omega}$ be a bounded downset of ordinals. Then $\PredFuncClass{\mathsf{A}} \subseteq  \GrzClass{l(\mathsf{A})+2}$.
\end{theorem}
\begin{proof}
First, we prove that $\PredFuncClass{\Phi_{k}} \subseteq \mathcal{E}_{k+2}$ for any $k \geq 1$ and that $\PredFuncClass{\Psi_{\omega}} \allowbreak\subseteq \mathcal{E}_{2}$. For the first claim, note that, by Lemma~\ref{lem:coded-func-in-grz-space}, for any $m \geq 1$, we have $\PredFuncClass{\Phi^m_{k-1}} \subseteq \mathcal{E}_{k+2}$.  
Next, by Lemma~\ref{MonotonicityOfPhi}, we have $\Phi_k = \bigcup_{m \geq 1} \Phi^m_{k-1}$ and $\Phi_{k-1}^m \subseteq \Phi^{m+1}_{k-1}$ for all $m \geq 1$. Therefore, by Lemma~\ref{lem:PredInLimit}, it follows that  
$\PredFuncClass{\Phi_k} = \bigcup_{m \geq 1} \PredFuncClass{\Phi^m_{k-1}}$.  
Consequently, $\PredFuncClass{\Phi_{k}} \subseteq \mathcal{E}_{k+2}$.
For the second claim, for any 
$k \geq 1$ and $m \geq 1$, we have 
$\PredFuncClass{\Psi_k^m} \subseteq \mathcal{E}_{2}$, by Lemma \ref{lem:coded-func-in-grz-space}.
Now, by Lemma \ref{MonotonicityOfPsi},  
$\Psi_{k}^m \subseteq \Psi_k^{m+1}$, for any $k, m \geq 1$. Moreover, for any $k \geq 1$, we have $\Psi_k = \bigcup_{m \geq 1} \Psi_k^m$.
Thus, by Lemma~\ref{lem:PredInLimit}, we reach $\PredFuncClass{\Psi_k} = \bigcup_{m \geq 1}\PredFuncClass{\Psi_k^m}$ which implies $\PredFuncClass{\Psi_k} \subseteq \GrzClass{2}$.
Similarly,
we have that
$\Psi_{k} \subseteq \Psi_{k+1}$
for any $k \geq 1$
and $\Psi_\omega = \bigcup_{k \geq 1} \Psi_k$, by Lemma \ref{MonotonicityOfPsi}. Thus,
$\PredFuncClass{\Psi_\omega} = \bigcup_{k \geq 1}\PredFuncClass{\Psi_k}$. Therefore,
$\PredFuncClass{\Psi_\omega}\subseteq\GrzClass{2}$. 

Now, if $l(\mathsf{A}) \geq 1$, then by definition we have $\mathsf{A} \subseteq \Phi_{l(\mathsf{A})}$. Hence, by Lemma~\ref{MonotonicityOfPredR}, it follows that 
$\PredFuncClass{\mathsf{A}} \subseteq \PredFuncClass{\Phi_{l(\mathsf{A})}}$. 
Therefore, setting $k = l(\mathsf{A})$ in the above claim yields
$\PredFuncClass{\mathsf{A}} \subseteq \mathcal{E}_{l(\mathsf{A})+2}$.
If $l(\mathsf{A})=0$, by definition $\mathsf{A} \subseteq \Psi_{\omega}$. By the same line of argument, we reach  $\PredFuncClass{\mathsf{A}} \subseteq \mathcal{E}_{2}$.
\end{proof}

\section{Conclusion and Future Work}

In this paper, we generalized Bellantoni-Cook's predicative recursive functions by extending their predicative recursion on natural numbers to predicative ordinal recursion on any ordinal in a given downset $\mathsf{A}$ of constructive ordinals. We denote this class of functions by $\PredFuncClass{\mathsf{A}}$. We then introduced the downset $\Phi_{\omega}$ (resp.\ $\Phi_k$) of constructive ordinals as the constructive counterpart of the set-theoretic ordinals below $\bm{\phi}_{\bm{\omega}}(\bm{0})$ (resp.\ $\bm{\phi}_{k}(\bm{0})$), where $\bm{\phi}_{k}$ is the $k$th Veblen function. Finally, we established a complete classification of $\PredFuncClass{\mathsf{A}}$ for any downset $\mathsf{A} \subseteq \Phi_{\omega}$ that contains at least one infinite constructive ordinal. More precisely, we showed that for any such downset, if $\mathsf{A}$ is not contained in any $\Phi_k$, then $\PredFuncClass{\mathsf{A}}$ coincides with the class of all primitive recursive functions; otherwise,
$
\PredFuncClass{\mathsf{A}} \;=\; \GrzClass{l(\mathsf{A})+2},
$
where $l(\mathsf{A})$ is a natural number measuring the level of $\mathsf{A}$ in $\Phi_{\omega}$, and $\GrzClass{k}$ denotes the $k$th level of the Grzegorczyk hierarchy. In particular, we obtained a machine-independent and structural characterization of $\GrzClass{k}$ as 
$\PredFuncClass{\Phi_{k-2}}$ for any $k \geq 3$. For $\GrzClass{2}$, we generalized Bellantoni-Cook's characterization to 
$\PredFuncClass{\mathsf{D}_{\omega^\omega}}=\PredFuncClass{\mathsf{D}_{\omega^m}}$ for any $m \geq 2$, where $\mathsf{D}_{\alpha}$ is the set of all constructive ordinals strictly below $\alpha$. This shows that even if we extend predicative ordinal recursion from $\omega$ to ordinals below $\omega^m$ for $m \geq 2$, or to any constructive ordinal below $\omega^{\omega}$, the resulting functions still belong to $\mathcal{E}_2$.

For future work, we outline three directions. First, we aim to extend the present classification beyond $\Phi_{\omega}$ to encompass broader classes of constructive ordinals. 
Alternatively, 
this may be viewed as extending
the characterization of the $\mathcal{E}_k$'s to $\mathcal{E}_{\alpha}$ for infinite $\alpha$. Second, by using a suitable form of predicative recursion on arbitrary well-founded partial orders \cite{curzi2022cyclic}, we conjecture that the predicatively recursive functions on a well-founded partial order correspond to the predicatively recursive functions on its associated ordinal (i.e., its height). Proving this conjecture 
would generalize the main classification of this paper to all well-founded structures whose height is below $\bm{\phi}_{\bm{\omega}}(\bm{0})$. Third, we plan to develop the arithmetical theories corresponding to these predicative ordinal recursive functions. These absolutely predicative theories of arithmetic 
would be 
defined in terms of different comprehension schemes. A similar approach, using separation in the types of involved terms, has already been proposed and studied \cite{ostrin2005elementary}. After developing such theories, we intend to apply the machinery of ordinal analysis to them, with the aim of characterizing their provably total recursive functions in terms of predicative ordinal recursive functions.
For instance, a predicative arithmetical theory with $\Delta^0_{\infty}$-comprehension can be regarded as the absolutely predicative version of $\mathsf{PA}$. Therefore, we expect its proof-theoretic ordinal to be $\bm{\epsilon_0}$ and, since $\Phi_1$ is the set of constructive representations of ordinals below $\bm{\epsilon_0}$, the provably recursive functions of the theory must correspond to $\PredFuncClass{\Phi_1} = \mathcal{E}_3$, i.e., the Kalmár elementary functions. This is consistent with similar results obtained for other proposals of absolutely predicative versions of Peano arithmetic; see \cite{ostrin2005elementary}.

\subsection*{Acknowledgements}

A. Akbar Tabatabai and R. Ramanayake acknowledge the financial support of the Dutch Research Council (NWO) project OCENW.M.22.258. R. Ramanayake also acknowledges the CogniGron research center, and Ubbo Emmius Funds (Univ. of Groningen).





 \bibliographystyle{elsarticle-num} 
 \bibliography{cas-refs}

\begin{thebibliography}{10}
\expandafter\ifx\csname url\endcsname\relax
  \def\url#1{\texttt{#1}}\fi
\expandafter\ifx\csname urlprefix\endcsname\relax\def\urlprefix{URL }\fi
\expandafter\ifx\csname href\endcsname\relax
  \def\href#1#2{#2} \def\path#1{#1}\fi

\bibitem{leivant1991foundational}
D.~Leivant, A foundational delineation of computational feasibility, in: Proceedings 1991 Sixth Annual IEEE Symposium on Logic in Computer Science, IEEE Computer Society, 1991, pp. 2--3.

\bibitem{bellantoni92thesis}
S.~J. Bellantoni, Predicative recursion and computational complexity, Ph.D. thesis, University of Toronto, Canada (1992).

\bibitem{bellantonicook1992}
S.~Bellantoni, S.~Cook, A new recursion-theoretic characterization of the polytime functions, Computational Complexity 2~(2) (1992) 97--110.
\newblock \href {https://doi.org/10.1007/BF01201998} {\path{doi:10.1007/BF01201998}}.

\bibitem{Weyl}
H.~Weyl, The continuum, Thomas Jefferson University Press, Kirksville, MO, 1987.

\bibitem{FromFregeToGodel}
J.~Van~Heijenoort, From Frege to G{\"o}del: A source book in mathematical logic, 1879--1931, Vol.~9, Harvard University Press, 1967.

\bibitem{buss1985bounded}
S.~R. Buss, Bounded arithmetic, Princeton University, 1985.

\bibitem{Jan}
J.~Kraj\'{i}\v{c}ek, Bounded arithmetic, propositional logic, and complexity theory, Vol.~60 of Encyclopedia of Mathematics and its Applications, Cambridge University Press, Cambridge, 1995.

\bibitem{Cobham}
A.~Cobham, The intrinsic computational difficulty of functions, in: Logic, Methodology and Philosophy of Science ({P}roc. 1964 {I}nternat. {C}ongr.), North-Holland, Amsterdam, 1965, pp. 24--30.

\bibitem{CloteSurvey}
P.~Clote, Computation models and function algebras, in: Handbook of computability theory, Vol. 140 of Studies in Logic and the Foundations of Mathematics, North-Holland, Amsterdam, 1999, pp. 589--681.

\bibitem{Nelson}
E.~Nelson, Predicative arithmetic, Vol.~32 of Mathematical Notes, Princeton University Press, Princeton, NJ, 1986.

\bibitem{LeivantI}
D.~Leivant, Ramified recurrence and computational complexity. {I}. {W}ord recurrence and poly-time, in: Feasible mathematics, {II} ({I}thaca, {NY}, 1992), Vol.~13 of Progress in Computer Science and Applied Logic, Birkh\"auser Boston, Boston, MA, 1995, pp. 320--343.

\bibitem{LeivantII}
D.~Leivant, J.-Y. Marion, Ramified recurrence and computational complexity. {II}. {S}ubstitution and poly-space, in: Computer science logic ({K}azimierz, 1994), Vol. 933 of Lecture Notes in Computer Science, Springer, Berlin, 1995, pp. 486--500.

\bibitem{LeivantIII}
D.~Leivant, Ramified recurrence and computational complexity. {III}. {H}igher type recurrence and elementary complexity, Vol.~96, 1999, pp. 209--229, festschrift on the occasion of Professor Rohit Parikh's 60th birthday.

\bibitem{AraiExp}
T.~Arai, N.~Eguchi, A new function algebra of {EXPTIME} functions by safe nested recursion, ACM Transactions on Computational Logic 10~(4) (2009) Art. 24, 19.

\bibitem{wirz}
M.~Wirz, Characterizing the Grzegorczyk hierarchy by safe recursion, Universit{\"a}t Bern. Institut f{\"u}r Informatik und Angewandte Mathematik, 1999.

\bibitem{Arai}
T.~Arai, Predicatively computable functions on sets, Archive for Mathematical Logic 54~(3-4) (2015) 471--485.

\bibitem{Other1}
J.-Y. Marion, Predicative analysis of feasibility and diagonalization, in: Typed lambda calculi and applications, Vol. 4583 of Lecture Notes in Comput. Sci., Springer, Berlin, 2007, pp. 290--304.

\bibitem{Other2}
J.-Y. Marion, On tiered small jump operators, Logical Methods in Computer Science 5~(1) (2009) 1:7, 19.

\bibitem{Other3}
O.~Bournez, P.~de~Naurois, J.-Y. Marion, Safe recursion and calculus over an arbitrary structure, Implicit Computational Complexity 2 (2002).

\bibitem{Other4}
S.~J. Bellantoni, K.-H. Niggl, H.~Schwichtenberg, Higher type recursion, ramification and polynomial time, in: Proceedings of the {W}orkshop on {P}roof {T}heory and {C}omplexity, {PTAC}'98 ({A}arhus), Vol. 104, 2000, pp. 17--30.

\bibitem{LeivantPoly}
D.~Leivant, Finitely stratified polymorphism, Vol.~93, 1991, pp. 93--113, selections from the 1989 IEEE Symposium on Logic in Computer Science.

\bibitem{CloteExp}
P.~Clote, A safe recursion scheme for exponential time, in: Logical foundations of computer science ({Y}aroslavl, 1997), Vol. 1234 of Lecture Notes in Computer Science, Springer, Berlin, 1997, pp. 44--52.

\bibitem{HofmannHabil}
M.~Hofmann, Type systems for polynomial-time computation, Habilitation (1999).

\bibitem{HofmannBCK}
M.~Hofmann, Safe recursion with higher types and {BCK}-algebra, in: Proceedings of the {W}orkshop on {P}roof {T}heory and {C}omplexity, {PTAC}'98 ({A}arhus), Vol. 104, 2000, pp. 113--166.

\bibitem{HofmannLinear}
M.~Hofmann, A mixed modal/linear lambda calculus with applications to {B}ellantoni-{C}ook safe recursion, in: Computer science logic ({A}arhus, 1997), Vol. 1414 of Lecture Notes in Computer Science, Springer, Berlin, 1998, pp. 275--294.

\bibitem{Cockett}
M.~Burrell, R.~Cockett, B.~Redmond, Safe recursion revisited {I}: categorical semantics for lower complexity, Theoretical Computer Science 515 (2014) 19--45.

\bibitem{curzi2022cyclic}
G.~Curzi, A.~Das, Cyclic implicit complexity, in: Proceedings of the 37th Annual ACM/IEEE Symposium on Logic in Computer Science, 2022, pp. 1--13.

\bibitem{girardF2}
J.-Y. Girard, Interpr{\'e}tation fonctionnelle et {\'e}limination des coupures de l'arithm{\'e}tique d'ordre sup{\'e}rieur, Ph.D. thesis (1972).

\bibitem{reynoldsF2}
J.~C. Reynolds, Towards a theory of type structure, in: Programming Symposium: Proceedings, Colloque sur la Programmation Paris, April 9--11, 1974, Springer, 1974, pp. 408--425.

\bibitem{grzegorczyk1953some}
A.~Grzegorczyk, Some classes of recursive functions, Instytut Matematyczny Polskiej Akademi Nauk, 1953.

\bibitem{rose1984}
H.~E. Rose, Subrecursion: Functions and Hierarchies, Oxford University Press, New York, 1984.

\bibitem{odifreddi1999crtv2}
P.~Odifreddi, Classical Recursion Theory II, Elsevier, 1999.
\newblock \href {https://doi.org/10.1016/s0049-237x(99)x8037-2} {\path{doi:10.1016/s0049-237x(99)x8037-2}}.

\bibitem{ritchie1963}
R.~W. Ritchie, Classes of predictably computable functions, Transactions of the American Mathematical Society 106~(1) (1963) 139.
\newblock \href {https://doi.org/10.2307/1993719} {\path{doi:10.2307/1993719}}.

\bibitem{fairtlough1998ptchapter}
M.~Fairtlough, S.~S. Wainer, {H}ierarchies of {P}rovably {R}ecursive {F}unctions, in: S.~R. Buss (Ed.), Handbook of Proof Theory, Vol. 137 of Studies in Logic and the Foundations of Mathematics, Elsevier, 1998, pp. 149--207.
\newblock \href {https://doi.org/10.1016/S0049-237X(98)80018-9} {\path{doi:10.1016/S0049-237X(98)80018-9}}.

\bibitem{veblen1908}
O.~Veblen, Continuous increasing functions of finite and transfinite ordinals, Transactions of the American Mathematical Society 9~(3) (1908) 280–292.
\newblock \href {https://doi.org/10.1090/s0002-9947-1908-1500814-9} {\path{doi:10.1090/s0002-9947-1908-1500814-9}}.

\bibitem{Schutte1977}
K.~Sch\"{u}tte, Proof Theory, Springer Verlag, New York, 1977.

\bibitem{ostrin2005elementary}
G.~E. Ostrin, S.~S. Wainer, Elementary arithmetic, Annals of Pure and Applied Logic 133~(1-3) (2005) 275--292.

\end{thebibliography}





\end{document}